\documentclass[10pt,a4paper]{article}
\usepackage[utf8]{inputenc}
\usepackage[english]{babel}
\usepackage{tocbibind}
\usepackage{amsthm}
\usepackage{tikz-cd}
\usepackage{amsfonts}
\usepackage{amssymb}
\usepackage[T1]{fontenc}
\usepackage{stix}
\usepackage{ragged2e}
\usepackage{textcomp}
\usepackage{amssymb,amscd,amsthm,amsfonts,amsmath,verbatim}

\usepackage[toc,page]{appendix}
\usepackage{eucal}
\usepackage[all]{xy}
\usepackage{amsthm}
\usepackage{tikz}
\usepackage{stackengine}
\usepackage{scalerel}
\usetikzlibrary{decorations.markings}
\usetikzlibrary{matrix,arrows}
\usepackage{bbm}
\usepackage{textcomp}
\usepackage{setspace} %para separar mas las lineas
\usetikzlibrary{shapes,arrows}
\usepackage{pgf}
\usepackage{wrapfig,lipsum,booktabs}
\usepackage{lipsum}
\usepackage{calrsfs}
\DeclareMathAlphabet{\pazocal}{OMS}{zplm}{m}{n}

\newcommand{\Ab}{\pazocal{A}}
\newcommand{\Ba}{\mathcal{B}}
\newcommand{\Bb}{\pazocal{B}}
\newcommand{\Ca}{\mathcal{C}}
\newcommand{\Cb}{\pazocal{C}}
\newcommand{\Da}{\mathcal{D}}
\newcommand{\Db}{\pazocal{D}}
\newcommand{\Ea}{\mathcal{E}}
\newcommand{\Eb}{\pazocal{E}}
\newcommand{\Fa}{\mathcal{F}}
\newcommand{\Fb}{\pazocal{F}}

\newcommand{\Gb}{\pazocal{G}}

\newcommand{\Ib}{\pazocal{I}}
\newcommand{\Ia}{\mathcal{I}}

\newcommand{\Ja}{\mathcal{J}}
\newcommand{\Ka}{\mathcal{K}}

\newcommand{\La}{\mathcal{L}}
\newcommand{\Lb}{\pazocal{L}}
\newcommand{\Ma}{\mathcal{M}}
\newcommand{\Mb}{\pazocal{M}}
\newcommand{\Na}{\mathcal{N}}
\newcommand{\Nb}{\pazocal{N}}

\newcommand{\Ob}{\pazocal{O}}

\newcommand{\Pb}{\pazocal{P}}

\newcommand{\Sb}{\pazocal{S}}

\newcommand{\Tb}{\pazocal{T}}

\newcommand{\Ub}{\pazocal{U}}

\newcommand{\Vb}{\pazocal{V}}
 \newtheorem{lem}{Lemma}[subsection]
 \newtheorem{prop}[lem]{Proposition}
  \newtheorem{coro}[lem]{Corollary}
 \newtheorem{theo}[lem]{Theorem}
  \newtheorem{defi}[lem]{Definition}
  \newtheorem{rem}[lem]{Remark}
 
 \newtheorem{exa}[lem]{Example}
 \newtheorem{comm}{Commentary}
  \newtheorem{roots}[lem]{}

  \newtheorem{line_bundles}[lem]{}
  \newtheorem{cong_sub_groups}[lem]{}  
  \newtheorem{Analytic_algebra}[lem]{}  
  \newtheorem{Notation}[lem]{}
    \newtheorem{Final_remark}[lem]{}  
 \newcommand{\Q }{{\mathbb Q }}
 \newcommand{\G }{{\mathbb G }}
 \newcommand{\T }{{\mathbb T }}
 \newcommand{\B }{{\mathbb B }}
 \newcommand{\Z }{{\mathbb Z }}
 \newcommand{\N }{{\mathbb N }}
 \newcommand{\stirlingii}{\genfrac{\{}{\}}{0pt}{}}
\DeclareRobustCommand{\abinom}{\genfrac{\langle}{\rangle}{0pt}{}}
 \usepackage{graphicx}
\newcommand{\simeqd}{\mathrel{\rotatebox[origin=c]{-90}{$\simeq$}}}
\usepackage{vmargin}
\setpapersize{A4}
\setmargins{2.5cm}       % margen izquierdo
{1.5cm}                        % margen superior
{17cm}                      % anchura del texto
{23.42cm}                    % altura del texto
{10pt}                           % altura de los  encabezados
{1cm}                           % espacio entre el texto y los encabezados
{0pt}                             % altura del pie de página
{2cm}
\begin{document}
% Needed for binamial coefficients--------------------------------------------
\def \bangle{ \atopwithdelims \langle \rangle}
%--------------------------------------------------------------------------------------
%--------------------------------------------------------------------------------------
\title{\textbf{G-equivariance of formal models
of flag varieties}}
\author{ Andr\'{e}s Sarrazola Alzate}
\date{}
\maketitle
\begin{abstract}
Let $\G$ be a split connected reductive group scheme over the ring of integers $\mathfrak{o}$ of a finite extension $L|\Q_p$ and $\lambda\in X(\T)$ an algebraic character of a split maximal torus $\T\subseteq\G$. Let us also consider $X^{\text{rig}}$ the rigid analytic flag variety of $\G$ and $G=\G(L)$. In the first part of this paper, we introduce a family of $\lambda$-twisted differential operators on a formal model $\mathfrak{Y}$ of $X^{\text{rig}}$. We compute their global sections and we prove coherence together with several cohomological properties. In the second part, we define the category of coadmissible $G$-equivariant arithmetic $\Da(\lambda)$-modules over the family of formal models of the rigid flag variety $X^{\text{rig}}$. We show that if $\lambda$ is such that $\lambda + \rho$ is dominant and regular ($\rho$ being the Weyl character), then the preceding category is anti-equivalent to the category of admissible locally analytic $G$-representations, with central character $\lambda$. In particular, we generalize the results in \cite{HPSS} for algebraic characters.
\end{abstract}
\justify

\makeatletter
\@starttoc{toc}
\makeatother

\section{Introduction}
\justify
Let $L|\Q_p$ be a finite extension and $\mathfrak{o}$ its ring of integers. Throughout this work $\G$ will always denote a split connected reductive group scheme over  $\mathfrak{o}$. We will fix a Borel subgroup $\B\subset \G$ which contain a split maximal torus $\T\subset\B$ of $\G$. We will also denote by $X:=\G/\B$ the smooth flag $\mathfrak{o}$-scheme associated to $\G$ and by $\mathfrak{X}$ the smooth formal scheme. In \cite{HPSS} the authors have introduced certain sheaves of differential operators (with congruence level $k\in\N$) $\Da^{\dag}_{\mathfrak{Y},k}$ on a family of formal models $\mathfrak{Y}$ of $X^{\text{rig}}$ the rigid analytic flag variety. They study their cohomological properties and show that $\mathfrak{Y}$ is $\Da^{\dag}_{\mathfrak{Y},k}$-affine, which is a fundamental result. Moreover, it is proved in \cite[Theorem 5.3.12]{HPSS} that the category of admissible locally analytic representations of the $L$-analytic group $\G(L)$ can be described in terms of $\G(L)$-equivariant families $(\Ma_{\mathfrak{Y},k})$ of modules over $\Da^{\dag}_{\mathfrak{Y},k}$ on the projective system of all formal models $\mathfrak{Y}$ of $X^{\text{rig}}$.
\justify
Our motivation is to study the preceding localization theorem in the twisted situation. In this work we will treat the algebraic case. This is, we will only consider characters of the Lie algebra $\mathfrak{t}=\text{Lie}(\T)$ arriving from characters $\lambda\in X(\T)=\text{Hom}(\T,\G_m)$ via differentiation. In this situation $\lambda$ induces an invertible sheaf $\La(\lambda)$ on $\mathfrak{X}$ and we define $\Da^{\dag}_{\mathfrak{X},k}(\lambda)$ as the sheaf of differential operators (with conguence level $k$) acting on $\La(\lambda)$. We will follow the philosophy described in \cite{HPSS} introducing sheaves of differential operators on each admisisble blow-up of $\mathfrak{X}$. Let $\text{pr}:\mathfrak{Y}\rightarrow\mathfrak{X}$ be an admissible blow-up, then for $k>>0$\footnote{This technical condition is clarified in proposition \ref{mult_blow_up}. It is also explained in (\ref{congruence level}) below.}
\begin{eqnarray*}
\Da^{\dag}_{\mathfrak{Y},k}(\lambda):=\text{pr}^*\Da^{\dag}_{\mathfrak{X},k}(\lambda) = \Ob_{\mathfrak{Y}}\otimes_{\text{pr}^{-1}\Ob_{\mathfrak{X}}}\text{pr}^{-1}\Da^{\dag}_{\mathfrak{X},k}(\lambda)
\end{eqnarray*}
is a sheaf of rings on $\mathfrak{Y}$. Let us denote by $\rho$ the so-called Weyl character and let us assume that $\lambda+\rho\in\mathfrak{t}^*_L = \text{Hom}_{L}(\mathfrak{t}\otimes_{\mathfrak{o}}L,L)$ is a dominant and regular character of $\mathfrak{t}_L := \mathfrak{t}\otimes_{\mathfrak{o}}L$. In this situation, we will show that the direct image functor $\text{pr}_*$ induces an equivalence of categories between the category of coherent $\Da^{\dag}_{\mathfrak{Y},k}(\lambda)$-modules and the category of coherent $\Da^{\dag}_{\mathfrak{X},k}(\lambda)$-modules. In addition, we have $\text{pr}_*\Da^{\dag}_{\mathfrak{Y},k}(\lambda)=\Da^{\dag}_{\mathfrak{X},k}(\lambda)$, which implies that 
\begin{eqnarray*}
H^{0}(\mathfrak{Y},\Da^{\dag}_{\mathfrak{Y},k}(\lambda)) = H^{0}(\mathfrak{X},\Da^{\dag}_{\mathfrak{X},k}(\lambda)).
\end{eqnarray*}
It is shown that $H^0(\mathfrak{X},\Da^{\dag}_{\mathfrak{X},k}(\lambda))$ can be identified with the central redaction $\Db^{\text{an}}(\G(k)^{\circ})_{\lambda}$\footnote{Via the classical isomorphism of Harish-Chandra $\lambda$ induces a character $\chi_{\lambda}: Z(\text{Lie}(\G)\otimes_{\mathfrak{o}}L)\rightarrow L$ which allows to consider the central redaction.} of Emerton's analytic distribution algebra  $\Db^{\text{an}}(\G(k)^{\circ})$ of the wide open rigid-analytic k-th congruence group $\G(k)^{\circ}$. Our first result is
\justify
\textbf{Theorem 1.}  Let $\text{pr}:\mathfrak{Y}\rightarrow\mathfrak{X}$ be an admissible blow-up. Suppose that $\lambda\in\text{Hom}(\T,\G_m)$ is an algebraic character such that $\lambda+\rho\in\mathfrak{t}_L^*$ is a dominant and regular character of $\mathfrak{t}_L$. Then $H^{0}(\mathfrak{Y},\bullet)$ induces an equivalence between the categories of coherent $\Da^{\dag}_{\mathfrak{Y},k}(\lambda)$-modules and finitely presented $\Db^{\text{an}}(\G(k)^{\circ})_{\lambda}$-modules. 
\justify
As in the classical case, the inverse functor is determined by the localization functor 
\begin{eqnarray*}
\La oc^{\dag}_{\mathfrak{Y},k}(\lambda) (\bullet) := \Da^{\dag}_{\mathfrak{Y},k}(\lambda)\otimes_{\Db^{\text{an}}(\G(k)^{\circ})_{\lambda}}(\bullet).
\end{eqnarray*}
Let us now describe the most important tools in our localization theorem. On the algebraic side, we will first assume that $G_0= \G(\mathfrak{o})$ and that $D(G_0,L)$ is the algebra of locally analytic distributions of the compact analytic group $G_0$ (in the sense of \cite{ST}). The key point will be to construct a structure of weak Fréchet-Stein algebra on $D(G_0,L)$ (in the sense of \cite[Definition 1.2.6]{Emerton2}) that will allow us to localize the coadmissible $D(G_0,L)$-modules (relative to this weak Fréchet-Stein structure). In fact, if $\Cb^{\text{cts}}(G_0,L)_{\G(k)^{\circ}-{\text{an}}}$ is the vector space of locally analytic vectors of the space of continuous $L$-valued functions and $D(\G(k)^{\circ},G_0):=(\Cb^{\text{cts}}(G_0,L)_{\G(k)^{\circ}-\text{an}})'_b$ is its strong dual, then we have an isomorphism
\begin{eqnarray*}
D(G_0,L) \xrightarrow{\simeq}\varprojlim_{k\in\N} D(\G(k)^{\circ},G_0)
\end{eqnarray*}
\justify
which defines the structure of weak Fréchet-Stein algebra and such that
\begin{eqnarray} \label{rel_1.2}
D(\G(k)^{\circ},G_0) = \displaystyle\bigoplus_{g\in G_0/G_k} \Db^{\text{an}}(\G(k)^{\circ})\delta_g.
\end{eqnarray}
\justify
Here $G_k:=\G(k)(\mathfrak{o})$ is a normal subgroup of $G_0$, the direct sum runs through a set of representatives of the cosets of $G_k$ in $G_0$ and $\delta_g$ is the Dirac distribution supported in $g$. We will denote by $\Cb_{G_0,\lambda}$ the category of coadmssible $D(G_0,L)$-modules with central character $\lambda$ (coadmissible $D(G_0,L)_{\lambda}$-modules, where $D(G_0,L)_{\lambda}$ denotes the central reduction). 
\justify
Now, on the geometric side, we will consider $\text{pr}:\mathfrak{Y}\rightarrow\mathfrak{X}$ a $G_0$-equivariant admssible blow-up such that the invertible sheaf $\La (\lambda)$ is equipped with a $G_0$-action that allows us to define a left $G_0$-action $T_g:\Da^{\dag}_{\mathfrak{Y},k}(\lambda)\rightarrow (\rho_g)_*\Da^{\dag}_{\mathfrak{Y},k}(\lambda)$ on $\Da^{\dag}_{\mathfrak{Y},k}(\lambda)$\footnote{Here $g\in G_0$ and $\rho_g :\mathfrak{Y}\rightarrow\mathfrak{Y}$ is the morphism induced by $G_0$-equivariance.}, in the sense that for every $g,h\in G_0$ we have the cocycle condition $T_{hg} = (\rho_g)_* T_h\;\circ\; T_g$. So, we will say that a coherent $\Da^{\dag}_{\mathfrak{Y},k}(\lambda)$-module $\Ma$ is \textit{strongly $G_0$-equivariant} if there is a family $(\varphi_g)_{g\in G_0}$ of isomorphisms $\varphi_g : \Ma\rightarrow (\rho_g)_*\Ma$ of sheaves of $L$-vector spaces, which satisfy the following properties (conditions ($\dag$)) : 
\begin{itemize}
\item[$\bullet$] For every $g,h\in G_0$ we have $(\rho_{g})_*\varphi_h\;\circ\;\varphi_g = \varphi_{hg}$.
\item[$\bullet$] If $\Ub\subseteq \mathfrak{Y}$ is an open subset, $P\in\Da^{\dag}_{\mathfrak{Y},k}(\lambda)(\Ub)$ and $m\in\Ma(\Ub)$ then $\varphi_{g}(P\bullet m)=T_{g}(P)\bullet \varphi_{g}(m)$.
\item[$\bullet$]\footnote{We identify here $H^0(\mathfrak{Y},\Da^{\dag}_{\mathfrak{Y},k}(\lambda))$ with $\Db^{\text{an}}(\G(k)^{\circ})_{\lambda}$ and we use lemma \ref{trivial_on_k+1} to give a sense to this condition.} For any $g\in G_{k+1}$ the application $\varphi_g:\Ma\rightarrow (\rho_g)_*\Ma$ is equal to the multiplication by $\delta_g\in \Db^{\text{an}}(\G(k))_{\lambda}$. 
\end{itemize}
\justify
A morphism between two strongly $G_0$-equivariant $\Da^{\dag}_{\mathfrak{Y},k}(\lambda)$-modules $(\Ma, (\varphi_g^{\Ma})_{g\in G_0})$ and $(\Na, (\varphi_g^{\Na})_{g\in G_0})$ is a morphism $\psi:\Ma\rightarrow \Na$ which is $\Da^{\dag}_{\mathfrak{Y},k}(\lambda)$-linear and such that, for every $g\in G_0$, we have $\varphi_g^{\Na}\;\circ\; \psi = (\rho_g)_*\psi\;\circ\; \varphi^{\Ma}_{g}$.  We denote by $\text{Coh}(\Da^{\dag}_{\mathfrak{Y},k}(\lambda),G_0)$ the category of strongly $G_0$-equivariant $\Da^{\dag}_{\mathfrak{Y},k}(\lambda)$-modules. We have the following result \footnote{We use the relationship (\ref{rel_1.2}) to give a sense to the assertion of the theorem.} 
\justify
\textbf{Theorem 2.} Let $\lambda\in\text{Hom}(\T,\G_m)$ be an algebraic character such that $\lambda+\rho\in\mathfrak{t}_L^*$ is a dominant and regular character of $\mathfrak{t}_L$. The functors $\La oc^{\dag}_{\mathfrak{Y},k}(\lambda)$ and $H^{0}(\mathfrak{Y},\bullet)$ induce equivalences between the categories of finitely presented $D(\G(k)^{\circ},G_0)$-modules (with central character $\lambda$) and $\text{Coh}(\Da^{\dag}_{\mathfrak{Y},k}(\lambda),G_0)$.
\justify
Still on the geometric side, let us consider the set $\underline{\Fb}_{\mathfrak{X}}$ of couples $(\mathfrak{Y},k)$ such that $\mathfrak{Y}$ is an admissible blow-up of $\mathfrak{X}$ and $k\ge k_{\mathfrak{Y}}$, where
\begin{eqnarray}\label{congruence level}
k_{\mathfrak{Y}} := \underset{\Ia}{\text{min}}\;\text{min}\{k\in\N\;|\; \varpi^k\in\Ia\} 
\end{eqnarray}
\justify
and $\Ia$ is an ideal subsheaf of $\Ob_{\mathfrak{X}}$, such that $\mathfrak{Y}\simeq V(\Ia)$. This set carries a partial order. As is shown in \cite{HPSS} the group $G_0$ acts on $\underline{\Fb}_{\mathfrak{X}}$ and this action respects the congruence level. This means that for every couple $(\mathfrak{Y},k)\in\underline{\Fb}_{\mathfrak{X}}$ there is a couple $(\mathfrak{Y}.g,k_{\mathfrak{Y}.g})\in\underline{\Fb}_{\mathfrak{X}}$ with an isomorphism $\rho_g:\mathfrak{Y}\rightarrow \mathfrak{Y}.g$ and such that $k_{\mathfrak{Y}}=k_{\mathfrak{Y}.g}$. So, we will say that a family $\Ma:=(\Ma_{\mathfrak{Y},k})_{(\mathfrak{Y},k)\in\underline{\Fb}_{\mathfrak{X}}}$ of coherent $\Da^{\dag}_{\mathfrak{Y},k}(\lambda)$-modules is a coadmissible $G_0$-equivariant $\Da(\lambda)$-module on $\underline{\Fb}_{\mathfrak{X}}$ if for any $g\in G_0$, with morphism $\rho_g:\mathfrak{Y}\rightarrow\mathfrak{Y}.g$, there is an isomorphism 
\begin{eqnarray*}
\varphi:\Ma_{\mathfrak{Y}.g,k}\rightarrow(\rho_g)_*\Ma_{\mathfrak{Y},k}
\end{eqnarray*}
that satisfies the conditions $(\dag)$ and such that, if $(\mathfrak{Y}',k')\succeq (\mathfrak{Y},k)$ with $\pi:\mathfrak{Y}'\rightarrow\mathfrak{Y}$, then there is a transition morphism $\pi_*\Ma_{\mathfrak{Y}',k'}\rightarrow\Ma_{\mathfrak{Y},k}$ which satisfies obvious transitivity conditions. Moreover, a morphism $\Ma\rightarrow\Na$ between two such a modules is a morphism $\Ma_{\mathfrak{Y},k}\rightarrow \Na_{\mathfrak{Y},k}$ of $\Da^{\dag}_{\mathfrak{Y},k}(\lambda)$-modules which is compatible with the additional structures. We will note this category $\Cb^{G_0}_{\mathfrak{X},\lambda}$ and for every $\Mb\in \Cb^{G_0}_{\mathfrak{X},\lambda}$, we will consider the projective limit
\begin{eqnarray*}
\Gamma(\Ma):= \varprojlim_{(\mathfrak{Y},k)\in\underline{\Fb}_{\mathfrak{X}}} H^{0}(\mathfrak{Y},\Ma_{\mathfrak{Y},k})
\end{eqnarray*}   
\justify
in the sense of the Abelian groups.
\justify
Now, let $M$ be a coadmissible $D(G_0,L)_{\lambda}$-module and  $V:=M'_b$ its associated locally analytic representation. The vector space of $\G(k)^{\circ}$-analytic vectors $V_{\G(k)^{\circ}-\text{an}}\subseteq V$ is stable under the action of $G_0$ and its dual $M_k:= (V_{\G(k)^{\circ}-\text{an}})'_b$ is a finitely presented $D(\G(k)^{\circ},G_0)$-module. In this situation, theorem 2 produces a coherent $\Da^{\dag}_{\mathfrak{Y},k}(\lambda)$-module 
\begin{eqnarray*}
\La oc^{\dag}_{\mathfrak{Y},k}(\lambda)(M_k) := \Da^{\dag}_{\mathfrak{Y},k}(\lambda)\otimes_{\Db^{\text{an}}(\G(k)^{\circ})_{\lambda}}M_k
\end{eqnarray*}  
for each element $(\mathfrak{Y},k)\in\underline{\Fb}_{\mathfrak{X}}$. We will note this family
\begin{eqnarray*}
\La oc^{G_0}_{\lambda}(M):= \left(\La oc^{\dag}_{\mathfrak{Y},k}(\lambda)(M_k)\right)_{(\mathfrak{Y},k)\in\underline{\Fb}_{\mathfrak{X}}}.
\end{eqnarray*}
\justify
\textbf{Theorem 3.} Let $\lambda\in\text{Hom}(\T,\G_m)$ be an algebraic character such that $\lambda+\rho\in\mathfrak{t}_L^*$ is a dominant and regular character of $\mathfrak{t}_L$. The functors $\La oc^{G_0}_{\lambda}(\bullet)$ and $\Gamma(\bullet)$ induce equivalences of categories between the category $\Cb_{G_0,\lambda}$ (of coadmissible $D(G_0,L)_{\lambda}$-modules) and the category $\Cb^{G_0}_{\mathfrak{X},\lambda}$.
\justify
Finally, the last part of this work is devoted to the study of coadmissible $D(G,L)_{\lambda}$-modules, where $G:=\G(L)$\footnote{Here $G_0$ is a (maximal) compact subgroup of $G$. This compactness property allows to define the structure of weak Fréchet-Stein algebra.}. To do this, we will consider the Bruhat-Tits building $\Bb$ of $G$ (\cite{BT1} and \cite{BT2}). It is a simplicial complex equipped with a  $G$-action. For any special vertex $v\in \Bb$, the theory of Bruhat and
Tits associates a reductive group $\G_v$ whose generic fiber is canonically isomorphic to $\G\times_{\text{Spec}(\mathfrak{o})} \text{Spec}(L)$. Let $X_v$ be the flag scheme of $\G_v$, and $\mathfrak{X}_v$ its formal completion along its special fiber. We consider the set $\underline{\Fb}$ composed of triples $(\mathfrak{Y}_v, k, v)$ such that $v$ is a special vertex, $\mathfrak{Y}_v\rightarrow \mathfrak{X}_v$ is an admissible blow-up of $\mathfrak{X}_v$ and $k \ge k_{\mathfrak{Y}_v}$. According to (\ref{partial_order_underline_F}) $\underline{\Fb}$ is partially ordered. In addition, for each special vertex $v\in\Bb$, each element $g\in G$ induces an isomorphism $\rho^v_g : \mathfrak{X}_v\rightarrow \mathfrak{X}_{vg}$, such that if $(\rho^v_g)^{\natural}: \Ob_{\mathfrak{X}_{vg}}\rightarrow (\rho^v_g)_*\Ob_{\mathfrak{X}_{v}}$ is the comorphism map and $\pi : \mathfrak{Y}_v \rightarrow \mathfrak{X}_v$ is an admissible blow-up along $V(\Ia)$, then the (admissible) blow-up along $V((\rho^v_g)^{-1}(\rho^v_g)_*\Ia)$ produces a scheme $\mathfrak{Y}_{vg}$ with an isomorphism $\rho^v_g:\mathfrak{Y}_v\rightarrow\mathfrak{Y}_{vg}$, such that $k_{\mathfrak{Y}_v}=k_{\mathfrak{Y}_{vg}}$ and for every $g,h\in G$ we have $\rho^{vg}_h\;\circ\;\rho^v_g=\rho^v_{gh}$.  
\justify
A coadmissible $G$-equivariant arithmetic $\Da(\lambda)$-module on $\underline{\Fb}$, consists of a family $(\Ma_{(\mathfrak{Y}_v,k,v)})_{(\mathfrak{Y}_v,k,v)\in \underline{\Fb}}$ of coherent $\Da^{\dag}_{\mathfrak{Y}_v,k}(\lambda)$-modules satisfying the condition $(\dag)$ plus some compatibility properties (definition \ref{coadmissible_G_arithmetic}) that allow us to form the projective limit
\begin{eqnarray*}
\Gamma(\Ma):= \varprojlim_{(\mathfrak{Y}_v,k,v)\in\underline{\Fb}} H^{0}(\mathfrak{Y}_v,\Ma_{(\mathfrak{Y}_v,k,v)}),
\end{eqnarray*}
\justify
which, as we will show, has a structure of coadmissible $D(G,L)_{\lambda}$-module. On the other hand, given a coadmissible $D(G,L)_{\lambda}$-module $M$, we consider
$V:= M'_b$ its continuous dual, which is a locally analytic representation of $G$. Then let $M_{v,k}$ be the dual space of the subspace $V_{\G_v(k)^{\circ}-\text{an}}\subseteq V$ of $\G_v(k)^{\circ}$-analytic vectors. For every $(\mathfrak{Y}_v,k,v) \in\underline{\Fb}$, we have a coherent $\Da^{\dag}_{\mathfrak{Y}_v,k}(\lambda)$-module
\begin{eqnarray*}
\La oc^{\dag}_{\mathfrak{Y}_v,k}(\lambda)(M_{v,k})= \Da^{\dag}_{\mathfrak{Y}_v,k}(\lambda)\otimes_{\Db^{\text{an}}(\G_v(k)^{\circ})_{\lambda}}M_{v,k}.
\end{eqnarray*}
We note this family $\La oc^{G}_{\lambda}(M)$. We will show the following result (theorem \ref{third_equivalence}).
\justify
\textbf{Theorem 4.} Let $\lambda\in\text{Hom}(\T,\G_m)$ be an algebraic character such that $\lambda+\rho\in\mathfrak{t}_L^*$ is a dominant and regular character of $\mathfrak{t}_L$. The functors $\La oc^G_{\lambda}(\bullet)$ and $\Gamma(\bullet)$ give an equivalence between the
categories of coadmissible $D(G,L)_{\lambda}$-modules and  coadmissible $G$-equivariant arithmetic $\Da(\lambda)$-modules.
\justify
The last task will be to study the projective limit
\begin{eqnarray*}
X_{\infty} := \varprojlim_{(\mathfrak{Y}_v,k,v)} \mathfrak{Y}_v.
\end{eqnarray*}
This is the Zariski-Riemann space associated to the rigid flag variety $X^{\text{rig}}$. We can also form the projective limit $\Da(\lambda)$ of the sheaves $\Da^{\dag}_{\mathfrak{Y},k}(\lambda)$ which is a sheaf of
 $G$-equivariant  differential operators on $\mathfrak{X}_{\infty}$. Similarly, if $(\Ma_{(\mathfrak{Y}_v,k,v)})_{(\mathfrak{Y}_v,k,v)\in\underline{\Fb}}$ is a coadmissible $G$-equivariant arithmetic $\Da(\lambda)$-module, then we can form the projective limit $\Ma_{\infty}$. The data $\Ma_{(\mathfrak{Y}_v,k,v)\in\underline{\Fb}}\rightsquigarrow \Ma_{\infty}$ induces a faithful functor from the category of coadmissible $G$-equivariant arithmetic $\Da(\lambda)$-modules on $\underline{\Fb}$ to the category of $G$-equivariant $\Da(\lambda)$-modules on $\mathfrak{X}_{\infty}$ (theorem \ref{fully_faithful}). In fact, this is a fully faithful functor as we will briefly explain in our Final remark (\ref{Final_remar}).
\justify
We summarize the main results of this work with the following commutative diagrams of functors (cf. \cite[Theorem 5.4.10]{PSS2}) 

\begin{eqnarray*}
\begin{tikzcd} [column sep=large, row sep=large]
\left\{
\hspace{0,2 cm}
\begin{matrix}
\text{Coadmissible} \\
D(G,L)_{\lambda}-\text{modules}
\end{matrix}
\hspace{0,2 cm}
\right\}\arrow[r, "\Lb oc^{G}_{\lambda}" , "\simeq" '] \arrow[d]
& 
\left\{
\hspace{0,2 cm}
\begin{matrix}
\text{Coadmissible}\;G-\text{equivariant}\\
\text{arithmetic}\;\Da(\lambda)-\text{modules}
\end{matrix}
\hspace{0,2 cm}
\right\} \arrow[d] \\
\left\{
\hspace{0,2 cm}
\begin{matrix}
\text{Coadmissible} \\
D(G_0,L)_{\lambda}-\text{modules}
\end{matrix}
\hspace{0,2 cm}
\right\} \arrow[r, "\Lb oc^{G_0}_{\lambda}" , "\simeq" ']
&
\left\{
\hspace{0,2 cm}
\begin{matrix}
\text{Coadmissible}\;G_0-\text{equivariant}\\
\text{arithmetic}\;\Da(\lambda)-\text{modules}
\end{matrix}
\hspace{0,2 cm}
\right\}
\end{tikzcd}
\end{eqnarray*} 
Here the left-hand vertical arrow is the restriction functor coming from the homomorphism
\begin{eqnarray*}
D(G_0,L)_{\lambda}\rightarrow D(G,L)_{\lambda}
\end{eqnarray*}
and the right-hand vertical arrow is the forgetful functor.
\justify
\textbf{Acknowledgements:} The present article contains a part of the author PhD thesis written at the Universities of Strasbourg and Rennes 1 under the supervision of C. Huyghe and T. Schmidt. Both have always been very patient and attentive supervisors. For this, I express my deep gratitude to them. 
\justify
 \textbf{Notation:} Throughout this work $\varpi$ will denote the uniformizer of $\mathfrak{o}$. Furthermore, if $Y$ is an arbitrary noetherian scheme over $\mathfrak{o}$, then for every $j\in\N$ we will denote by $Y_j:=Y\times_{\text{Spec}(\mathfrak{o})}\text{Spec}(\mathfrak{o}/\varpi^{j+1})$ the reduction modulo $\varpi^{j+1}$, and by
\begin{eqnarray*}
\mathfrak{Y}= \varinjlim_{j}Y_j
\end{eqnarray*}
the formal completion of $Y$ along the special fiber. Moreover, if  $\Eb$ is a sheaf of $\mathfrak{o}$-modules on $Y$ then its $\varpi$-completion $\Ea:=\varprojlim_{j}\Eb/\varpi^{j+1}\Eb$ will be considered as a sheaf on $\mathfrak{Y}$. Finally, the base change of a sheaf of $\mathfrak{o}$-modules on $Y$ (resp. on $\mathfrak{Y}$) to $L$ will always be denoted by the subscript $\Q$. For instance $\Eb_{\Q}:=\Eb\otimes_{\mathfrak{o}}L$ (resp. $\Ea_{\Q}:=\Ea\otimes_{\mathfrak{o}}L$).
\newpage

\section{Arithmetic definitions}

\subsection{p-adic coefficients and divided powers}\label{p-adic_coeff}
\justify
Let $p$ be a prime number and let us fix a positive integer $m$. Throughout this work, we will denote by $\Z_p$ the ring of $p$-adic integers and by $\Z_{(p)}$ the localization of $\Z$ with respect to the prime ideal $(p)$. Moreover, if $k\in\N$, we will note $q_k$ the quotient of the euclidean division of $k$ by $p^m$. Berthelot has introduced in \cite{Berthelot1} the following coefficients for any two integers $k,k'$ with $k\ge k'$

\begin{eqnarray*}
\stirlingii{k}{k'}:=\displaystyle\frac{q_k!}{q'_k!q''_k!},
\hspace{0.5 cm}
k'':= k-k'.
\end{eqnarray*}
\justify
In fact, we can generalize these coefficients for multi-indexes $\underline{k}=(k_1,\; ...,\; k_N)\in \N^{N}$ by defining $
q_{\underline{k}}! := q_{k_1}!\; ... \; q_{k_N}!$ and 
\begin{eqnarray*}
\stirlingii{\underline{k}}{\underline{k'}} := \displaystyle \frac{q_{\underline{k}}!}{q_{\underline{k'}}! q_{\underline{k''}}!}\in \N
\hspace{0.5 cm}
\text{and}
\hspace{0.5 cm}
\abinom{\underline{k}}{\underline{k'}} := \binom{\underline{k}}{\underline{k'}}\stirlingii{\underline{k}}{\underline{k'}}^{-1} \in \Z_p.
\end{eqnarray*}
\justify
Now, let A be a $\Z_{(p)}$ algebra. We say that a triple $(I,J,\gamma)$ is an $m$-PD ideal of $A$, if $\gamma$ defines a structure of divided powers on $J$ (a PD-structure in the sense of \cite{Berthelot2}) and $I$ is endowed of a system of partial divided powers, meaning that for any integer $k$, which decomposes as $k=p^mq + r$ (with $r<p^m$), there exists an operation defined for every $x\in I$ by

\begin{eqnarray*}
x^{k} = x^r \gamma_k\left(x^{p^m}\right).
\end{eqnarray*} 

\begin{exa}\label{Example_PD_structure}
Let  $\mathfrak{o}$ be a discrete valuation ring of unequal characteristic $(0,p)$ and uniformizing parameter $\varpi$. Let us write $p=u\varpi^e$, with $u$ a unit of $\mathfrak{o}$ and $e$ a positive integer (called the absolute ramification index of $\mathfrak{o}$). Let $k\in\N$. Then $\gamma_v(x):=x^v/v!$ defines a PD-structure on $(\varpi)^{k}$ if and only if $e\le k(p-1)$ \cite[Section 3, examples 3.2 (3)]{Berthelot2}. In particular, we dispose of a PD-structure on $(p)\subseteq \Z_{(p)}$. We let $x^{[k]}:=\gamma_{k}(x)$ and we denote by $((p), [])$ this PD-ideal. Moreover, if $k\le e-1$ and $m\ge log_p(k)$, then $(\varpi)^k$ endowed with the preceding PD-structure defines an m-PD-structure on $(\varpi)$ \cite [Section 1.3, examples (i)]{Berthelot1}.
\end{exa}

\subsection{Arithmetic differential operators}\label{ADP}
\justify
Let us suppose that $\mathfrak{o}$ is endowed with the m-PD-structure $(\mathfrak{a},\mathfrak{b},[\; ])$ defined in example \ref{Example_PD_structure}. Let $X$ be a smooth $\mathfrak{o}$-scheme, and $\Ib\subset\Ob_{X}$ a quasi-coherent ideal. Let us consider the sheaf of principal parties $\Pb_{(m)}(X)$ \cite[Section 2.1]{Berthelot1},which contains an m-PD structure $(\bar{\Ib},\tilde{\Ib}, [\; ])$ and the sequence of ideals $(\bar{\Ib}^{\{n\}})_{n\in\N}$ defining the m-PD-filtration \cite[1.3]{Berthelot3}.\\
\medskip
For every $n\in\N$, the algebra

\begin{eqnarray*}
\Pb^{n}_{X,(m)}:=\Pb^{n}_{(m)}(\Ib)/\bar{\Ib}^{\{n\}}
\end{eqnarray*}
is quasi-coherent and can be considered as a sheaf on $X$. Moreover, the projections $p_1,p_2: X\times_{\mathfrak{o}}X\rightarrow X$ induce two morphisms $d_1,d_2:\Ob_{X}\rightarrow \Pb^{n}_{X,(m)}$ endowing $ \Pb^{n}_{X,(m)}$ of a \textit{left} and a \textit{right} structure of $\Ob_{X}$-algebra, respectively. 

\begin{defi}\label{Diff order n}
Let $m,n$ be positive integers. The sheaf of differential operators of level $m$ and order less or equal to n on $X$ is defined by
\begin{eqnarray*}
\Db^{(m)}_{X,n}:=\mathcal{H}om_{\Ob_{X}}(\Pb^n_{X,(m)},\Ob_{X}).
\end{eqnarray*}
\end{defi}
\justify
If $n\le n'$ then \cite[Proposition 1.4.1]{Berthelot1} gives us a canonical surjection $\Pb^{n'}_{X,(m)}\rightarrow\Pb^{n}_{X,(m)}$ which induces the injection $\Db^{(m)}_{X,n}\hookrightarrow\Db_{X,n'}^{(m)}$ and the sheaf of \textit{differential operators of level} $m$ is defined by
\begin{eqnarray*}
\Db_{X}^{(m)}:=\bigcup_{n\in\N}\Db^{(m)}_{X,n}.
\end{eqnarray*}
We remark for the reader that by definition $\Db_{X}^{(m)}$ is endowed with a natural filtration called the \textit{order filtration}, and like the sheaves $\Pb^{n}_{X,(m)}$, the sheaves $\Db_{X,n}^{(m)}$ are endowed with two natural structures of $\Ob_{X}$-modules. Moreover, the sheaf $\Db^{(m)}_{X}$ acts on $\Ob_{X}$: if $P\in\Db^{(m)}_{X,n}$, then this action is given by the composition $\Ob_{X}\xrightarrow{d_1} \Pb^{n}_{X,(m)}\xrightarrow{P} \Ob_{X}$.\\
Finally, let us give a local description of $\Db^{(m)}_{X,n}$. Let $U$ be a smooth open affine subset of $X$ endowed with a family of local coordinates $x_1,\;.\;.\;.\;,x_N$. Let $dx_1,\;.\;.\;.\;,dx_N$ be a basis of $\Omega_{X}(U)$ and $\partial_{x_1},\;.\;.\;.\;,\partial_{x_N}$ the dual basis of $\Tb_X(U)$ (as usual, $\Tb_{X}$ and $\Omega_X$ denote the tangent and cotangent sheaf on $X$, respectively). Let $\underline{k}\in\N^N$. Let us denote by $|\underline{k}|=\sum_{i=1}^{N}k_i$ and   $\partial_i^{[k_i]}=\partial_{x_i}/k_i!$  for every $1\le i\le N$. Then, using multi-index notation, we have $\underline{\partial}^{[\underline{k}]}=\prod_{i=1}^{N}\partial_i^{[k_i]}$ and $\underline{\partial}^{<\underline{k}>}=q_{\underline{k}}!\underline{\partial}^{[\underline{k}]}$. In this case, the sheaf $\Db_{X,n}^{(m)}$ has the following description on $U$

\begin{eqnarray}\label{locally_Berthelot}
\Db_{X,n}^{(m)}(U)=\bigg\{\sum_{|\underline{k}|\le n}a_{\underline{k}}\underline{\partial}^{<\underline{k}>}\;|\; a_{\underline{k}}\in\Ob_{X}(U)\;\text{and}\;\underline{k}\in\N^{N}\bigg\}.
\end{eqnarray}

\subsection{Symmetric algebra of finite level}\label{Sym_alg}
\justify
In this subsection we will focus on introducing the constructions in \cite{Huyghe1}. As before, let $X$ denote a smooth $\mathfrak{o}$-scheme and let us consider $\Lb$ a locally free $\Ob_X$-module of finite rank, $\textbf{S}_X(\Lb)$ the symmetric algebra associated to $\Lb$ and $\Ib$ the ideal of homogeneous elements of degree 1. If $\Pb_{\textbf{S}_{X}(\Lb),(m)}(\Ib)$ denotes the $m$-divided power enveloping of $(\textbf{S}_X(\Lb),\Ib)$ (\cite[Proposition 1.4.1]{Berthelot1}) then we can consider the coherent sheaves on $X$ 
\begin{eqnarray}
\Gamma_{X,(m)}(\Lb):=\Pb_{\textbf{S}_{X}(\Lb),(m)}(\Ib)\;\;\;\;\text{and}\;\;\;\; \Gamma_{X,(m)}^n(\Lb):= \Gamma_{X,(m)}(\Lb)/\bar{\Ib}^{\{n+1\}}.
\end{eqnarray}
Those algebras are graded \cite[Proposition 1.3.3]{Huyghe1} and if $\eta_1,...,\eta_N$ is a local basis of $\Lb$, we have 
\begin{eqnarray*}
\Gamma_{X,(m)}^{n}(\Lb)=\displaystyle\bigoplus_{|\underline{l}|\le n}\Ob_{X}\underline{\eta}^{\{\underline{l}\}}.
\end{eqnarray*}
As before $\underline{\eta}^{\{\underline{l}\}}=\prod_{i=1}^{N}\eta_i^{\{l_i\}}$ and $q_i!\eta_i^{\{l_i\}}=\eta^{l_i}$. We define by duality 
\begin{eqnarray*}
\text{Sym}^{(m)}(\Lb):=\displaystyle\bigcup_{k\in\N}\mathcal{H}om_{\Ob_X}\left(\Gamma_{X,(m)}^k(\Lb^{\vee}), \Ob_{X}\right),
\end{eqnarray*}
By \cite[Propositions 1.3.1, 1.3.3 and 1.3.6]{Huyghe1} we know that $\text{Sym}^{(m)}(\Lb)=\oplus_{n\in\N}\text{Sym}^{(m)}_n(\Lb)$ is a commutative graded algebra with noetherian sections over any open affine subset. Moreover, locally  over a basis $\eta_1,...,\eta_N$ of $\Lb$ we have the following description
\begin{eqnarray*}
\text{Sym}^{(m)}_n(\Lb)=\displaystyle\bigoplus_{|\underline{l}|=n}\Ob_{X}\underline{\eta}^{<\underline{l}>},\;\;\;\text{where}\;\;\; \frac{l_i !}{q_i!}\eta_i^{<l_i>}=\eta_i^{l_i}.
\end{eqnarray*}

\begin{rem}
By \cite[A.10]{Berthelot2} we have that $\textbf{S}^{(0)}_X(\Lb)$ is the symmetric algebra of $\Lb$, which justifies the terminology.
\end{rem}
\justify
We end this subsection by  remarking the following results from \cite{Huyghe1}. Let $\Ib$ be the kernel of the comorphism $\Delta^{\sharp}$ of the diagonal embedding  $\Delta: X\rightarrow X\times_{\text{Spec}(\mathfrak{o})}X$. In \cite[Proposition 1.3.7.3]{Huyghe1} Huyghe shows that the graded algebra associated to the m-PD-adic filtration of $\Pb_{X,(m)}$ it is identified with the graded m-PD-algebra $\Gamma_{X,(m)}(\Ib/\Ib^2)=\Gamma_{X,(m)}(\Omega_{X}^1)$. More exactly, we dispose of a canonical morphism of $\Ob_X$-algebras 
\begin{eqnarray*}
\textbf{S}_X(\Omega_X)\rightarrow gr_{\bullet} \Pb_{X, (m)}
\end{eqnarray*}
which extends, via universal property  \cite[Proposition 1.4.1]{Berthelot1}, to a morphism $\Gamma^n_{X,(m)}(\Omega_{X}^1) \xrightarrow{\simeq}gr_{\bullet}(\Pb^n_{X,(m)})$. By definition, it induces a graded morphism
\begin{eqnarray}\label{graded}
\text{Sym}^{(m)}(\Tb_{X})\rightarrow gr_{\bullet}\Db^{(m)}_{X}
\end{eqnarray}
which is in fact an isomorphism of $\Ob_X$-algebras.

\subsection{Arithmetic distribution algebra of finite level}\label{ADA}
\justify
As in the introduction, let us consider $\G$ a split connected reductive group scheme over $\mathfrak{o}$ and $m\in\N$ fixed. We propose to give a description of the algebra of distributions of level $m$ introduced in \cite{HS1}. Let $I$ denote the kernel of the surjective morphism of $\mathfrak{o}$-algebras $\epsilon_\G:\mathfrak{o}[\G]\rightarrow \mathfrak{o}$, given by the identity element of $\G$. We know that $I/I^2$ is a free $\mathfrak{o}=\mathfrak{o}[\G]/I$-module of finite rank. Let $t_1,\;.\;.\;.\;,t_l\in I$ such that modulo $I^2$ these elements form a basis of $I/I^2$. The $m$-divided power enveloping algebra of $(\mathfrak{o}[\G],\; I)$, denoted by $P_{(m)}(\G)$, it is a free $\mathfrak{o}$-module with basis the elements $\underline{t}^{\{\underline{k}\}}=t_{1}^{\{k_1\}}\;.\;.\;.\;t_{l}^{\{k_l\}}$,
where $q_{i}!t_{i}^{\{k_i\}}=t_i^{k_i}$, for every $k_i=p^mq_i+r_i$ and $0\le r_i<p^m$. These algebras are endowed with a decreasing filtration by ideals $I^{\{n\}}$ (the $m$-PD filtration), such that $I^{\{n\}}=\oplus_{|\underline{k}|\ge n}\mathfrak{o}\;\underline{t}^{\{\underline{k}\}}$.
The quotients $P^n_{(m)}(\G):=P_{(m)}(\G)/I^{\{n+1\}}$ are therefore $\mathfrak{o}$-modules generated by the elements $\underline{t}^{\{\underline{k}\}}$ with $|\underline{k}|\le n$ \cite[Proposition 1.5.3 (ii)]{Berthelot1}. Moreover, there exists an isomorphism of $\mathfrak{o}$-modules 
\begin{eqnarray*}
P^n_{(m)}(\G)\simeq \displaystyle\bigoplus_{|\underline{k}|\le n} \mathfrak{o}\;\underline{t}^{\{\underline{k}\}}
\end{eqnarray*}
and for any two integers $n$, $n'$ such that $n\le n'$ we have a canonical surjection $\pi^{n',n}:P^{n'}_{(m)}(\G)\rightarrow P^n_{(m)}(\G)$. The module of distributions of level $m$ and order $n$ is $D_n^{(m)}(\G):=\text{Hom}(P^n_{(m)}(\G),\mathfrak{o})$. \textit{The algebra of distributions of level m} is
\begin{eqnarray*}
D^{(m)}(\G):=\varinjlim_{n} D^{(m)}_{n}(\G),
\end{eqnarray*}
where the limit is formed respect to the maps  $\text{Hom}_{\mathfrak{o}}(\pi^{n',n},\mathfrak{o})$. The multiplication is defined as follows. By universal property (\cite[Proposition 1.4.1]{Berthelot1}) there exists a canonical application $\delta^{n,n'}:P_{(m)}^{n+n'}(\G)\rightarrow P^{n}_{(m)}\G)\otimes_{\mathfrak{o}}P^{n'}_{(m)}(\G)$. If $(u,v)\in D^{(m)}_n(\G)\times D^{(m)}_{n'}(\G)$, we define $u.v$ as the composition
\begin{eqnarray*}
u.v: P^{n+n'}_{(m)}(\G)\xrightarrow{\delta^{n,n'}} P^{n}_{(m)}(\G)\otimes_{\mathfrak{o}}P^{n'}_{(m)}(\G)\xrightarrow{u\otimes v}\mathfrak{o}.
\end{eqnarray*}
Let us denote by $\mathfrak{g}:=\text{Hom}_{\mathfrak{o}}(I/I^2,\mathfrak{o})$
the Lie  algebra of $\G$. This is a free $\mathfrak{o}$-module with basis $\xi_1,\;.\;.\;.\;,\xi_l$  defined as the dual basis of the elements $t_1,\;.\;.\;.\;, t_l$. If for every multi-index $\underline{k}\in\N^l$, $|\underline{k}|\le n$, we denote by $\underline{\xi}^{<\underline{k}>}$ the dual of the element $\underline{t}^{\{\underline{k}\}}\in P^n_{(m)}(\G)$, then $D^{(m)}_n(\G)$ is a free $\mathfrak{o}$-module of finite rank with a basis given by the elements $\underline{\xi}^{<\underline{k}>}$ with $|\underline{k}|\le n$ \cite[proposition 4.1.6]{HS1}. 

\begin{rem}\footnote{This remark exemplifies the local situation when $X=\text{Spec}(A)$ with $A$ a $\Z_{(p)}$-algebra \cite[Subsection 1.3.1]{Huyghe1}.}\label{Sym-Locally}
Let A be an $\mathfrak{o}$-algebra and $E$ a free $A$-module of finite rank with base  $(x_1,...,x_N)$. Let $(y_1, ... ,y_{N})$ be the dual base of  $E^{\vee}:=\text{Hom}_{A}(E,A)$. As in the preceding subsection, let $\textbf{S}(E^{\vee})$ be the symmetric algebra and $\textbf{I}(E^{\vee})$ the augmentation ideal. Let $\Gamma_{(m)}(E^{\vee})$  be the $m$-divided power enveloping algebra of $(\textbf{S}(E^{\vee}),\;\textbf{I}(E^{\vee}))$.  We put $\Gamma^n_{(m)}(E^{\vee}):=\Gamma_{(m)}(E^{\vee})/\overline{I}^{\{n+1\}}$. These are free $A$-modules with base $y_1^{\{k_1\}}...y_N^{\{k_N\}}$ with $\sum k_i\le n$ \cite[1.1.2]{Huyghe1}. Let $\{\underline{x}^{<\underline{k}>}\}_{|\underline{k}|\le n}$ be the dual base of Hom$_A(\Gamma_{(m)}^n(E^{\vee}),A)$. We define
\begin{eqnarray*}
Sym^{(m)}(E):=\displaystyle\bigcup_{n\in\N}\text{Hom}_{A}\left(\Gamma^{n}_{(m)}(E^{\vee}),A\right).
\end{eqnarray*}
This is a free $A$-module with a base given by all the $\underline{x}^{<\underline{k}>}$. The inclusion $\text{Sym}^{(m)}(E)\subseteq \text{Sym}^{(m)}(E)\otimes_{\mathfrak{o}}L$ gives  the relation
\begin{eqnarray}\label{Rel_base_change}
x_i^{<k_i>} = \displaystyle\frac{k_i !}{q_i !} x^{k_i}.
\end{eqnarray}
Moreover, it also has a structure of algebra defined as follows. By \cite[Proposition 1.3.1]{Huyghe1} there exists an application $\Delta_{n,n'}:\Gamma_{(m)}^{n+n'}(E^{\vee})\rightarrow\Gamma^{n}_{(m)}(E^{\vee})\otimes_{A}\Gamma^{n'}_{(m)}(E^{\vee})$, which allows to define the product of $u\in Hom_{A}(\Gamma^{n}_{(m)}(E^{\vee}),A)$ and $v\in Hom_{A}(\Gamma^{n'}_{(m)}(E^{\vee}),A)$  by the composition 
\begin{eqnarray*}
u.v: \Gamma^{n+n'}_{(m)}(E^{\vee})\xrightarrow{\Delta_{n,n'}}\Gamma^{n}_{(m)}(E^{\vee})\otimes_{A}\Gamma^{n'}_{(m)}(E^{\vee})\xrightarrow{u\otimes v} A.
\end{eqnarray*}
This maps endows $Sym^{(m)}(E)$  of a structure of a graded noetherian $\mathfrak{o}$-algebra \cite[Propositions 1.3.1, 1.3.3 and 1.3.6]{Huyghe1}.
\end{rem}
\justify
We have the following important properties \cite[Proposition 4.1.15]{HS1}.

\begin{prop}\label{gr.Dist algebra}
\begin{itemize}
\item[(i)] There exists a canonical isomorphism of graded $\mathfrak{o}$-algebras $gr_{\bullet}(D^{(m)}(\G))\simeq Sym^{(m)}(\mathfrak{g})$.
\item[(ii)] The $\mathfrak{o}$-algebras $gr_{\bullet}(D^{(m)}(\G))$ and $D^{(m)}(\G)$ are noetherian. 
\end{itemize}
\end{prop}

\subsection{Integral models}\label{Integral_models}
\justify
In this section we will assume that $X$ is a smooth $\mathfrak{o}$-scheme endowed with a right $\G$-action.

\begin{defi}\label{defi I.models}
Let $A$ be an $L$-algebra (resp. a sheaf of $L$-algebras). We say that an $\mathfrak{o}$-subalgebra $A_0$ (resp. a subsheaf of $\mathfrak{o}$-algebras) is an integral model of $A$ if $A_0\otimes_{\mathfrak{o}}L= A$.
\end{defi} 

\begin{rem}
Let us recall that throughout this paper $\mathfrak{g}$ denotes the Lie algebra of the split connected reductive group $\mathfrak{o}$-scheme $\G$ and $\Ub(\mathfrak{g})$ its universal enveloping algebra. As we have remarked in the introduction, if $\mathfrak{g}_L$ denotes the $L$-Lie algebra of the algebraic group  $\G_L=\G\times_{\text{Spec}(\mathfrak{o})}\text{Spec}(L)$ and $\Ub(\mathfrak{g}_L)$ its universal enveloping algebra, then $\Ub(\mathfrak{g})$ is an integral model of $\Ub(\mathfrak{g}_L)$.  Moreover, the algebra of distributions of level $m$, introduced in the preceding subsection,  it is also an integral model of $\Ub(\mathfrak{g}_{L})$ \cite[subsection 4.1]{HS1}. This latest example will be specially important in this work.
\end{rem}

\begin{prop}\label{morp.HS}
The right $\G$-action induces a canonical homomorphism of filtered $\mathfrak{o}$-algebras 
\begin{eqnarray*}
\Phi^{(m)}: D^{(m)}(\G)\rightarrow H^{0}(X,\Db^{(m)}_{X}).
\end{eqnarray*}
\end{prop}
\begin{proof}
The reader can find the proof of this proposition in \cite[Proposition 4.4.1 (ii)]{HS1}, we will briefly discuss the construction of $\Phi^{(m)}$.  The central idea in the construction is that if $\rho:X\times_{\mathfrak{o}}\G\rightarrow X$ denotes the $\G$-action, then the comorphism $\rho^{\natural}:\Ob_{X}\rightarrow\Ob_{X}\otimes_{\mathfrak{o}}\mathfrak{o}[\G]$ induces a morphism 
\begin{eqnarray*}
\rho^{(n)}_{m}:\Pb^{n}_{X,(m)}\rightarrow\Ob_{X}\otimes_{\mathfrak{o}}P^{n}_{(m)}(\G)
\end{eqnarray*}
for every $n\in\N$. Those applications are compatible when varying $n$. Let $u\in D^{(m)}_n(\G)$ we define $\Phi^{(m)}(u)$ by 
\begin{eqnarray*}
\Phi^{(m)}(u):\Pb^{n}_{X,(m)}\xrightarrow{\rho^{(n)}_m} \Ob_{X}\otimes_{\mathfrak{o}}P^{n}_{(m)}(\G)\xrightarrow{id\otimes u}\Ob_{X}.
\end{eqnarray*}
Again, those applications are compatible when varying $n$ and we get the morphism of the proposition.
\end{proof}

\begin{rem}  \label{Invarian global sections of group}
\begin{itemize}
\item[(i)] If $X$ is endowed with a left $\G$-action, then it turns out that $\Phi^{(m)}$ is an anti-homomorphism.
\item[(ii)] In \cite[Theorem 4.4.8.3]{HS1} Huyghe and Schmidt have shown that if $X=\G$ and we consider the right (resp. left) regular action, then the morphism of the preceding proposition is in fact a canonical filtered isomorphism (resp. an anti-isomorphism) between $D^{(m)}(\G)$ and  $H^0(\G, \Db^{(m)}_{\G})^{\G}$, the $\mathfrak{o}$-submodule of (left) $\G$-invariant global sections. This isomorphism induces a bijection between $D_n^{(m)}(\G)$ and $H^{0}(\G,\Db^{(m)}_{\G,n})^{\mathbb{G}}$, and it is compatible when varying $m$.  
\end{itemize}
\end{rem}
\justify
We will denote by 
\begin{eqnarray*}
\Phi^{(m)}_X : \Ob_{X}\otimes_{\mathfrak{o}} D^{(m)}(\G) \rightarrow \Db^{(m)}_X
\end{eqnarray*}
the morphism of sheaves (of $\mathfrak{o}$-modules) defined by: if $U\subset X$ is an open subset and $f\in \Ob_X (U)$, $u\in D^{(m)}(\G)$, then 
\begin{eqnarray*}
\Phi^{(m)}_{X,U} (f\otimes u):= f\cdot \Phi^{(m)}(u)|_{U}.
\end{eqnarray*}
\justify
Let us define $\Ab_{X}^{(m)}=\Ob_{X}\otimes_{\mathfrak{o}}D^{(m)}(\G)$, and let us remark that we can endow this sheaf with the skew ring multiplication coming from the action of $D^{(m)}(\G)$ on $\Ob_{X}$ via the morphism $\Phi^{(m)}_X$. This is
\begin{eqnarray}\label{skew_mult}
(f\otimes u) \cdot (g\otimes v) := \left( f\cdot \Phi^{(m)}_X (u)\right) g\otimes v + fg\otimes uv.
\end{eqnarray}
\justify
This multiplication defines over $\Ab_X^{(m)}$ a structure of a sheaf of associative $\mathfrak{o}$-algebras, such that it becomes an integral model of the sheaf of $L$-algebras $\Ub^\circ:=\Ob_{X_L}\otimes_{L}\Ub(\mathfrak{g}_L)$. To see this, let us recall how the multiplicative structure of the sheaf $\Ub^\circ$ is defined (cf. \cite[subsection 5.1]{PSS1} or \cite[section 2]{Milicic1}).
\justify
Differentiating the right action of $\G_L$ on $X_L$ we get a morphism of Lie algebras
\begin{eqnarray}\label{Action_Lie}
\tau:\ \mathfrak{g}_L\rightarrow H^0(X_L,\Tb_{X_L}).
\end{eqnarray}
This implies that $\mathfrak{g}_{L}$ acts on $\Ob_{X_L}$ by derivations and we can endow $\Ub^{\circ}$ with the skew ring multiplication
\begin{eqnarray}\label{smash}
(f\otimes\eta)(g\otimes\zeta)=\left( f\tau(\eta)\right) g\otimes\zeta + fg\otimes\eta\zeta
\end{eqnarray}
for $\eta\in\mathfrak{g}_L$, $\zeta\in\Ub(\mathfrak{g}_L)$ and $f,g\in\Ob_{X_L}$. With this product the sheaf $\Ub^{\circ}$ becomes a sheaf of associative algebras \cite[Section 2, page 11]{Milicic1}.

\begin{rem}\label{Operator_rep}
As in (\ref{skew_mult}) we can define a morphism (called the operator-representation) of sheaves of $L$-algebras
\begin{eqnarray*}
\Psi_{X_L}: \Ob_{X_L}\otimes_{L}\Ub(\mathfrak{g}_L) \rightarrow \Db_{X_L};
\hspace{0,3 cm}
f\otimes \eta\mapsto  f\; \tau (\eta)
\hspace{0,4 cm}
(f\in \Ob_{X_L},\;\eta\in\mathfrak{g}_L.).
\end{eqnarray*}
\justify
We get the following commutative diagram
\begin{eqnarray*}
\begin{tikzcd}
D^{(m)}(\G) \arrow[r,"\Phi^{(m)}"] \arrow[d,hook] & H^{0}(X,\Db^{(m)}_X)\arrow[d, hook]\\
\Ub(\mathfrak{g}_L) \arrow[r,"\Psi_{X_L}"]  & H^{0}(X_L,\Db_{X_L}).
\end{tikzcd}
\end{eqnarray*}
\end{rem}
\justify
Given that $D^{(m)}(\G)$ is an integral model of the universal enveloping algebra $\Ub(\mathfrak{g}_L)$, then by (\ref{skew_mult}) and (\ref{smash}) we can conclude that $\Ab_{X}^{(m)}$ is also a sheaf of associative $\mathfrak{o}$-algebras being a subsheaf of $\Ub^{\circ}$.

\begin{prop}\label{prop 1.4.1} \cite[Corollary 4.4.6]{HS1}
\begin{itemize}
\item[(i)]  The sheaf $\Ab_{X}^{(m)}$ is a locally free $\Ob_X$-module.
\item[(ii)]  There exists a unique structure over $\Ab_{X}^{(m)}$ of filtered $\Ob_{X}$-ring and there is a canonical isomorphism of graded $\Ob_{X}$-algebras $gr (\Ab_{X}^{(m)})= \Ob_{X}\otimes_{\mathfrak{o}}\text{Sym}^{(m)}(\mathfrak{g})$.
\item[(iii)]  The sheaf $\Ab_{X}^{(m)}$ (resp. $gr(\Ab_{X}^{(m)})$) is a coherent sheaf of $\Ob_{X}$-rings (resp. a coherent sheaf of $\Ob_{X}$-algebras), with noetherian sections over open affine subsets of $X$. 
\end{itemize}
\end{prop}

\section{Twisted arithmetic differential operators with congruence level}
\justify
In this chapter we will introduce  congruence levels to the constructions given in the preceding sections. This means, deformations of our (integral) differential operators. This notion will be a fundamental tool to define differential operators on an admissible blow-up of the flag $\mathfrak{o}$-scheme $X$.

\subsection{Linearization of group actions}\label{Linearization_alg_gps}

Let us start with the following definition from \cite[Chapter II, exercise 5.18]{Hartshorne1} (cf. \cite[Definition 3.1.1]{Brion1}).

\begin{defi}
Let $Y$ be an $\mathfrak{o}$-scheme. A (geometric) line bundle over $Y$ is a scheme $\textbf{L}$ together with a morphism $\pi:\textbf{L}\rightarrow Y$ such that $Y$ admits an open covering $(U_i)_{i\in I}$ satisfying the following two conditions:
\begin{itemize}
\item[(i)] For any $i\in I$ there exists an isomorphism $\psi_{i}: \pi^{-1}(U_i)\xrightarrow{\simeq} \mathbb{A}_{U_i}^1$.
\item[(ii)] For any $i,j\in I$ and for any open affine subset $V=\text{Spec}(A[x])\subseteq U_i\cap U_i$ the automorphism $\theta_{ij}: \psi_{j}\circ\psi_{i}^{-1}|_{V}: \mathbb{A}^{1}_{V}\rightarrow \mathbb{A}^{1}_{V}$ of $\mathbb{A}^{1}_V$ is given by a linear automorphism $\theta^{\natural}_{ij}$ of $A[x]$. This means, $\theta_{ij}^{\natural}(a)=a$ for any $a\in A$, and $\theta_{ij}^{\natural} (x)= a_{ij}x$ for a suitable $a_{ij}\in A$.
\end{itemize} 

\end{defi}
\justify
In the preceding definition, the scheme $\textbf{L}$ is obtained by glueing the trivial line bundles $p_{1,i}:U_i\times \mathbb{A}^1_{\mathfrak{o}}\rightarrow U_i$ via the linear transition functions $(a_{ij})$. Thus, each fibre $\textbf{L}_x$ is a line, in the sense that it has a canonical structure of a 1-dimensional affine space.

\begin{defi}
Given a line bundle $\pi: \textbf{L}\rightarrow Y$ and a morphism $\phi: Y'\rightarrow Y$, the pull-back $\phi^{*}(\textbf{L})$ is the fiber product $\textbf{L}\times_{Y}Y'$ equipped with its projection to $Y'$. 
\end{defi} 
\justify
Now, let $\pi:\textbf{L}\rightarrow Y$ be a line bundle over $Y$, then a \textit{section} of $\pi$ over an open subset $U\subset Y$ is a morphism $s: U\rightarrow \textbf{L}$ such that $\pi\circ s = id_U$. Moreover the presheaf $\Lb$ defined by 
\begin{eqnarray*}
U\subseteq Y\;\; \mapsto\;\;\{s:U\rightarrow \textbf{L}\;|\;\;s\;\; \text{is a section over}\;\; U\}
\end{eqnarray*}
is a sheaf called the \textit{sheaf of sections} of the line bundle $\textbf{L}$. This is an invertible sheaf (i.e., a locally free sheaf of rank 1).
\justify
On the other hand, if $\Eb$ is a locally free sheaf of rank 1 on $Y$ and we let 
\begin{eqnarray*}
\textbf{V}(\Eb):= \underline{\text{Spec}}_{Y}\left(\text{Sym}_{\Ob_{Y}}\left(\Eb\right)\right)
\end{eqnarray*}
be the line bundle over $Y$ associated to $\Eb$ \cite[1.7.8]{Grothendieck2}, then we  have a one-to-one correspondence between isomorphism classes of locally free sheaves of rank 1 on $Y$ and isomorphic classes of (geometric) line bundles over $Y$ \cite[Chapter II, Exercises 5.1 (a) and 5.18 (d)]{Hartshorne1}
\begin{eqnarray}\label{Correspondence_alg_geo}
\begin{matrix}
\{\text{Isomorphism classes of locally free sheaves of rank 1}\} &\leftrightarrow & \{\text{Isomorphic classes of line bundles}\}\\
\Eb &\mapsto & \textbf{V}(\Eb^{\vee})\\
\Lb &\mapsfrom & \textbf{L} 
\end{matrix}
\end{eqnarray}
\justify
Let $\pi:\textbf{L}\rightarrow Y$ be a line bundle over $Y$, let $\Lb$ be its sheaf of sections and $\phi:Y'\rightarrow Y$ a morphism of schemes; an easy calculation shows that the sheaf of sections of the pull-back line bundle $\phi^*(\textbf{L}):=\textbf{L}\times_{Y}Y'\rightarrow Y'$ is equal to $\phi^{*}(\Lb)$.
\justify
Let us suppose now that $Y$ is endowed with a right $\G$-action $\alpha: Y \times_{\text{Spec}(\mathfrak{o})}\G\rightarrow Y$. In particular, for every $g\in\G(\mathfrak{o})$ we dispose of a translation morphism
\begin{eqnarray*}
\rho_g : Y=Y\times_{\text{Spec}(\mathfrak{o})}\text{Spec}(\mathfrak{o}) \xrightarrow{id_Y\times g} Y\times_{\text{Spec}(\mathfrak{o})}\G \xrightarrow{\alpha} Y
\end{eqnarray*}
\justify
In the next lines we will study (geometric) line bundles which are endowed with a right $\G$-action.
 
\begin{defi}
Let $\pi:\textbf{L}\rightarrow Y$ be a line bundle. A $\G$-Linearization of $\textbf{L}$ is a right $\G$-action $\beta:\textbf{L}\times_{\text{Spec}(\mathfrak{o})}\G\rightarrow \textbf{L}$ satisfying the following two conditions:
\begin{itemize}
\item[(i)] The diagram 
\begin{eqnarray*}
\begin{tikzcd}
\textbf{L}\times_{\text{Spec}(\mathfrak{o})}\G \arrow[r, "\beta"] \arrow[d, "\pi\times id_{\G}"]
& \textbf{L} \arrow[d,"\pi"]\\
Y\times_{\text{Spec}(\mathfrak{o})}\G \arrow[r,"\alpha"]
& Y
\end{tikzcd}
\end{eqnarray*}
is commutative. 
\item[(ii)] The action on the fibers is $\mathfrak{o}$-linear.
\end{itemize} 
\end{defi}
\justify
Let $g\in\G(\mathfrak{o})$ and let us suppose that  $\Psi:\alpha^{*}(\textbf{L})\rightarrow p_1^*(\textbf{L})$ is a morphism of line bundles over $Y\times_{\text{Spec}(\mathfrak{o})} \G$. Let us consider the translation morphism 
\begin{eqnarray*}
\rho_g: Y=Y \times_{\text{Spec}(\mathfrak{o})}\text{Spec}(\mathfrak{o}) \xrightarrow{id_Y\times g} Y\times_{\text{Spec}(\mathfrak{o})}\G \xrightarrow{\alpha} Y.
\end{eqnarray*}
We have the relations $(id_Y \times g)^*\alpha^*(\textbf{L})=\rho_g^*(L)$ and $(id_Y\times g)^*p_1^*(\textbf{L})= \textbf{L}$. So every morphism of line bundles $\Psi:\alpha^{*}(\textbf{L})\rightarrow p_1^*(\textbf{L})$ induces morphisms $\Psi_{g}: \rho_g^*(L)\rightarrow L$ for all $g\in\G(\mathfrak{o})$. The following reasoning can be found in \cite[Page 104]{Dolgachev} or \cite[Lemma 3.2.4]{Brion1}.

\begin{prop}\label{Action_line_bundle}
Let $\pi:\textbf{L}\rightarrow Y$ be a line bundle over $Y$ endowed with a $\G$-linearization $\beta:\textbf{L}\times_{\text{Spec}(\mathfrak{o})}\G\rightarrow \textbf{L}$. Then there exists an isomorphism 
\begin{eqnarray*}
\Psi: \alpha^{*}(\textbf{L})\rightarrow p_1^*(\textbf{L})
\end{eqnarray*}
of line bundles over $\textbf{L}\times_{\text{Spec}(\mathfrak{o})}\G$, such that $\Psi_{gh}=\Psi_{g}\circ \rho_{g}^{*}(\Psi_{h})$\;  for all $g,h\in\G(\mathfrak{o})$.
\end{prop}

\begin{proof}
By definition of linearization we have the following commutative diagram
\begin{eqnarray*}
\begin{tikzcd}
\textbf{L}\times_{\text{Spec}(\mathfrak{o})}\G
\arrow[drr, bend left, "\pi \times id_{\G}"]
\arrow[ddr, bend right, "\beta"]
\arrow[dr, dotted, "\psi" description] & & \\
& \alpha^*(\textbf{L}) \arrow[r, "p_2"] \arrow[d, "p_1"]
& Y\times_{\text{Spec}(\mathfrak{o})}\G \arrow[d, "\alpha"] \\
& \textbf{L} \arrow[r, "\pi"]
& Y.
\end{tikzcd}
\end{eqnarray*}
By universal property there exists a unique morphism of line bundles $\psi: p_1^{*}(\textbf{L})\rightarrow \alpha^{*}(\textbf{L})$, which is linear on the fibers since so is $\beta$. Let $g\in\G(\mathfrak{o})$. To see that $\psi$ is an isomorphism we can use the correspondence (\ref{Correspondence_alg_geo}). In this case, if $x\in Y$,  $g\in\G(\mathfrak{o})$ and  $\psi_{(x,g)}:\Lb_{x}\rightarrow \Lb_{xg}$ denotes the respective morphism between the stalks, then $\psi_{(x,g)}$ is an isomorphism, being $\psi_{(xg,g^{-1})}$ the inverse.  
\justify
Let $g,h\in\G(\mathfrak{o})$. Applying $(id_X\times g)^*$ to $\psi$ we get the morphism $\psi_g :\textbf{L}\rightarrow \rho_g^*(\textbf{L})$ and given that $\beta$ is a right action ($\rho_h\;\circ\;\rho_g=\rho_{gh}$), it  fits into the following commutative diagram 
\begin{eqnarray*}
\begin{tikzcd}[row sep = large]
    \textbf{L} \arrow{r}{\psi_g} \arrow[swap]{dr}{\psi_{gh}} & \rho_g^*(\textbf{L}) \arrow{d}{\rho^*_g (\psi_h)} \\
     & \rho_g^* \rho_h^* (\textbf{L}) = \rho_{gh}^*(\textbf{L}).
  \end{tikzcd}
\end{eqnarray*}
Moreover, since $\psi_{g}: \textbf{L}\rightarrow \rho^*_g (\textbf{L})$ is an isomorphism for every $g\in\G(\mathfrak{o})$ (the fiber over $x\in Y$ coincides with $\psi_{(x,g)}$) then we can consider the morphism $\Psi_g :=\psi_g^{-1}: \rho^*_g (\textbf{L})\rightarrow \textbf{L}$ which coincides with the fibers of the morphism 
\begin{eqnarray*}
\Psi := \psi^{-1}: \alpha^* (\textbf{L})\rightarrow p_1^* (\textbf{L}).
\end{eqnarray*}
By construction, these morphism satisfy the cocycle condition of the proposition. This means that for every $g,h\in \G(\mathfrak{o})$, we have 
\begin{eqnarray*}
\Psi_{gh}=\Psi_{g}\;\circ\; \rho_{g}^{*}(\Psi_{h}).
\end{eqnarray*} 
\end{proof}

\subsection{Associated Rees rings and arithmetic differential operators with congruence level}
\justify
Throughout this sections $X$ will denote a smooth scheme over $\mathfrak{o}$. As usual, we will denote by $\Db^{(m)}_{X}$ the sheaf of level $m$ differential operators on $X$. As we have remarked in the previous chapter, those sheaves come equipped with a filtration 
\begin{eqnarray*}
\Ob_{X}\subseteq \Db^{(m)}_{X,1}\subseteq ... \subseteq \Db^{(m)}_{X,d}\subseteq ... \subseteq \Db^{(m)}_{X},
\end{eqnarray*}
with $\Db^{(m)}_{X,d}$ the sheaf of level $m$ differential operators of order less or equal than $d$.
\justify
\footnote{This digression can be found before the proof of \cite[Proposition 3.3.7]{HPSS}.} Now, let $\Ab$ be a sheaf of $\mathfrak{o}$-algebras endowed with a positive filtration $(F_{d}\Ab)_{d\in\N}$ and such that $\mathfrak{o}\subset F_0\Ab$. The sheaf $\Ab$ gives rise to a subsheaf of graded rings $R(\Ab)$  of the polynomial algebra $\Ab[t]$ over $\Ab$. This is defined by
\begin{eqnarray*}
R(\Ab):=\displaystyle\bigoplus_{i\in\N}F_i\Ab\cdot t^i,
\end{eqnarray*}
its associated Rees ring. This subsheaf comes equipped with a filtration by the sheaves of subgroups 
\begin{eqnarray*}
R_{d}(\Ab):= \displaystyle\bigoplus_{i=0}^d F_{i}\Ab\cdot t^i \subseteq R(\Ab).
\end{eqnarray*}
Specializing $R(\Ab)$ in an element $\mu\in\mathfrak{o}$ we get a sheaf of filtered subrings $\Ab_{\mu}$ of $\Ab$. More exactly, $\Ab_{\mu}$ equals the image under the homomorphism of sheaves of rings $\phi_{\mu}:R(\Ab)\rightarrow \Ab$, sending $t\mapsto \mu$, and it is equipped with the filtration induced by $\Ab$. Moreover, if the sheaf of graded rings gr$(\Ab)$, associated to the filtration $(F_d\Ab)_{d\in\N}$, is flat over $\mathfrak{o}$ then\footnote{This is \cite[Claim 3.3.10.]{HPSS}.} 
\begin{eqnarray}\label{Filt_esp_of_Rees}
F_d\Ab_{\mu}=\displaystyle\sum_{i=0}^d \mu^{i}F_{i}\Ab.
\end{eqnarray}
\justify
If $\psi:\Ab\rightarrow \Bb$ is a morphism of positive filtered $\mathfrak{o}$-algebras (with $\mathfrak{o}\subseteq F_0\Ab$ and $\mathfrak{o}\subseteq F_0\Bb$), then the commutative diagram
\begin{eqnarray*}
\begin{tikzcd}[column sep=huge]
R(\Ab) \arrow[r,"a_dt^d\mapsto\psi(a_d)t^d"] \arrow[d, "\phi_{\mu}"]
& R(\Bb) \arrow[d, "\phi_{\mu}"]\\
\Ab \arrow[r, "\psi"]
&\Bb
\end{tikzcd}
\end{eqnarray*}
gives us a filtered morphism of rings $\psi_{\mu}:\Ab_{\mu}\rightarrow\Bb_{\mu}$. This in particular implies that for $\mu\in\mathfrak{o}$ fixed, the preceding process is functorial.

\begin{rem}\label{Rees_rings_for_algebras}
The previous digression is well-known for rings. In this setting we have results completely analogues to the ones presented so far (\cite[Chapter 12, section 6]{MRS}). We will use these results in the next sections.
\end{rem}
\justify
Now, let $k$ be a non-negative integer called a \textit{congruence level} \cite[Subsection 2.1]{HSS}. By using the order filtration $(\Db^{(m)}_{\widetilde{X}})_{d\in\N}$ of the sheaf $\Db^{(m)}_{X}$, we can define the sheaf of \textit{arithmetic differential operators of congruence level k}, $\Db^{(m,k)}_{X}$, as the subsheaf of $\Db^{(m)}_{X}$ given by the specialization of $R(\Db^{(m)}_{X})$ in $\varpi^{k}\in\mathfrak{o}$. This means
\begin{eqnarray*}
\Db^{(m,k)}_{X}:=\displaystyle\sum_{d\in\N}\varpi^{kd} \Db^{(m)}_{X,d}.
\end{eqnarray*}
By (\ref{graded}) and \cite[Proposition 1.3.4.2]{Huyghe1} we can also conclude that, if $(\Db^{(m,k)}_{X,d})_{d\in\N}$ denotes the order filtration induced by $\Db^{(m)}_{X}$, then 
\begin{eqnarray*}
\Db^{(m,k)}_{X,d}=\displaystyle\sum_{i=0}^{d}\varpi^{ki}\Db_{X,i}^{(m)}.
\end{eqnarray*}
In local coordinates we can describe the sheaf $\Db^{(m,k)}_{X}$ in the following way. Let $U\subseteq X$ be an open affine subset endowed with coordinates $x_1,...,x_N$. Let $dx_1,...,dx_N$ be a basis of $\Omega_{X}(U)$ and $\partial_{x_1},...,\partial_{x_N}$ the dual basis of $\Tb_{X}(U)$. By using the notation in section \ref{ADP}, one has the following description \cite[Subsection 2.1]{HSS}
\begin{eqnarray*}
\Db^{(m,k)}_{X}(U)=\bigg\{\displaystyle\sum_{\underline{v}}^{<\infty}\varpi^{k|\underline{v}|}a_{\underline{v}}\underline{\partial}^{<\underline{v}>}|\; a_{\underline{v}}\in\Ob_{X}(U) \bigg\}.
\end{eqnarray*}

\subsection{Arithmetic differential operators acting on a line bundle}\label{Diff_op_acting_bundle}
\justify
Throughout this subsection $X:=\G/\B$ will always denote the flag scheme. For technical reasons (cf. Proposition \ref{morp.HS}) in this work we will always suppose that the group $\G$, and the scheme $X$ are endowed with the right regular $\G$-action. This means that for any $\mathfrak{o}$-algebra $A$ and $g_0,g\in \G(A)$ we have 
\begin{eqnarray*}
g_0\bullet g = g^{-1}g_0,\;\;\; g_0\;\textbf{N}(A)\bullet g = g^{-1}g_0\; \textbf{N}(A) \;\;\text{and}\;\; g_0\;\B(A)\bullet g=g^{-1}g_0\;\B(A).
\end{eqnarray*}
Under this action, the canonical projection $\G\rightarrow X$ is clearly $\G$-equivariant.
\justify
Finally, we recall for the reader that the sheaf $\Db^{(m)}_{X}$ is endowed with a left and a right structure of $\Ob_{X}$-module. These structures come from the canonical morphisms of rings $d_1,\;d_2:\Ob_{X}\rightarrow\Pb^{n}_{X,(m)}$, which are induced by universal property and the projections.  By construction, these actions also endow the sheaf $\Db^{(m,k)}_X$  with a left and a right structure of $\Ob_{X}$-module.

\begin{roots}\label{roots}\textbf{Dominant and regular characters.}
\end{roots}
\justify
Let us consider the positive system $\Lambda^{+}\subset\Lambda \subset X(\T)$ ($X(\T)=\text{Hom}(\T,\G_m)$ the group of algebraic characters) associated to the Borel subgroup scheme $\B\subset\G$. The Weyl subgroup $W=N_{\G}(\T)/\T$ acts naturally on the space $\mathfrak{t}^*_L:=\text{Hom}_L(\mathfrak{t}_L,L)$, and via differentiation $d: X(\T)\hookrightarrow \mathfrak{t}^*$ we may view $X(\T)$ as a subgroup of $\mathfrak{t}^*$ in such a way that $X^*(\T)\otimes_{\mathfrak{o}}L=\mathfrak{t}^*_L$. Let $\rho=\frac{1}{2}\sum_{\alpha\in\Lambda^{+}}\alpha$  be the so-called Weyl vector. Let $\check{\alpha}$ be a coroot of $\alpha\in\Lambda$ viewed as an element of $\mathfrak{t}_L$. An arbitrary weight $\lambda\in\mathfrak{t}^*_L$ is called \textit{dominant} if $\lambda(\check{\alpha})\ge 0$ for all $\alpha\in\Lambda^{+}$. The weight $\lambda$ is called \textit{regular} if its stabilizer under the $W$-action is trivial.

\begin{line_bundles}\label{line_bundles}\textbf{Line bundles on homogeneous spaces.}
\end{line_bundles}
\justify
Let us suppose now that $X:=\G/\B$ is again the smooth flag $\mathfrak{o}$-scheme. We dispose of a canonical isomorphism $\T\simeq \B/\textbf{N}$ which in particular implies that every algebraic character $\lambda\in\text{Hom}(\T,\G_m)$ induces a character of the Borel subgroup $\lambda:\B\rightarrow\G_m$. Let us consider the locally free action of $\B$ on the trivial fiber bundle $\G\times \mathfrak{o}$ over $\G$ given by 
\begin{eqnarray*}
b.(g,u):=(gb^{-1}, \lambda(b)u);\;\;\;(g\in\G,\; b\in\B,\; u\in\mathfrak{o}).
\end{eqnarray*}
\justify
We denote by $\textbf{L}(\lambda):=\B\diagdown (\G\times \mathfrak{o})$ the quotient space obtained by this action.\\
Let $\pi: \G\rightarrow X$ be the canonical projection. Since the map $\G\times\mathfrak{o}\rightarrow X$, $(g,u)\mapsto \pi(x)$ is constant on $\B$-orbits, it induces a morphism $\pi_{\lambda}:\textbf{L}(\lambda)\rightarrow X$. Moreover, given that $\pi$ is locally trivial \cite[Part II, 1.10 (2)]{Jantzen} $\pi_{\lambda}: \textbf{L}(\lambda)\rightarrow X$ defines a line bundle over $X$ \cite[Part I, 5.16]{Jantzen}. Furthermore, the right $\G$-action on $\G\times \mathfrak{o}$ given by
\begin{eqnarray*}
(g_0,u)\bullet g\mapsto (g^{-1}g_0,u)\;\;\;(g\in\G,\;(g_0,u)\in\G\times\mathfrak{o})
\end{eqnarray*}
induces a right action on $\textbf{L}(\lambda)$ for which $\textbf{L}(\lambda)$ turns out to be a $\G$-linearized line bundle on $X$. By  proposition \ref{Action_line_bundle}, the sheaf of sections $\Lb(\lambda)$ of the line bundle $\textbf{L}(\lambda)$ is a $\G$-equivariant invertible sheaf.

\begin{defi}
Let $\lambda\in\text{Hom}(\T,\G_m)$ be an algebraic character. For every congruence level $k\in\N$, we define the sheaf of level $m$ arithmetic differential operators acting on the line bundle $\Lb(\lambda)$ by
\begin{eqnarray*}
\Db_{X}^{(m,k)}(\lambda):=\Lb(\lambda)\otimes_{\Ob_{X}}\Db^{(m,k)}_X\otimes_{\Ob_{X}}\Lb(\lambda)^{\vee}.
\end{eqnarray*}
\end{defi}
\justify
The multiplicative structure of the sheaf $\Db^{(m,k)}_{X}(\lambda)$ is defined as follows. If $\alpha^{\vee},\beta^{\vee}\in\Lb(\lambda)^{\vee}$, $P,Q\in\Db^{(m,k)}_{X}$ and $\alpha,\beta\in\Lb(\lambda)$ then 
\begin{eqnarray}\label{mt}
\alpha\otimes P\otimes \alpha^{\vee}\bullet \beta\otimes Q \otimes \beta^{\vee} = \alpha\otimes P \left<\alpha^{\vee},\beta\right> Q\otimes \beta^{\vee}.
\end{eqnarray}
\justify
Moreover, the action of $\Db_{X}^{(m,k)}(\lambda)$ on $\Lb(\lambda)$ is given by
\begin{eqnarray*}
\left(t\otimes P\otimes t^{\vee}\right)\bullet s := \left(P\bullet<t^{\vee},s>\right)t \hspace{1 cm} (s,t\in \Lb(\lambda)\;\text{and}\; t^{\vee}\in\Lb(\lambda)^{\vee}).
\end{eqnarray*}

\begin{rem}
Given that the locally free $\Ob_{X}$-modules of rank one $\Lb(\lambda)^{\vee}$ and $\Lb(\lambda)$  are in particular flat, the sheaf $\Db^{(m,k)}_{X}(\lambda)$ is filtered by the order of twisted differential operators. This is, the subsheaf $\Db^{(m,k)}_{X,d}$ of $\Db^{(m,k)}_{X}$, of differential operators of order less that $d$, induces a subsheaf of twisted differential operators of order less than $d$ defined by 
\begin{eqnarray*}
\Db^{(m,k)}_{X,d}(\lambda):=\Lb(\lambda)\otimes_{\Ob_{X}}\Db^{(m,k)}_{X,d}\otimes_{\Ob_{X}}\Lb(\lambda)^{\vee},
\end{eqnarray*}
and given than the tensor product preserves inductive limits, we obtain 
\begin{eqnarray*}
\Db^{(m,k)}_{X}(\lambda)=\varinjlim_{d}\Db^{(m,k)}_{X,d}(\lambda).
\end{eqnarray*}
Moreover, the exact sequence
\begin{eqnarray*}
0\rightarrow \Db^{(m,k)}_{X,d-1}\rightarrow \Db^{(m,k)}_{X,d}\rightarrow \Db^{(m,k)}_{X,d}/\Db^{(m,k)}_{X,d-1}\rightarrow 0
\end{eqnarray*}
induces the exact sequence 
\begin{eqnarray*}
0\rightarrow \Db^{(m,k)}_{X, d-1}(\lambda)\rightarrow \Db^{(m,k)}_{X,d}(\lambda)\rightarrow \Lb(\lambda)\otimes_{\Ob_{X}} \Db^{(m,k)}_{X,d}/\Db^{(m,k)}_{X,d-1}\otimes_{\Ob_{X}}\Lb(\lambda)^{\vee}\rightarrow 0
\end{eqnarray*}
which tells us that 
\begin{eqnarray*}
gr \left(\Db^{(m,k)}_{X}(\lambda)\right)\simeq \Lb(\lambda)\otimes_{\Ob_{X}}gr \left(\Db^{(m,k)}_{X}\right)\otimes_{\Ob_{X}}\Lb(\lambda)^{\vee}\simeq gr \left(\Db^{(m,k)}_{X}\right).
\end{eqnarray*}
The second isomorphism is defined by $\alpha\otimes P \otimes \alpha^{\vee} \mapsto \alpha^{\vee}(\alpha)P$. This is well defined because $gr\left(\Db^{(m,k)}_{X}\right)$ is in particular a commutative ring.
\end{rem}

\begin{prop}\label{graded_lambda}
There exists a canonical isomorphism of graded sheaves of algebras 
\begin{eqnarray*}
gr_{\bullet}\left(\Db^{(m,k)}_{X}(\lambda)\right)\xrightarrow{\simeq} \text{Sym}^{(m)}(\varpi^{k}\Tb_{X}).
\end{eqnarray*}
\end{prop}
\begin{proof}
By (\ref{graded}), and the fact that $\Db^{(m,k)}_{X}$ and $\varpi^k\Tb_{X}$ are locally free sheaves (and therefore free $\varpi$-torsion) we have the following short exact sequence
\begin{eqnarray*}
0\rightarrow \Db^{(m,k)}_{X,d-1}\rightarrow \Db^{(m,k)}_{X,d}\rightarrow \text{Sym}^{(m)}_d\left(\varpi^k\Tb_{X}\right)\rightarrow 0,
\end{eqnarray*}
which gives us the isomorphisms
\begin{eqnarray*}
\text{Sym}^{(m)}\left(\varpi^k\Tb_{X}\right)\simeq gr_{\bullet}\left(\Db^{(m,k)}_{X}\right)\simeq gr_{\bullet}\left(\Db^{(m,k)}_{X}(\lambda)\right).
\end{eqnarray*}
\end{proof}
\justify
In the next proposition we will use the notation introduced in subsections \ref{p-adic_coeff} and \ref{ADP}.

\begin{prop}\label{algebraic_local_desc}
There exists a covering $\Sb$ of $X$ by affine open subsets such that, over every open subset $U\in\Sb$ the rings $\Db^{(m,k)}_{U}(\lambda)$ and $ \Db^{(m,k)}_{U}$ are isomorphic. 
\end{prop}
\begin{proof}
Let us star by considering $U\subset X$ an affine open subset endowed with local coordinates $x_1,...,x_M$. For every $\underline{v}\in\N^M$ and $f\in\Ob_{X}(U)$ we have the following relation \cite[proposition 2.2.4, iv]{Berthelot1}
\begin{eqnarray*}
\underline{\partial}^{<\underline{v}>}f=\displaystyle\sum_{\underline{v}'+\underline{v}''=\underline{v}}\stirlingii{\underline{v}}{\underline{v}'}\underline{\partial}^{<\underline{v}'>}(f)\underline{\partial}^{<\underline{v}''>} \in \Db^{(m,0)}_{U}=\Db^{(m)}_{U}.
\end{eqnarray*}
Now, let's take an affine covering $\Sb$ of $X$ such that for every $U\in\Sb$, $U$ is endowed with local coordinates and there exists a local section $\alpha\in\Lb(\lambda)(U)$ such that $\Lb(\lambda)|_{U}=\alpha\Ob_{U}$ and $\mathcal{L}(\lambda)^{\vee}|_{U}=\alpha^{\vee}\Ob_{U}$, where $\alpha^{\vee}$ denotes the dual element associated to $\alpha$. Let us show that
\begin{eqnarray}\label{B1}
\Db^{(m,k)}_{U}(\lambda)=\displaystyle\bigoplus_{\underline{v}}\varpi^{k|\underline{v}|}\Ob_{U}\alpha\otimes\underline{\partial}^{<\underline{v}>}\otimes\alpha^{\vee}.
\end{eqnarray}
It is enough to show that for every $\underline{v}\in\N^M$ and $f,g\in\Ob_U$ the section $\alpha\otimes \varpi^{k|\underline{v}|}f\underline{\partial}^{<\underline{v}>}\otimes g\alpha^{\vee}$ belongs to the right side of (\ref{B1}). In fact, from the first part of the proof
\begin{eqnarray*}
\alpha\otimes\varpi^{k|\underline{v}|}f\underline{\partial}^{<\underline{v}>}\otimes g\alpha^{\vee}
=\alpha\otimes \varpi^{k|\underline{v}|}f\underline{\partial}^{<\underline{v}>}g\otimes \alpha^{\vee}= \displaystyle\sum_{\underline{v}'+\underline{v}''=\underline{v}}\varpi^{k|\underline{v}|}f\stirlingii{\underline{v}}{\underline{v}'}\underline{\partial}^{<\underline{v}'>}(g)\alpha\otimes\underline{\partial}^{<\underline{v}''>}\otimes \alpha^{\vee}
\end{eqnarray*}
and we get the relation (\ref{B1}). Let us define $\theta: \Db^{(m,k)}_{U}(\lambda)\rightarrow \Db^{(m,k)}_{U}$ by $\theta\left(\varpi^{k|\underline{v}|}f\alpha\otimes\underline{\partial}^{<\underline{v}>}\otimes\alpha^{\vee}\right)= \varpi^{k|\underline{v}|}f\underline{\partial}^{<\underline{v}>}$ and let us see that $\theta$ is a homomorphism of rings (the multiplication on the left is given by (\ref{mt})). By (\ref{B1}), the elements in $\Db^{(m,k)}_{U}(\lambda)$ are linear combinations of the terms $\varpi^{k|\underline{v}|}f\alpha\otimes\underline{\partial}^{<\underline{v}>}\otimes\alpha^{\vee}$ and therefore, it is enough to show that $\theta$ preserves the multiplicative structure over the elements of this form. So, let us take $\underline{v},\underline{u}\in\N$ and $f,g\in \Ob_U$. On the one hand
\begin{align*}
\theta(\varpi^{k|\underline{v}|}f\alpha \otimes \underline{\partial}^{<\underline{v}>}\otimes\alpha^{\vee}\bullet \varpi^{k|\underline{u}|}g\alpha\otimes\underline{\partial}^{<\underline{u}>}\otimes\alpha^{\vee})
& =  \theta(\varpi^{k|\underline{v}|}f\alpha\otimes\underline{\partial}^{<\underline{v}>}\varpi^{k|\underline{u}|}g\underline{\partial}^{<\underline{u}>}\otimes\alpha^{\vee})\\
& = \displaystyle\sum_{\underline{v}'+\underline{v}''=\underline{v}}\varpi^{k|\underline{v}|}f\stirlingii{\underline{v}}{\underline{v}'}\underline{\partial}^{<\underline{v}'>}(\varpi^{k|\underline{u}|}g)\underline{\partial}^{<\underline{v}''>}\underline{\partial}^{<\underline{u}>},
\end{align*}
and on the other hand
\begin{align*}
\theta(\varpi^{k|\underline{v}|}f\alpha \otimes \underline{\partial}^{<\underline{v}>}\otimes\alpha^{\vee})\bullet\theta(\varpi^{k|\underline{u}|}g\alpha \otimes \underline{\partial}^{<\underline{u}>}\otimes\alpha^{\vee})
&=\varpi^{k|\underline{v}|}f\underline{\partial}^{<\underline{v}>}\bullet\varpi^{k|\underline{u}|}g\partial^{<\underline{u}>}\\
&= \displaystyle\sum_{\underline{v}'+\underline{v}''=\underline{v}}\varpi^{k|\underline{v}|}f\stirlingii{\underline{v}}{\underline{v}'}\underline{\partial}^{<\underline{v}'>}(\varpi^{k|\underline{u}|}g)\underline{\partial}^{<\underline{v}''>}\underline{\partial}^{<\underline{u}>}.
\end{align*}
Both equations show that $\theta$ is a ring homomorphism. \\
Finally, a reasoning completely analogous shows that the morphism $\theta^{-1}:\Db^{(m,k)}_{U}\rightarrow\Db^{(m,k)}_{U}(\lambda)$ defined by
\begin{eqnarray*}
\theta^{-1}(\varpi^{k|\underline{v}|}f\underline{\partial}^{<\underline{v}>})=\varpi^{k|\underline{v}|}f\alpha\otimes\underline{\partial}^{<\underline{v}>}\otimes\alpha^{\vee}
\end{eqnarray*}
it is also a homomorphism of rings and $\theta\circ\theta^{-1}=\theta^{-1}\circ\theta=id$. This ends the proof of the lemma.
\end{proof}

\begin{cong_sub_groups}\label{cong_sub_groups}\textbf{Congruence subgroups and wide open congruence subgroups}
\end{cong_sub_groups}
\justify
Let us denote by $\mathbb{F}_q=\mathfrak{o}/(\varpi)$ the residue field of $\mathfrak{o}$, and let us consider $\G_L := \G\times_{\text{Spec}(\mathfrak{o})}\text{Spec}(L)$ the generic fiber of $\G$ and $\G_{\mathbb{F}_q}:=\G\times_{\text{Spec}(\mathfrak{o})}\text{Spec}(\mathbb{F}_q)$ the special fiber. For every $k\in\N$, there exists a smooth model $\G(k)$ of $\G$ such that $\text{Lie}(\G(k))=\varpi^{k}\mathfrak{g}$. In fact, we take $\G(0):=\G$ and we construct $\G(1)$ as the dilatation of the trivial subgroup of $\G_{\mathbb{F}_q}$ in $\G$ \cite[Chapter 3, Section 3.2]{BLR}. This is a flat $\mathfrak{o}$-scheme which is an integral model of  $\G_L$ \cite[Proposition 1.1]{WW}. In general, we let $\G(k+1)$ be the dilatation of the trivial subgroup of $\G(k)_{\mathbb{F}_q}$ in $\G(k)$, in such a way that for every $k\in\N$ we dispose of a canonical morphism $\G(k+1)\rightarrow \G(k)$.
\justify
We end this briefly discussion about the congruence subgroups with the following description of the distribution algebra $D^{(m)}(\G(k))$.\footnote{This is exactly as in \cite[3.3.2]{HPSS}}. Let us take a triangular decomposition $\mathfrak{g}=\mathfrak{n}\oplus \mathfrak{t}\oplus\overline{\mathfrak{n}}$ and let us consider basis $(f_i)$, $(h_j)$ and $(e_l)$ of the $\mathfrak{o}$-Lie algebras $\mathfrak{n}$, $\mathfrak{t}$ and $\overline{\mathfrak{n}}$, respectively. Then $D^{(m)}(k)$ equals the $\mathfrak{o}$-subalgebra of $\Ub(\mathfrak{g})\otimes_{\mathfrak{o}}L$ generated as an $\mathfrak{o}$-module by the elements 
\begin{eqnarray}\label{Description_D(Gk)}
q_{\underline{v}}!\varpi^{k|\underline{v}|}\frac{\underline{f}^{\underline{v}}}{\underline{v}!}q_{\underline{v}'}!\varpi^{k|\underline{v}'|}{\underline{h}\choose \underline{v}'}q_{\underline{v}''}!\varpi^{k|\underline{v}''|}\frac{\underline{e}^{\underline{v}''}}{\underline{v}''!}.
\end{eqnarray}
An element of the preceding form has order $d=|v|+|v'|+|v''|$ and the $\mathfrak{o}$-spam of elements of order  less or equal than $d$ forms an $\mathfrak{o}$-submodule $D^{(m)}_d(\G(k))\subset D^{(m)}(\G(k))$. In this way $D^{(m)}(\G(k))$ becomes a filtered $\mathfrak{o}$-algebra, such that by (\ref{Description_D(Gk)}) and the well known Poincaré–Birkhoff–Witt theorem we have $D^{(m)}(\G(k))\otimes_{\mathfrak{o}}L = \Ub(\mathfrak{g})\otimes_{\mathfrak{o}}L$.
\justify
The preceding discussion also tells us that 
\begin{eqnarray*}
D^{(m)}(\G(0))_{\varpi^k} = D^{(m)}(\G(k)).
\end{eqnarray*}
\justify
Finally, let us introduce a family of certain rigid-analytic "wide-open" groups $\G(k)^{\circ}$, which will be important in our work. To do this, let us first consider  the formal completion $\mathfrak{G}(k)$ of the group scheme $\G(k)$ along its special fiber, which is a formal scheme of topological finite type over $\text{Spf}(\mathfrak{o})$. Now, we consider $\widehat{\mathfrak{G}}(k)^{\circ}$ be the completion of $\mathfrak{G}(k)^{\circ}$ along its unit section $\text{Spf}(\mathfrak{o})\rightarrow \mathfrak{G}(k)$, and we denote by $\G(k)^{\circ}$ its associated rigid-analytic space \cite[(0.2.6)]{Berthelot3}, which is a rigid-analytic group.
\justify
We recall for the reader  that in subsection \ref{Integral_models} we have introduced the sheaves $\Ab^{(m,k)}_{X}:=\Ob_{X}\otimes_{\mathfrak{o}}D^{(m)}(\G(k))$, which carries a structure of filtered $\Ob_X$-ring, such that $gr(\Ab^{(m,k)}_X)=\Ob_X\otimes_{\mathfrak{o}}\text{Sym}^{(m)}(\varpi^k\mathfrak{g})$.

\begin{prop}\label{canonical_lambda}
There exists a canonical surjective homomorphism of sheaves of filtered $\mathfrak{o}$-algebras
\begin{eqnarray*}
\Phi^{(m,k)}_X:\Ab^{(m,k)}_{X}\rightarrow \Db^{(m,k)}_X(\lambda).
\end{eqnarray*}
\end{prop}
\begin{proof}
Let us star by showing the existence of such a morphism. By \cite[Corollary 4.5.2]{HS1}, there exists a morphism of sheaves of filtered $\mathfrak{o}$-algebras 
\begin{eqnarray}\label{fil}
\Ab_{X}^{(m,0)}\rightarrow \Db^{(m,0)}_{X}(\lambda).
\end{eqnarray} 
Let's first show that after specialising in $\varpi^k$ the Rees ring associated to the twisted order filtration of $\Db^{(m,0)}_{X_0,\lambda}$ we get $\Db^{(m,k)}_{X}(\lambda)$. To do that, we consider  $\Db^{(m,0)}_{X}$ filtered by the order of differential operators and we define the following homomorphisms of $\Ob_{X}$-modules
\begin{eqnarray}\label{rees}
\begin{tikzcd}[column sep=4em]
 \Lb(\lambda)\otimes_{\Ob_{X}}R\left(\Db^{(m,0)}_{X}\right)\otimes_{\Ob_{X}}\Lb(\lambda)^{\vee} \arrow[r, "\theta", shift left=1.5ex]
& R\left(\Db^{(m,0)}_{X}(\lambda)\right), \arrow[l, "\theta^{-1}"]
\end{tikzcd}
\end{eqnarray}
by
\begin{eqnarray*}
\theta\left(\alpha\otimes\sum_i P_i t^i\otimes\alpha^{\vee}\right)=\sum_i(\alpha\otimes P_i\otimes\alpha^{\vee})t^i
\hspace{0.4 cm}
\text{and}
\hspace{0.4 cm}
	\theta^{-1}\left(\sum_j(\alpha_j\otimes P_j\otimes\alpha_j^{\vee})t^j\right)=\sum_j\alpha_j\otimes P_j t^j\otimes\alpha_j^{\vee}
\end{eqnarray*} 
with ord$(P_i)=i$ for every $i$ in the sum (resp. ord$(P_j)=j$ for every $j$). It's clear that $\theta\circ\theta^{-1}=\theta^{-1}\circ\theta=id$ and therefore (\ref{rees}) is an isomorphism of $\Ob_{X}$-modules. We remark that an easy calculation shows that (\ref{rees}) is in fact an isomorphism of rings.\\
Let's denote by $\sigma_1: R\left(\Db^{(m,0)}_{X}(\lambda)\right)\rightarrow \Db^{(m,k)}_{X}(\lambda)$; $t\mapsto \varpi^{k}$ and by $\sigma_2:R\left(\Db^{(m,0)}_{X}\right)\rightarrow \Db^{(m,k)}_{X}$; $t\mapsto \varpi^{k}$, and let's consider the following diagrams 
\begin{eqnarray*}
	\begin{tikzcd}[column sep=7em]
 \Lb(\lambda)\otimes_{\Ob_{X}}R\left(\Db^{(m,0)}_{X}\right)\otimes_{\Ob_{X}}\Lb(\lambda)^{\vee} \arrow[rd, " id_{\Lb(\lambda)}\otimes\sigma_2\otimes id_{\Lb(\lambda)^{\vee}}"] \arrow[r,"\theta",shift left=1.5ex] & 
R\left(\Db^{(m,0)}_{X}(\lambda)\right) \arrow[l,"\theta^{-1}"] \arrow[d, "\sigma_1"]\\
& \Db^{(m,k)}_{X}(\lambda)
\end{tikzcd}
\end{eqnarray*}
It is straightforward to check that both diagrams are commutative and we can conclude that
\begin{eqnarray*}
	\left(\Db^{(m,0)}_{X_0}(\lambda)\right)_{\varpi^{k}}= \text{Im}(\sigma_1)=\text{Im}(id_{\Lb(\lambda)}\otimes\sigma_2\otimes id_{\Lb(\lambda)^{\vee}})
	= \Lb(\lambda)\otimes_{\Ob_{X}}\text{Im}(\sigma_2)\otimes_{\Ob_{X}}\Lb(\lambda)^{\vee}
 = \Db^{(m,k)}_{X}(\lambda).
\end{eqnarray*} 
On the other hand, taking the natural filtration of $\Ab^{(m,0)}_{X}$ we have
\begin{eqnarray*}
	R\left(\Ab_{X}^{(m,0)}\right)=\Ob_{X}\otimes_{\mathfrak{o}}R\left(D^{(m)}(\G(0))\right)
\end{eqnarray*} 
and therefore $(\Ab_{X}^{(m,0)})_{\varpi^k}= \Ab^{(m,k)}_{X}$. The above two calculations tell us that passing to the Rees rings and specialising in $\varpi^{k}$ the map (\ref{fil}), we get the desired homomorphism of filtered sheaves of $\mathfrak{o}$-algebras
\begin{eqnarray}\label{morph_with_cong_level_alg}
\Phi^{(m,k)}_{X} : \Ab^{(m,k)}_{X}\rightarrow \Db^{(m,k)}_{X}(\lambda).
\end{eqnarray} 
Let us finally show that this morphism is surjective. To do that, let us recall that the right $\G$-action on $X$ induces a canonical application 
\begin{eqnarray}\label{sur_map-H}
\Ob_{X}\otimes_{\mathfrak{o}}\varpi^k\mathfrak{g}\rightarrow \varpi^k\Tb_{X}
\end{eqnarray}
which is surjective by \cite[Subsection 1.6]{Huyghe2}. By using proposition \ref{graded_lambda} and $gr(\Ab^{(m,k)}_{X})=\Ob_{X}\otimes_{\mathfrak{o}}\text{Sym}^{(m)}(\mathfrak{g})$, we can conclude that $\Phi^{(m,k)}_X$ is surjective.
\end{proof}
\justify
Proposition \ref{algebraic_local_desc} and the same reasoning given in \cite[Proposition 2.2.2 (iii)]{HSS} imply the following meaningful result.\footnote{Of course, this is also an immediately consequence of proposition \ref{canonical_lambda} and \cite[proposition 1.3.6]{Huyghe1}.}

\begin{prop}\label{Noetherian_sections}
The sheaf $\Db^{(m,k)}_{X}(\lambda)$ is a sheaf of $\Ob_{X}$-rings with noetherian sections over all open affine subsets of $X$. 
\end{prop}

\subsection{Finiteness properties}
\justify
Throeught this section $\lambda\in X(\T)$ will denote an algebraic character. By abuse of notation, we will denote again by $\lambda$ the character $d\lambda \in \text{Hom}_{\mathfrak{o}-\text{mod}}(\mathfrak{t},\mathfrak{o})$ induced via differentiation. In this subsection we will show one important property about the $p$-torsion of the cohomology groups of coherent $\Db^{(m,k)}_{X,\lambda}$-modules, when the character $\lambda + \rho \in \text{Hom}_{L-\text{mod}}(\mathfrak{t}_L,L)$ is dominant and regular. We will follow the arguments of  \cite{Huyghe2}.
\justify
Let $Y$ be a projective scheme. There exists a very ample sheaf $\Ob(1)$ on $Y$ \cite[Chapter II, remark 5.16.1]{Hartshorne1}. Therefore, for any arbitrary $\Ob_{Y}$-module $\Eb$ we can consider the twist
\begin{eqnarray*}
\Eb(r) := \Eb\otimes_{\Ob_{X}}\Ob(r)
\end{eqnarray*} 
where $r\in\Z$ means the $r$-th tensor product of $\Ob(1)$  with itself. We recall to the reader that there exists $r_0\in\Z$, depending of $\Ob(1)$, such that for every $k\in\Z_{>0}$ and for every $s\le r_0$, $H^{k}(Y,\Ob(s))=0$ \cite[Chapter II, theorem 5.2 (b)]{Hartshorne1}.
\justify
Let us start the results of this subsection with the following proposition which states three important properties  of coherent $\Ab^{(m,k)}_{Y}$-modules \cite[Proposition A.2.6.1]{HS1}. This is a key result in our work. Let $\Eb$ be a coherent $\Ab^{(m,k)}_{Y}$-module.

\begin{prop}\label{Result_HS}
\begin{itemize}
\item[(i)] $H^{0}(X,\Ab^{(m,k)}_Y)=D^{(m)}(\G(k))$ is a noetherian $\mathfrak{o}$-algebra.
\item[(ii)] There exists a surjection on $\Ab^{(m,k)}_Y$-modules $\left(\Ab^{(m,k)}_{Y}(-r)\right)^{\oplus a}\rightarrow \Eb \rightarrow 0$ for suitable $r\in\Z$ and $a\in\N$.
\item[(iii)] For any $k\ge 0$ the group $H^{k}(X,\Eb)$ is a finitely generated $D^{(m)}(\G(k))$-module. 
\end{itemize}
\end{prop}
\justify
Inspired in proposition \ref{canonical_lambda}, in a first time we will be concentrated in coherent $\Ab^{(m,k)}_{Y}$-modules. The next two results will play an important role when we consider formal completions. 

\begin{lem}\label{Coherent_A(m,k)}
For every coherent $\Ab^{(m,k)}_{Y}$-module $\Eb$, there exists $r=r(\Eb)\in\Z$ such that $H^{k}(X,\Eb(s))=0$ for every $s\ge r$.
\end{lem} 
\begin{proof}
Let us fix $r_0\in\Z$ such that $H^{k}(Y,\Ob(s))=0$ for every $k>0$ and $s\ge r_0$.  We have
\begin{eqnarray*}
H^{k}(Y,\Ab^{(m,k)}_Y(s))= H^{k}(Y,\Ob(s))\otimes_{\mathfrak{o}}D^{(m)}(\G(k)) = 0.
\end{eqnarray*}
The rest of the proof follows the same inductive argument given in \cite[Proposition 2.2.1]{Huyghe2}.
\end{proof}
\justify
Let us suppose now that $X:=\G/\B$ is the smooth flag $\mathfrak{o}$-scheme of $\G$. From proposition \ref{canonical_lambda} and lemma \ref{Coherent_A(m,k)} we have the following result.

\begin{lem}\label{Integral_resolution}
For every coherent $\Db^{(m,k)}_{X,\lambda}$-module $\Eb$, there exist $r=r(\Eb)\in\Z$, a natural number $a\in\N$ and an epimorphism of $\Db^{(m,k)}_{X,\lambda}$-modules
\begin{eqnarray*}
\left(\Db^{(m,k)}_{X,\lambda}(-r)\right)^{\oplus a}\rightarrow\Eb\rightarrow 0.
\end{eqnarray*}
\end{lem}

\begin{prop}\label{bounded_coh_Int}
Let us suppose that $\lambda + \rho \in \mathfrak{t}^*_L$ is a dominant and regular character (cf. \ref{roots}).
\begin{itemize}
\item[(i)] Let us fix $r\in\Z$. For every positive integer $k\in\Z_{>0}$, the cohomology group $H^{k}(X,\Db^{(m,k)}_{X,\lambda}(r))$ has bounded $p$-torsion.
\item[(ii)] For every coherent $\Db^{(m,k)}_{X,\lambda}$-module $\Eb$, the cohomology group $H^k(X,\Eb)$ has bounded $p$-torsion for all $k>0$.
\end{itemize} 
\end{prop}
\begin{proof}
To show $(i)$ we remark for the reader that by construction $\Db^{(m,k)}_{X,\lambda,\Q}=\Db_{\lambda}$ is the usual sheaf of twisted differential operators on the flag variety $X_L$ \cite[Part I, 5.17]{Jantzen}. As $\Db^{(m,k)}_{X,\lambda,\Q}(r)$ is a coherent $\Db_{\lambda}$-module, the classical Beilinson-Bernstein theorem \cite{BB} allows us to conclude that $H^k(X,\Db^{(m,k)}_{X,\lambda}(r))\otimes_{\mathfrak{o}}L=0$ for every positive integer $k\in\Z_{>0}$. This in particular implies that the sheaf $\Db^{(m,k)}_{X,\lambda}(r)$ has $p$-torsion cohomology groups $H^k(X,\Db^{(m,k)}_{X,\lambda}(r))$, for every $k>0$ and $r\in\Z$. Now, by proposition \ref{canonical_lambda}, we know that $\Db^{(m,k)}_{X,\lambda}(r)$ is in particular a coherent $\Ab_X^{(m,k)}$-module and hence, by the third part of  proposition \ref{Result_HS} we get that for every $k\ge 0$ the cohomology groups $H^k(X,\Db^{(m)}_{X,\lambda}(r))$ are finitely generated $D^{(m)}(\G(k))$-modules. Consequently, of finite $p$-torsion for every  integer $0<k\le \text{dim}(X)$ and $r\in\Z$.
\justify
By lemma \ref{Integral_resolution} we can use the same inductive reasoning that in \cite[Corollary 2.2.4]{Huyghe2} to show (ii).
\end{proof}

\subsection{Passing to formal completions}
\justify
We recall for the reader that throughout this work 
\begin{eqnarray*}
\mathfrak{X}:= \varinjlim_{j\in\N} X_j,
\hspace{0.7 cm}
X_j := X\times_{\text{Spec}(\mathfrak{o})}\text{Spec}(\mathfrak{o}/\varpi^{j+1})
\end{eqnarray*}
denote the formal completion of $X$ along its special fiber. 

\begin{defi}
We will denote by 
\begin{eqnarray*}
\widehat{\Da}^{(m,k)}_{\mathfrak{X},\lambda}:= \varprojlim_{j\in\N} \Db^{(m,k)}_{X,\lambda}/\varpi^{j+1}\Db^{(m,k)}_{X,\lambda}
\end{eqnarray*}
the $\varpi$-adic completion of $\Db^{(m,k)}_{X,\lambda}$ and we consider it as a sheaf on $\mathfrak{X}$. Following the notation given at the beginning of this work, the sheaf $\widehat{\Da}^{(m,k)}_{\mathfrak{X},\lambda,\Q}$ will denote our sheaf of level $m$ twisted differential operators with congruence level $k$ on the formal flag scheme $\mathfrak{X}$.  
\end{defi}

\begin{prop}\label{locally_hatD}
\begin{itemize}
\item[(i)] There exists a basis $\Ba$ of the topology of $\mathfrak{X}$, consisting of open affine subset, such that for every $\mathfrak{U}\in\Ba$ the ring $\widehat{\Da}^{(m,k)}_{\mathfrak{X},\lambda}(\mathfrak{U})$ is two-sided noetherian.
\item[(ii)] The sheaf of rings $\widehat{\Da}^{(m,k)}_{\mathfrak{X},\lambda,\Q}$ is coherent.
\end{itemize}
\end{prop}
\begin{proof}
To show (i) we can take an open affine subset $U\in\Sb$ and to consider $\mathfrak{U}$ its formal completion along the special fiber. We have
\begin{eqnarray*}
H^{0}(\mathfrak{U},\widehat{\Da}^{(m,k)}_{ \mathfrak{X},\lambda})
\simeq \widehat{H^{0}(U,\Db^{(m,k)}_{X,\lambda})}
\simeq \widehat{H^{0}(U,\Db^{(m,k)}_{X})}
\simeq H^{0}(\mathfrak{U},\widehat{\Da}^{(m,k)}_{\mathfrak{X}})
\end{eqnarray*}
The first and third isomorphisms are given by \cite[($0_I$, 3.2.6)]{Grothendieck1} and the second one arises from proposition \ref{algebraic_local_desc}. By \cite[Proposition 2.2.2 (v)]{HSS}  $H^{0}(\mathfrak{U},\widehat{\Da}^{(m,k)}_{\mathfrak{X}})$ is twosided noetherian. Therefore, we can take $\Ba$ as the set of affine open subsets of $\mathfrak{X}$ contained in the $\varpi$-adic completion of an affine open subset $U\in\Sb$. This proves (i). By \cite[proposition 3.3.4]{Berthelot1} we can conclude that (ii) is an immediately consequence of (i) because $H^{0}(\mathfrak{U},\widehat{\Da}_{\mathfrak{X},\lambda,\Q}^{(m,k)})=H^{0}(\mathfrak{U},\widehat{\Da}^{(m,k)}_{\mathfrak{X},\lambda})\otimes_{\mathfrak{o}}L$ \cite[(3.4.0.1)]{Berthelot1}.
\end{proof}
\justify
The objective of this subsection is to prove an analogue of proposition \ref{bounded_coh_Int} for coherent $\widehat{\Da}_{\mathfrak{X},\lambda,\Q}^{(m,k)}$-modules and to conclude that $H^{0}(\mathfrak{X},\bullet)$ is an exact functor over the category of coherent $\widehat{\Da}_{\mathfrak{X},\lambda,\Q}^{(m,k)}$-modules.

\begin{prop}\label{Properties_completion}
Let $\Eb$ be a coherent $\Db^{(m,k)}_{X,\lambda}$-module and $\widehat{\Eb}$ its $\varpi$-adic completion, which we consider as a sheaf on $\mathfrak{X}$.
\begin{itemize}
\item[(i)] For all $k\le 0$ one has $H^{k}(\mathfrak{X},\widehat{\Eb})=\varprojlim_{j} H^{k}(X,\Eb/\varpi^j\Eb)$.
\item[(ii)] For all $k>0$ one has $H^k(\mathfrak{X},\widehat{\Eb}) = H^k(X,\Eb)$.
\item[(iii)] The global section functor $H^{0}(\mathfrak{X},\bullet)$ satisfies $H^{0}(\mathfrak{X},\widehat{\Eb})=\varprojlim_j H^0(X,\Eb)/\varpi^j H^0(X,\Eb)$.
\end{itemize}
\end{prop}
\begin{proof}
Let $\Eb_t$ denote the torsion subpresheaf of $\Eb$. As $X$ is a noetherian space and $\Db^{(m,k)}_{X,\lambda}$ has noetherian rings sections over open affine subsets of $X$ (proposition \ref{Noetherian_sections}), we can conclude that $\Eb_t$ is in fact a coherent $\Db^{(m,k)}_{X,\lambda}$-module. This is generated by a coherent $\Ob_X$-module which is annihilated by a power $\varpi^c$ of $\varpi$, and so is $\Eb_t$. The quotient $\Gb:=\Eb/\Eb_t$ is again a coherent $\Db^{(m,k)}_{X,\lambda}$-module and therefore we can assume, after possibly replacing $c$ by a larger number, that $\varpi^c\Eb_t=0$ and $\varpi^cH^k(X,\Eb)=\varpi^cH^k(X,\Gb)=0$ , for all $k>0$. From here on the proof of the proposition follows the same lines of reasoning that in \cite[proposition 3.2]{Huyghe1}.
\end{proof}
\justify
Next proposition is a natural result from lemmas \ref{Coherent_A(m,k)} and \ref{Integral_resolution}. The proof is exactly the same that in \cite[Proposition 4.2.2]{HPSS}.\footnote{We skip the proof here, but the reader can take a look to \cite[Proposition 4.1.2]{Sarrazola} when we have treated the case $k=0$. The proof for $k\in\Z_{>0}$ is exactly the same.}

\begin{prop}\label{Resolution_D}
Let $\Ea$ be a coherent $\widehat{\Da}^{(m,k)}_{\mathfrak{X},\lambda}$-module.
\begin{itemize}
\item[(i)] There exists $r_2=r_2(\Ea)\in\Z$ such that for all $r\ge r_2$ there is $a\in\Z$ and an epimorphism of $\widehat{\Da}^{(m,k)}_{\mathfrak{X},\lambda}$-modules
\begin{eqnarray*}
\left(\widehat{\Da}^{(m,k)}_{\mathfrak{X},\lambda}(-r)\right)^{\oplus a}\rightarrow\Ea\rightarrow 0
\end{eqnarray*} 
\item[(ii)] There exists $r_3 = r_3(\Ea)\in\Z$ such that for all $r\ge r_3$ we have $H^{k}(\mathfrak{X},\Ea)=0$, for all $k>0$. 
\end{itemize}
\end{prop}
\justify
The same inductive argument exhibited in proposition \cite[Proposition 3.4 (i)]{Huyghe1} shows

\begin{coro}\label{coro 3.1.3}
Let $\Ea$ be a coherent $\widehat{\Da}^{(m,k)}_{\mathfrak{X},\lambda}$-module. There exists $c=c(\Ea)\in\N$ such that for all $k>0$ the cohomology group $H^k(\mathfrak{X},\Ea)$ is annihilated by $\varpi^c$.
\end{coro}
\justify
Now, we want to extend the part (i) of the preceding proposition to the sheaves $\widehat{\Da}^{(m,k)}_{\mathfrak{X},\lambda,\Q}$. To do that, we need to show that the category of coherent $\widehat{\Da}^{(m,k)}_{\mathfrak{X},\lambda,\Q}$-modules admits integral models (definition \ref{defi I.models}).\\
Let $\text{Coh}(\widehat{\Da}^{(m,k)}_{\mathfrak{X},\lambda})$ be the category of coherent $\widehat{\Da}^{(m,k)}_{\mathfrak{X},\lambda}$-modules and let $\text{Coh}(\widehat{\Da}^{(m,k)}_{\mathfrak{X},\lambda})_\Q$ be the category of coherent $\widehat{\Da}^{(m,k)}_{\mathfrak{X},\lambda}$-modules  up to isogeny. This means that $\text{Coh}(\widehat{\Da}^{(m,k)}_{\mathfrak{X},\lambda})_\Q$ has the same class of objects as $\text{Coh}(\widehat{\Da}^{(m,k)}_{\mathfrak{X},\lambda})$ and, for any two objects $\Mb$ and $\Nb$ in Coh$(\widehat{\Da}^{(m,k)}_{\mathfrak{X},\lambda})_{\Q}$ one has
\begin{eqnarray*}
\text{Hom}_{\text{Coh}(\widehat{\Da}^{(m,k)}_{\mathfrak{X},\lambda})_{\Q}}(\Mb,\Nb)=\text{Hom}_{\text{Coh}(\widehat{\Da}^{(m,k)}_{\mathfrak{X},\lambda})}(\Mb,\Nb)\otimes_{\mathfrak{o}}L.
\end{eqnarray*}

\begin{prop}\label{prop 3.1.4}
The functor $\Mb\mapsto\Mb\otimes_{\mathfrak{o}}L$ induces an equivalence of categories between $\text{Coh}(\widehat{\Da}^{(m,k)}_{\mathfrak{X},\lambda})_\Q$ and $\text{Coh}(\widehat{\Da}^{(m,k)}_{\mathfrak{X},\lambda,\Q})$.
\end{prop}
\begin{proof}
By definition, the sheaf $\widehat{\Da}^{(m,k)}_{\mathfrak{X},\lambda,\Q}$ satisfies \cite[conditions 3.4.1]{Berthelot1} and therefore \cite[proposition 3.4.5]{Berthelot1} allows to conclude the proposition.
\end{proof}
\justify
The proof of the next theorem follows exactly the same lines than in \cite[theorem 4.2.8]{HPSS}. 

\begin{theo}\label{Resolution_D_Q}
Let $\Ea$ be a coherent $\widehat{\Da}^{(m,k)}_{\mathfrak{X},\lambda,\Q}$-module. 
\begin{itemize}
\item[(i)] There is $r(\Ea)\in\Z$ such that, for every $r\ge r(\Ea)$ there exists $a\in\N$ and an  epimorphism of $\widehat{\Da}^{(m,k)}_{\mathfrak{X},\lambda,\Q}$-modules
\begin{eqnarray*}
\left(\widehat{\Da}^{(m,k)}_{\mathfrak{X},\lambda,\Q}(-r)\right)^{\oplus a}\rightarrow\Ea\rightarrow 0.
\end{eqnarray*}
\item[(ii)] For all $i>0$ one has $H^i(\mathfrak{X},\Ea)=0$.
\end{itemize}
\end{theo}
\begin{proof}
By the preceding proposition, there exists a coherent $\widehat{\Da}^{(m,k)}_{\mathfrak{X},\lambda}$-module $\Fa$ such that $\Fa\otimes_{\mathfrak{o}}L\simeq\Ea$. Therefore, applying proposition \ref{Resolution_D} to $\Fa$ gives $(i)$. Moreover,  as $\mathfrak{X}$ is a noetherian space, corollary \ref{coro 3.1.3} allows us to conclude that
\begin{eqnarray*}
H^i(\mathfrak{X},\Ea)=H^i(\mathfrak{X},\Fa)\otimes_{\mathfrak{o}}L=0
\end{eqnarray*}
for every $k>0$ \cite[(3.4.0.1)]{Berthelot1}.
\end{proof}

\subsection{The arithmetic Beilinson-Bernstein theorem with congruence level}

\subsubsection{Calculation of global sections}
We recall for the reader that throughout this section $\lambda+\rho\in\mathfrak{t}^*_L$ denotes a dominant and regular character, which is induced by an algebraic character $\lambda\in X(\T)$. Insppired in the arguments exhibited in \cite{HS2}, in this subsection we propose to calculate the global sections of the sheaf $\widehat{\Da}^{(m,k)}_{\mathfrak{X},\lambda,\Q}$. 
\justify
Let us identify the universal enveloping algebra $\Ub(\mathfrak{t}_L)$ of the Cartan subalgebra $\mathfrak{t}_L$ with the symmetric algebra $S(\mathfrak{t}_L)$, and let $Z(\mathfrak{g}_L)$ denote the center of the universal enveloping algebra $\Ub(\mathfrak{g}_L)$ of $\mathfrak{g}_L$. The classical Harish-Chandra isomorphism $Z(\mathfrak{g}_L)\simeq S(\mathfrak{t}_L)^{W}$ (the subalgebra of Weyl invariants) \cite[theorem 7.4.5]{Dixmier}, allows us to define for every linear form $\lambda\in \mathfrak{t}_L^*$ a central character \cite[7.4.6]{Dixmier} 
\begin{eqnarray*}
\chi_\lambda : Z(\mathfrak{g}_L)\rightarrow L
\end{eqnarray*}   
which induces the central reduction $\Ub(\mathfrak{g}_L)_{\lambda}:=\Ub(\mathfrak{g}_L)\otimes_{Z(\mathfrak{g}_L),\chi_\lambda+\rho}L$. If Ker$(\chi_{\lambda+\rho})_{\mathfrak{o}}:=D^{(m)}(\G(k))\cap \text{Ker}(\chi_{\lambda+\rho})$, we can consider the central redaction
\begin{eqnarray*}
D^{(m)}(\G(k))_{\lambda}:=D^{(m)}(\G(k))/D^{(m)}(\G(k))\text{Ker}(\chi_{\lambda+\rho})_{\mathfrak{o}}
\end{eqnarray*}
and its $\varpi$-adic completion $\widehat{D}^{(m)}(\G(k))_{\lambda}$. It is clear that $D^{(m)}(\G(k))_{\lambda}$ is an integral model of $\Ub(\mathfrak{g}_L)_{\lambda}$. We will denote by $D^{\dag}(\G(k))$ the limit of the inductive system $\widehat{D}^{(m)}(\G(k))\otimes_{\mathfrak{o}}L \rightarrow \widehat{D}^{(m+1)}(\G(k))\otimes_{\mathfrak{o}}L$.

\begin{theo} \label{C.G sections}
The homomorphism of $\mathfrak{o}$-algebras $\Phi^{(m,k)}_{\lambda}: D^{(m)}(\G(k))\rightarrow  H^{0}(X,\Db^{(m,k)}_{X,\lambda})$, defined by taking global sections in (\ref{canonical_lambda}), induces an isomorphism of $\mathfrak{o}$-algebras
\begin{eqnarray*}
\widehat{D}^{(m)}(\G(k))_{\lambda}\otimes_{\mathfrak{o}}L \xrightarrow{\simeq} H^0\left(\mathfrak{X},\widehat{\Da}^{(m,k)}_{\mathfrak{X},\lambda,\Q}\right).
\end{eqnarray*}
\end{theo}
\begin{proof}
The key in the proof of the theorem is the following commutative diagram
\begin{eqnarray*}
\begin{tikzcd}
D^{(m)}(\G(k)) \arrow[r, "\Phi^{(m,k)}_{\lambda}"] \arrow[d,hook ]
& H^0(X,\Db^{(m,k)}_{X,\lambda}) \arrow[d,hook ] \\
\Ub(\mathfrak{g}_L) \arrow[r, "\Phi_{\lambda}"]
& H^{0}(X_L,\Db_{\lambda}).
\end{tikzcd}
\end{eqnarray*}
Here $\Phi_{\lambda}$ is the morphism in \cite[(11.2.2)]{HTT}\footnote{We recall for the reader  that $\Lb(\lambda)$ is a $\G$-equivariant line bundle, which implies the existence of this morphism \cite[Section 11.1]{HTT}.}. By the classical Beilinson-Bernsein theorem \cite{BB} and the preceding commutative diagram, we have that $\Phi^{(m,k)}_\lambda$ factors through the morphism $\overline{\Phi}^{(m,k)}_\lambda: D^{(m)}(\G(k))_{\lambda}\rightarrow H^{0}(X,\Db^{(m,k)}_{X,\lambda})$ which becomes an isomorphism after tensoring with $L$. By \cite[Lemma 3.3]{HS2} we have that $\overline{\Phi}^{(m)}_\lambda$ gives rise to an isomorphism
\begin{eqnarray*}
\widehat{D}^{(m)}(\G(k))_{\lambda}\otimes_{\mathfrak{o}}L\xrightarrow{\simeq} \widehat{H^0(X,\Db^{(m,k)}_{X,\lambda})}\otimes_{\mathfrak{o}}L,
\end{eqnarray*}
and proposition \ref{Properties_completion} together with the fact that $\mathfrak{X}$ is in particular a noetherian topological space end the proof of the theorem.
\end{proof}

\subsubsection{The localization functor}\label{Loc_m}
In this section we will introduce the localization functor. Let $E$ be a finitely generated $\widehat{D}^{(m)}(\G(k))_{\lambda}\otimes_{\mathfrak{o}}L$-module. We define $\La oc^{(m,k)}_{\mathfrak{X},\lambda}(E)$ as the associated sheaf to the presheaf on $\mathfrak{X}$ defined by
\begin{eqnarray*}
\mathfrak{U}\mapsto \widehat{\Da}^{(m,k)}_{\mathfrak{X},\lambda,\Q}(\mathfrak{U})\otimes_{\widehat{D}^{(m)}(\G(k))_{\lambda}\otimes_{\mathfrak{o}}L}E.
\end{eqnarray*}
It is clear that $\La oc^{(m,k)}_{\mathfrak{X},\lambda}$ is a functor from the category of finitely generated $\widehat{D}^{(m)}(\G(k))_{\lambda}\otimes_{\mathfrak{o}}L$-modules to the category of coherent $\widehat{\Da}^{(m,k)}_{\mathfrak{X},\lambda,\Q}$-modules.

\subsubsection{The arithmetic Beilinson-Bernstein theorem}
\justify
We are finally ready to prove one of the principal results of this work. To start with, we will enunciate the following proposition whose proof can be founded in \cite[Proposition 4.4.1]{Sarrazola}.

\begin{prop}\label{prop 3.3.1}
Let $\Ea$ be a coherent $\widehat{\Da}^{(m,k)}_{\mathfrak{X},\lambda,\Q}$-module. Then $\Ea$ is generated by its global sections as $\widehat{\Da}^{(m)}_{\mathfrak{X},\lambda,\Q}$-module. Furthermore, every coherent $\widehat{\Da}^{(m,k)}_{\mathfrak{X},\lambda,\Q}$-module admits a resolution by finite free $\widehat{\Da}^{(m,k)}_{\mathfrak{X},\lambda,\Q}$-modules.
\end{prop}

\begin{theo}\label{Affinity}
Let us suppose that $\lambda\in X(\T)$ is an algebraic character such that $\lambda + \rho\in\mathfrak{t}_L^*$ is a dominant and regular character of $\mathfrak{t}_L$. The functors $\La oc^{(m,k)}_{\mathfrak{X},\lambda}$ and $H^{0}(\mathfrak{X},\bullet)$ are quasi-inverse equivalence of categories between the abelian categories of finitely generated $\widehat{D}^{(m)}(\G(k))_{\lambda}\otimes_{\mathfrak{o}}L$-modules and coherent  $\widehat{\Da}^{(m)}_{\mathfrak{X},\lambda,\Q}$-modules.
\end{theo}
\begin{proof}
The proof of \cite[Proposition 5.2.1]{Huyghe1}  carries over word by word.
\end{proof}

\subsection{The sheaves $\Da^{\dag}_{\mathfrak{X},k}(\lambda)$}
\justify
In this section we will study the problem of passing to the inductive limit when $m$ varies, this means
\begin{eqnarray}
\Da^{\dag}_{\mathfrak{X},k}(\lambda) := \left( \varinjlim_{m\in\N} \widehat{\Da}^{(m,k)}_{\mathfrak{X},\lambda}\right)\otimes_{\mathfrak{o}} L
\hspace{0.5 cm}
\text{and}
\hspace{0.5 cm}
D^{\dag}(\G(k))_{\lambda} := \left(\varinjlim_{m\in\N} \widehat{D}^{(m)}(\G(k))_{\lambda} \right)\otimes_{\mathfrak{o}}L.
\end{eqnarray}
\justify
As in \ref{Loc_m} let us consider the following localization functor $\La oc^{\dag}_{\mathfrak{X},k}(\lambda)$ from the category of finitely presented $D^{\dag}(\G(k))_{\lambda}$-modules to the category of coherent $\Da^{\dag}_{\mathfrak{X},k}(\lambda)$. Let $E$ be a finitely presented $D^{\dag}(\G(k))_{\lambda}$-module, then $\La oc^{\dag}_{\mathfrak{X},k}(\lambda)(E)$ denotes the associated sheaf to the presheaf on $\mathfrak{X}$ defined by
\begin{eqnarray*}
\mathfrak{U}\subseteq\mathfrak{X} \mapsto \Da^{\dag}_{\mathfrak{X},k}(\lambda)\otimes_{D^{\dag}(\G(k))_{\lambda}}E.
\end{eqnarray*} 
As before, it is clear that $\La oc^{\dag}_{\mathfrak{X},k}(\lambda)$ is a functor from the category of finitely presented $D^{\dag}(\G(k))_{\lambda}$-modules to the category of coherent $\Da^{\dag}_{\mathfrak{X},k}(\lambda)$-modules.

\begin{Analytic_algebra}\label{Analytic_algebra}\textbf{Analytic distribution algebra}
\end{Analytic_algebra}
\justify
The wide-open rigid analytic groups, defined in \ref{cong_sub_groups}, play an important role in the work developed by Emerton in \cite{Emerton1}, to treat locally analytic representations of $p$-adic groups. The analytic distribution of $\G(k)^{\circ}$ is defined to be the continuous dual space of the space of rigid-analytic functions on $\G(k)^{\circ}$. This is,
\begin{eqnarray*}
\Db^{\text{an}}(\G(k)^{\circ}):= \left(\Ob_{\G(k)^{\circ}}(\G(k)^{\circ})\right)'_b = \text{Hom}^{\text{cont}}_L\left(\Ob_{\G(k)^{\circ}}(\G(k)^{\circ}),L\right)_b,
\end{eqnarray*}
this is a topological $L$-algebra of compact type. In \cite[Proposition 5.2.1]{HS1} Huyghe-Schmidt have shown that 
\begin{eqnarray*}
D^{\dag}(\G(k)) \simeq \Db^{\text{an}}(\G(k)^{\circ}).
\end{eqnarray*}
\justify
As $\mathfrak{X}$ is a noetherian space, theorem \ref{C.G sections} and the preceding relation tell us that
\begin{eqnarray}\label{Global_sect_dag}
H^0(\mathfrak{X},\Da^{\dag}_{\mathfrak{X},k}(\lambda))=D^{\dag}(\G(k))_{\lambda} = \Db^{\text{an}}(\G(k)^{\circ})_{\lambda} := \Db^{\text{an}}(\G(k)^{\circ})\big / \Db^{\text{an}}(\G(k)^{\circ}) (\text{Ker}(\chi_{\lambda+\rho})).
\end{eqnarray}
\justify
We will concentrate our efforts to prove the following Beilinson-Bernstein theorem for the sheaves $\Da^{\dag}_{\mathfrak{X},k}(\lambda)$.

\begin{theo}\label{BB_for_dag}
Let $\lambda\in X(\T)$ be an algebraic character, such that $\lambda+\rho\in\mathfrak{t}^*_L$ is dominant and regular. The functors $\mathcal{L}oc^{\dag}_{\mathfrak{X},k}(\lambda)$ and $H^{0}(\mathfrak{X},\bullet)$ are quasi-inverse equivalence of categories between the abelian categories of finitely presented (left) $D^{\dag}(\G(k))_{\lambda}$-modules and coherent $\Da^{\dag}_{\mathfrak{X},k}(\lambda)$-modules.
\end{theo}
\justify
Let us start by recalling the following proposition \cite[Proposition 3.6.1]{Berthelot1}.
\begin{prop}\label{coh_proj_lim} Let $Y$ be a topological space, and $\{\Db_{i}\}_{i\in J}$ be a filtered inductive system of coherent sheaves of rings on $Y$, such that for any $i\le j$ the morphisms $\Db_i\rightarrow\Db_j$ are flat. Then the sheaf $\Db^{\dag}:=\varinjlim_{i\in J}\Db_{i}$ is a coherent sheaf of rings.
\end{prop}

\begin{prop}\label{coh. dag}
The sheaf of rings $\Da^{\dag}_{\mathfrak{X},k}(\lambda)$ is coherent. 
\end{prop}
\begin{proof}
The previous proposition tells us that we only need to show that the transition morphisms $\widehat{\Da}^{(m,k)}_{\mathfrak{X},\lambda,\Q}\rightarrow \widehat{\Da}^{(m+1,k)}_{\mathfrak{X},k,\Q}$ are flat. As this is a local property we can take $U\in\Sb$ (covering in proposition \ref{algebraic_local_desc}) and to verify this  property over the formal completion $\mathfrak{U}$. In this case, the argument used in the proof of the first part of proposition \ref{locally_hatD} give us the following commutative diagram  

\begin{eqnarray*}
\begin{tikzcd}
\widehat{\Da}^{(m,k)}_{\mathfrak{X},\lambda,\Q}(\mathfrak{U}) \arrow[r] \arrow[d,"\simeqd"]
& \widehat{ \Da}^{(m+1,k)}_{\mathfrak{X},\lambda,\Q}(\mathfrak{U}) \arrow[d,"\simeqd"] \\
\widehat{\Da}^{(m,k)}_{\mathfrak{X},\Q}(\mathfrak{U}) \arrow[r]
&  \widehat{\Da}^{(m+1,k)}_{\mathfrak{X},\Q}(\mathfrak{U})  
\end{tikzcd}
\end{eqnarray*}
The flatness theorem \cite[Proposition 2.2.11 (iii)]{HSS} states that the lower morphism is flat and so is the morphism on the top.
\end{proof}

\begin{lem}\label{Dag_to_m}
For every coherent $\Da^{\dag}_{\mathfrak{X},k}(\lambda)$-module $\Ea$ there exists $m\ge 0$, a coherent $\widehat{\Da}^{(m,k)}_{\mathfrak{X},\lambda,\Q}$-module $\Ea_m$ and an isomorphism of $\Da^{\dag}_{\mathfrak{X},k}(\lambda)$-modules
\begin{eqnarray*}
\tau: \Da^{\dag}_{\mathfrak{X},k}(\lambda)\otimes_{\widehat{\Da}^{(m,k)}_{\mathfrak{X},\lambda,\Q}}\Ea_{m}\xrightarrow{\simeq}\Ea.
\end{eqnarray*}
Moreover, if $(m',\Ea_{m'},\tau')$ is another such triple, then there exists $l\ge\text{max}\{m,m'\}$ and an isomorphism of $\widehat{\Da}^{(l,k)}_{\mathfrak{X},\lambda,\Q}$-modules
\begin{eqnarray*}
\tau_{l}:\widehat{\Da}^{(l,k)}_{\mathfrak{X},\lambda,\Q}\otimes_{\widehat{\Da}_{\mathfrak{X},\lambda,\Q}^{(m,k)}}\Ea_{m}\xrightarrow{\simeq} \widehat{\Da}^{(l,k)}_{\mathfrak{X},\lambda,\Q}\otimes_{\widehat{\Da}_{\mathfrak{X},\lambda,\Q}^{(m',k)}}\Ea_{m'}
\end{eqnarray*}
such that $\tau'\circ\left(id_{\Da^{\dag}_{\mathfrak{X},k}(\lambda)}\otimes\tau_l\right)=\tau$.
\end{lem}
\begin{proof}
This is \cite[proposition 3.6.2 (ii)]{Berthelot1}. We remark that $\mathfrak{X}$ is quasi-compact and separated, and the sheaf $\hat{\Da}^{(m,k)}_{\mathfrak{X},\lambda,\Q}$ satisfies the conditions in \cite[3.4.1]{Berthelot1}.
\end{proof}

\begin{prop}
Let $\Ea$ be a coherent $\Da^{\dag}_{\mathfrak{X},k}(\lambda)$-module.
\begin{itemize}
\item[(i)] There exists an integer $r(\Ea)$ such that, for all $r\ge r(\Ea)$ there is $a\in\N$ and an epimorphism of $\Da^{\dag}_{\mathfrak{X},k}(\lambda)$-modules
\begin{eqnarray*}
\left(\Da^{\dag}_{\mathfrak{X},k}(\lambda)(-r)\right)^{\oplus a}\rightarrow \Ea\rightarrow 0.
\end{eqnarray*}
\item[(ii)] For all $i>0$ one has $H^{i}(\mathfrak{X},\Ea)=0$.
\end{itemize}
\end{prop}
\begin{proof}\footnote{This is exactly as in \cite[theorem 4.2.8]{HPSS}.}
Let $\Ea$ be a coherent $\Da^{\dag}_{\mathfrak{X},k}(\lambda)$-coherent module. The preceding proposition tells us that there exists $m\in\N$, a coherent $\widehat{\Da}^{(m,k)}_{\mathfrak{X},\lambda,\Q}$-module $\Ea_m$ and an isomorphism of $\Da^{\dag}_{\mathfrak{X},k}(\lambda)$-modules
\begin{eqnarray*}
\tau:\Da^{\dag}_{\mathfrak{X},k}(\lambda)\otimes_{\widehat{\Da}^{(m,k)}_{\mathfrak{X},\lambda,Q}}\Ea_m\xrightarrow{\simeq}\Ea.
\end{eqnarray*} 
Now we use proposition \ref{Resolution_D_Q} for $\Ea_m$ and we get the desired surjection in (i) after tensoring with $\Da^{\dag}_{\mathfrak{X},k}(\lambda)$. To show (ii) we may use the fact that, as $\mathfrak{X}$ is a noetherian topological space, cohomology commutes with direct limites. Therefore, given that $\widehat{\Da}^{(l,k)}_{\mathfrak{X},\lambda,\Q}\otimes_{{\widehat{\Da}^{(m,k)}_{\mathfrak{X},\lambda,\Q}}}\Ea_{m}$ is a coherent $\Da^{(l,k)}_{\mathfrak{X},\lambda,\Q}$-module for every $l\ge m$, we have for every $i>0$
\begin{eqnarray*}
H^{i}(\mathfrak{X},\Ea)=\varinjlim_{l\ge m}H^{i}\left(\mathfrak{X},\widehat{\Da}^{(l,k)}_{\mathfrak{X},\lambda,\Q}\otimes_{{\widehat{\Da}^{(m,k)}_{\mathfrak{X},\lambda,\Q}}}\Ea_{m}\right)=0.
\end{eqnarray*}
\end{proof}

\begin{prop}\label{generated_by_global_sect}
Let $\Ea$ be a coherent $\Da^{\dag}_{\mathfrak{X},k}(\lambda)$-module. Then $\Ea$ is generated by its global sections as $\Da^{\dag}_{\mathfrak{X},k}(\lambda)$-module. Moreover, $\Ea$ has a resolution by finite free $\Da^{\dag}_{\mathfrak{X},k}(\lambda)$-modules and $H^0(\mathfrak{X},\Ea)$ is a $D^{\dag}(\G(k))_{\lambda}\otimes_{\mathfrak{o}}L$-module of finite presentation. 
\end{prop}
\begin{proof}\footnote{This is exactly as in \cite[theorem 5.1]{Huyghe1}.}
Theorem \ref{Dag_to_m} gives us a coherent $\widehat{\Da}^{(m,k)}_{\mathfrak{X},\lambda,\Q}$-module $\Ea_m$ such that $\Ea\simeq\Da^{\dag}_{\mathfrak{X},k}(\lambda)\otimes_{\widehat{\Da}^{(m,k)}_{\mathfrak{X},\lambda,\Q}}\Ea_m$. Moreover, $\Ea_m$ has a resolution by finite free $\widehat{\Da}^{(m,k)}_{\mathfrak{X},\lambda,\Q}$-modules ( proposition \ref{prop 3.3.1}). Both results clearly imply the first and the second part of the lemma. The final part of the lemma is therefore a consequence of the first part and the acyclicity of the the functor $H^{0}(\mathfrak{X},\bullet)$.
\end{proof}

\begin{proof}[Proof of theorem \ref{BB_for_dag}]
All in all, we can follow the same arguments of \cite[corollary 2.3.7]{Huyghe2}. We start by taking 
\begin{eqnarray*}
(D^{\dag}(\G(k))_{\lambda}\otimes_{\mathfrak{o}}L)^{\oplus a}\rightarrow (D^{\dag}(\G(k))_{\lambda}\otimes_{\mathfrak{o}}L)^{\oplus b}\rightarrow E\rightarrow 0
\end{eqnarray*}
a finitely presented $D^{\dag}(\G(k))_{\lambda}\otimes_{\mathfrak{o}}L$-module. By localizing and applying the global sections functor, we obtain a commutative diagram
\begin{eqnarray*}
\begin{tikzcd}
(D^{\dag}(\G(k))_{\lambda}\otimes_{\mathfrak{o}}L)^{\oplus a} \arrow[d] \arrow[r] 
& (D^{\dag}(\G(k))_{\lambda}\otimes_{\mathfrak{o}}L)^{\oplus b} \arrow[d] \arrow[r] 
& E \arrow[d] \arrow[r] 
& 0\\
(D^{\dag}(\G(k))_{\lambda}\otimes_{\mathfrak{o}}L)^{\oplus a}  \arrow[r] 
& (D^{\dag}(\G(k))_{\lambda}\otimes_{\mathfrak{o}}L)^{\oplus b} \arrow[r] 
& H^0(\mathfrak{X},\mathscr Loc_{\mathfrak{X},k}^{\dag}(\lambda)(E)) \arrow[r] 
& 0.
\end{tikzcd} 
\end{eqnarray*}
which tells us that $E\rightarrow H^{0}(\mathfrak{X},\La oc^{\dag}_{\mathfrak{X},k}(\lambda)(E))$ is an isomorphism. To show that if $\Ea$ is coherent $\Da^{\dag}_{\mathfrak{X},k}(\lambda)$-module then the canonical morphism $\Da^{\dag}_{\mathfrak{X},k}(\lambda)\otimes_{D^{\dag}(\G(k))_{\lambda}\otimes_{\mathfrak{o}}L}H^{0}(\mathfrak{X},\Ea)\rightarrow\Ea$  is an isomorphism the reader can follow the same argument as before. As we have remarked, the second assertion follows because any equivalence between abelian categories is exact. 
\end{proof}

\section{Twisted differential operators on formal models of flag varieties}
Through out this section $X=\G/\B$ will denote the smooth flag $\mathfrak{o}$-scheme and $\lambda\in X(\T)=\text{Hom}(\T,\G_m)$ will always denote an algebraic character. As before, we will denote be $\Lb(\lambda)$ the (algebraic) line bundle on $X$ induced by $\lambda$ (subsection \ref{Linearization_alg_gps}). In this section we will generalize the construction given in \cite{HPSS} by introducing sheaves of twisted differential operators on an admissible blow-up of the smooth formal flag $\mathfrak{o}$-scheme $\mathfrak{X}$. The reader will figure out that some reasoning are inspired in the results of Huyghe-Patel-Strauch-Schmidt in \cite{HPSS}.

\subsection{Differential operators on admissible blow-ups}\label{diff_on_adm_bup}
\justify
 We start with the following definition.
\begin{defi}
Let $\Ib\subseteq\Ob_X$ be a coherent ideal sheaf. We say that a blow-up $\text{pr}: Y\rightarrow X$ along the closed subset $V(\Ib)$ is admissible if there is $k\in\N$ such that $\varpi^k\Ob_X\subseteq \Ib$.
\end{defi} 
\justify
Let us fix $\Ib\subseteq \Ob_X$ an open ideal and $\text{pr}:Y\rightarrow X$ an admissible blow-up along $V(\Ib)$. We point out to the reader that $\Ib$ is not uniquely determined by the space $Y$. In the sequel we will denote by
\begin{eqnarray*}
k_Y:= \underset{\Ib}{\text{min}}\;\text{min}\{k\in\N\;|\; \varpi^k\in\Ib\},
\end{eqnarray*}
where the first minimum runs over all open ideal sheaves $\Ib$ such that the blow-up along $V(\Ib)$ is isomorphic to $Y$.
\justify
Now, as $\Ib$ is an open ideal sheaf, the blow-up induces a canonical isomorphism $Y_L\simeq X_L$ between the generic fibers. Moreover, as $\varpi$ is invertible on $X_L$, we have $\Db^{(m,k)}_{X}|_{X_L}= \Db_{X}|_{X_L}=\Db_{X_L}$, the usual sheaf of (algebraic) differential operators on $X_L$. Therefore $\text{pr}^{-1}\left(\Db^{(m,k)}_X\right)|_{Y_L}=\Db_{Y_L}$. In particular, $\Ob_{Y_L}$ has a natural structure of (left) $\text{pr}^{-1}\left(\Db^{(m,k)}_X\right)|_{Y_L}$-module. The idea is to find those congruence levels $k\in\N$ such that the preceding structure extends to a module structure on $\Ob_{Y}$ over $\text{pr}^{-1}\left(\Db^{(m,k)}_X\right)$. Let us denote by
\begin{eqnarray}\label{Diff_blow_up}
\Db^{(m,k)}_Y:=\text{pr}^*\left(\Db^{(m,k)}_X\right)=\Ob_{Y}\otimes_{\text{pr}^{-1}\Ob_X}\text{pr}^{-1}\Db^{(m,k)}_X.
\end{eqnarray}
The problem to find those congruence levels was studied in \cite{HPSS} and \cite{HSS}. In fact, we have the following condition \cite[Corollary 2.1.18]{HPSS}.
 
\begin{prop}\label{mult_blow_up}
Let $k\ge k_Y$. The sheaf $\Db^{(m,k)}_Y$ is a sheaf of rings on $Y$. Moreover, it is locally free over  $\Ob_Y$.
\end{prop}
\justify
Explicitly, if $\partial_1,\partial_2$ are both local sections of $\text{pr}^{-1}\left(\Db^{(m,k)}_X\right)$, and if $f_1, f_2$ are local sections of $\Ob_Y$, then 
\begin{eqnarray*}
(f_1 \otimes \partial_1)\bullet (f_2\otimes\partial_2)=f_1\partial_1(f_2)\otimes\partial_2 + f_1 f_2\otimes \partial_1\partial_2.
\end{eqnarray*} 
\justify
We have all the ingredients that allow us to construct the desired sheaves over $Y$, this is, to extend the sheaves of rings defined in the preceding chapter to an admissible blow-up $Y$ of $X$. Let $k\ge k_Y$ fix. Let us first recall that taking arbitrary sections $P,Q\in\Db^{(m,k)}_{X}$, $s,t \in \Lb(\lambda)$ and $s^{\vee},t^{\vee}\in\Lb(\lambda)^{\vee}$ (the last two not necessarily the duals of $s$ and $t$) over an arbitrary open subset $U\subset X$, the multiplicative structure of the sheaf $\Db^{(m,k)}_{X,\lambda}$ is defined by (cf. (\ref{mt}))
\begin{eqnarray*}
s\otimes P\otimes s^{\vee}\bullet t\otimes Q\otimes t^{\vee}=s\otimes P \left<s^{\vee},t\right>Q\otimes t^{\vee}.
\end{eqnarray*}
Now, if pr$:Y\rightarrow X$ denotes the projection, we put
\begin{eqnarray*}
\Db^{(m,k)}_Y(\lambda):= \text{pr}^{*}\left(\Db^{(m,k)}_{X,\lambda}\right)=  \text{pr}^*\Lb(\lambda)\otimes_{\Ob_Y}\text{pr}^*\Db^{(m,k)}_X\otimes_{\Ob_Y}\text{pr}^*\Lb (\lambda)^{\vee}.
\end{eqnarray*}
\justify
Proposition \ref{mult_blow_up} allows us to endow the sheaf of $\Ob_Y$-modules $\Db^{(m,k)}_Y(\lambda)$ with a multiplicative structure for every $k\ge k_Y$. On local sections we have
\begin{eqnarray*}
s\otimes P\otimes s^{\vee}\bullet t\otimes Q\otimes t^{\vee}=s\otimes P \left<s^{\vee},t\right>Q\otimes t^{\vee},
\end{eqnarray*}
where $s,t\in\text{pr}^{*}\Lb (\lambda)$, $s^{\vee}, t^{\vee}\in \text{pr}^*\Lb (\lambda)^{\vee}$ and $P,Q\in \Db^{(m,k)}_Y$, are local sections.
\justify
Let $\mathfrak{Y}$ be the completion of $Y$ along its special fiber $Y_{\mathbb{F}_q}=Y\times_{\text{Spec}(\mathfrak{o})}\text{Spec}(\mathfrak{o}/\varpi)$.

\begin{Notation}\label{notation_k_y}
In this work we will only consider formal blow-ups $\mathfrak{Y}$ arising from the formal completion along the special fiber of an admissible blow-up $Y\rightarrow X$ \cite[Proposition 2.2.9]{HPSS}. Under this assumption we will identify $k_Y = k_{\mathfrak{Y}}$.  
\end{Notation}

\begin{defi}
Let $\text{pr}:Y\rightarrow X$ be an admissible blow-up of the flag variety $X$ and let $k\ge k_Y$.  The sheaves 
\begin{eqnarray*}
\widehat{\Da}^{(m,k)}_{\mathfrak{Y},\Q}(\lambda):=\left( \varprojlim_{i\in\N} \Db^{(m,k)}_Y(\lambda)/\varpi^{i+1}\Db^{(m,k)}_Y(\lambda)\right)\otimes_{\mathfrak{o}}L\;\;\;\text{and}\;\;\; \Da^{\dag}_{\mathfrak{Y},k}(\lambda):= \varinjlim_{m\in\N} \widehat{\Da}^{(m,k)}_{\mathfrak{Y},\Q}(\lambda).
\end{eqnarray*}
are called sheaves of $\lambda$-twisted arithmetic differential operators on $\mathfrak{Y}$.
\end{defi}

\begin{prop}\label{fini_blow_up}
\begin{itemize}
\item[(i)] The sheaves $\Db^{(m,k)}_{Y}(\lambda)$ are filtered by the order of twisted differential operators and there is a canonical isomorphism of graded sheaves of algebras
\begin{eqnarray*}
\textrm{gr}\left(\Db^{(m,k)}_{Y}(\lambda)\right)\simeq \text{Sym}^{(m)}\left(\varpi^k\text{pr}^*\Tb_{X}\right),
\end{eqnarray*}
where $k\ge k_{Y}$.
\item[(ii)] There is a basis for the topology of $Y$, consisting of affine open subsets, such that for any open subset $U\in Y$ in this basis, the ring $\Db^{(m,k)}_{Y}(\lambda)\left(U\right)$ is noetherian. In particular, the sheaf of rings $\Db^{(m,k)}_{Y}(\lambda)$ is coherent.
\item[(iii)] The sheaf $\widehat{\Da}^{(m,k)}_{\mathfrak{Y},\Q}(\lambda)$ is coherent.
\end{itemize}
\end{prop}
\begin{proof}
By \ref{graded}, we have an exact sequence of $\Ob_{X}$-modules
\begin{eqnarray*}
0\rightarrow \Db^{(m,k)}_{X,d-1}\rightarrow \Db^{(m,k)}_{X,d}\rightarrow \text{Sym}^{(m)}_{d}\left(\varpi^k\Tb_{X}\right)\rightarrow 0.  
\end{eqnarray*}
Taking the tensor product with $\Lb(\lambda)$ and $\Lb(\lambda)^{\vee}$ on the left and on the right, respectively, and applying $\text{pr}^*$ we obtain the exact sequence (since $\text{Sym}^{(m)}_d(\varpi^k\Tb_{X})$ is a locally free $\Ob_{X}$-module of finite rank)
\begin{eqnarray*}
0\rightarrow \Db^{(m,k)}_{Y,d-1}(\lambda)\rightarrow \Db^{(m,k)}_{Y,d}(\lambda)\rightarrow\textrm{pr}^*\Lb(\lambda) \otimes_{\Ob_{Y}}\text{Sym}^{(m)}_{d}\left(\varpi^k\text{pr}^*\Tb_{X}\right)\otimes_{\Ob_Y}\text{pr}^*\Lb(\lambda)^{\vee}\rightarrow 0,
\end{eqnarray*} 
which implies (i) because
\begin{eqnarray*}
\textrm{pr}^*\Lb(\lambda) \otimes_{\Ob_{Y}}\text{Sym}^{(m)}\left(\varpi^k\text{pr}^*\Tb_{X}\right)\otimes_{\Ob_Y}\text{pr}^*\Lb(\lambda)^{\vee}\simeq\text{Sym}^{(m)}\left(\varpi^k\text{pr}^*\Tb_{X}\right)
\end{eqnarray*}
by commutativity of the symmetric algebra.
\justify
Let $U\subseteq X$ be an affine open subset endowed with local coordinates $x_1,...,x_M$ and such that $\Lb (\lambda)|_{U}=s\Ob_{U}$ for some $s\in\Lb(\lambda)(U)$. Then, by lemma \ref{algebraic_local_desc} we have the following local description for $\Db^{(m,k)}_{Y}(\lambda)$ on $V=\textrm{pr}^{-1}(U)$ 
\begin{eqnarray*}
\Db^{(m,k)}_{Y}(\lambda)(V)=\left\{\sum_{\underline{v}}^{<\infty}\varpi^{k|\underline{v}|}a_{\underline{v}}\underline{\partial}^{\left<\underline{v}\right>}|\; \underline{v}=(v_1,...,v_M)\in\mathbb{N}^M\;and\; a_{\underline{v}}\in \Ob_{Y}(V)\right\}.
\end{eqnarray*}
By (i), the graded algebra $gr_{\bullet}\left(\Db^{(m,k)}_{Y}(\lambda)(V)\right)$ is isomorphic to $\text{Sym}^{(m)}\left(\varpi^k \text{pr}^*\Tb_{X}(V)\right)$ which is known to be noetherian \cite[Proposition 1.3.6]{Huyghe1}. Therefore, taking as a basis the set of affine open subsets of $Y$ that are contained in some $\textrm{pr}^{-1}(U)$ we get (ii). We also remark that, as $\Db^{(m,k)}_{Y}(\lambda)$ is $\Ob_Y$-quasi-coherent, and by (ii) in the actual proposition, it has noetherian sections over the affine open subsets of $Y$ (cf. \cite[Proposition 2.2.2 (iii)]{HSS}), it is certainly a sheaf of coherent rings \cite [proposition 3.1.3]{Berthelot1}. Finally, by definition, we see that  $\widehat{\Da}^{(m,k)}_{\mathfrak{Y}}(\lambda)$ satisfies the conditions (a) and (b) of 3.3.3 in \cite{Berthelot1} and hence \cite[Proposition 3.3.4]{Berthelot1} gives us (iii). 
\end{proof}
\justify
Let us briefly study the problem of passing to the inductive limit when $m$ varies.
\justify
Let $U\subset X$ such that $\Db^{(m,k)}_{X}(\lambda)|_{U}\simeq \Db^{(m,k)}_{X}|_{U}$ and let us take $V\subseteq Y$ an affine open subset such that $V\subseteq\text{pr}^{-1}(U)$. We have the commutative diagram
\begin{eqnarray*}
\begin{tikzcd}
V \arrow[r,"i_V",hook] \arrow[d, "pr"]
& Y \arrow[d, "pr"]\\
U \arrow[r, "i_U",hook]
& X,
\end{tikzcd}
\end{eqnarray*}
which implies that $\Db^{(m,k)}_{Y}(\lambda)|_{V}\simeq \Db^{(m,k)}_{Y}|_{V}$, as sheaves of rings. In particular, if $\mathfrak{V}$ denotes the formal $p$-adic completion of $V$ along the special fiber $V_{\mathbb{F}_q}$ we have the commutative diagram (cf. proposition \ref{coh. dag})
\begin{eqnarray}\label{trivial_line_bundle}
\begin{tikzcd}
\widehat{\Da}^{(m,k)}_{\mathfrak{Y},\Q}(\lambda)(\mathfrak{V}) \arrow[r] \arrow[d,"\simeqd"]
& \widehat{\Da}^{(m+1,k)}_{\mathfrak{Y},\Q}(\lambda)(\mathfrak{V}) \arrow[d,"\simeqd"]\\
\widehat{\Da}^{(m,k)}_{\mathfrak{Y},\Q}(\mathfrak{V}) \arrow[r]
& \widehat{\Da}^{(m+1,k)}_{\mathfrak{Y},\Q}(\mathfrak{V}).
\end{tikzcd}
\end{eqnarray}
Given that the morphism of sheaves $\widehat{\Da}^{(m,k)}_{\mathfrak{Y},\Q}\rightarrow \widehat{\Da}^{(m+1,k)}_{\mathfrak{Y},\Q}$ is left and right flat \cite[Proposition 2.2.11 (iii)]{HSS}, the preceding diagram allows us to conclude that the morphism $\widehat{\Da}^{(m,k)}_{\mathfrak{Y},\Q}(\lambda)\rightarrow \widehat{\Da}^{(m+1,k)}_{\mathfrak{Y},\Q}(\lambda)$ is also left and right flat. By proposition \ref{coh_proj_lim} we have the following result.

\begin{prop}
The sheaf of rings $\Da^{\dag}_{\mathfrak{Y},k}(\lambda)$ is coherent.
\end{prop}
\justify
As we will explain later, there exists a canonical epimorphism of sheaves of filtered $\mathfrak{o}$-algebras\footnote{We construct this morphism in (\ref{canonical_lambda_blow-up}). The arguments given there are independent and we won't introduce a circular argument.}
\begin{eqnarray*}
\Ab^{(m,k)}_{Y} := \Ob_{Y}\otimes_{\mathfrak{o}} D^{(m)}(\G(k)) \rightarrow \Db^{(m,k)}_{Y}(\lambda)
\end{eqnarray*}
which allows to conclude the following proposition exactly as we have done in the proof of proposition \ref{prop 3.3.1} (cf. \cite[Proposition 4.3.1]{HPSS}).

\begin{prop}\label{Globally_gene}
Let $\lambda\in\text{Hom}(\T,\G_m)$ be an algebraic character such that $\lambda+\rho\in\mathfrak{t}_L^*$ is a dominant and regular character of $\mathfrak{t}_L$.
\begin{itemize}
\item[(i)] Let $\Ea$ be a coherent $\widehat{\Da}^{(m,k)}_{\mathfrak{Y},\Q}(\lambda)$-module. Then $\Ea$ is generated by its global sections as $\widehat{\Da}^{(m,k)}_{\mathfrak{Y},\Q}(\lambda)$-module. Furthermore, $\Ea$ has a resolution by finite free $\widehat{\Da}^{(m,k)}_{\mathfrak{Y},\Q}(\lambda)$-modules.
\item[(ii)] Let $\Ea$ be a coherent $\Da^{\dag}_{\mathfrak{Y},k}(\lambda)$-module. Then $\Ea$ is generated by its global sections as $\Da^{\dag}_{\mathfrak{Y},k}(\lambda)$-module. Furthermore, $\Ea$ has a resolution by finite free $\Da^{\dag}_{\mathfrak{Y},k}(\lambda)$-modules. 
\end{itemize}
\end{prop}

\subsection{An Invariance theorem for admissible blow-ups}\label{section_inv_theorem}
\justify
Let $pr:\mathfrak{Y}\rightarrow\mathfrak{X}$ be an admissible blow-up along a closed subset $\textbf{V}(\Ia)$ defined by an open ideal sheaf $\Ia\subseteq \Ob_{\mathfrak{X}}$. Using (\ref{notation_k_y}), we can suppose that $\mathfrak{Y}$ is obtained as the formal completion of an admissible blow-up $Y\rightarrow X$ (we will abuse of the notation and we will denote again by $\text{pr}:Y\rightarrow X$  the canonical morphism of this (algebraic) blow-up) along a closed subset $\textbf{V}(\Ib)$ defined by an open ideal sheaf $\Ib\subseteq \Ob_{X}$, such that $\Ia$ is the restriction of the formal $\varpi$-adic completion of $\Ib$.  Let us denote by $Y_i:=Y\times_{\text{Spec}(\mathfrak{o})}\text{Spec}(\mathfrak{o}/\varpi^{i+1})$ the redaction module $\varpi^{i+1}$ and by $\gamma_i: Y_i\rightarrow Y$ the canonical closed embedding. In \cite{HSS} the authors have studied the cohomological properties of the sheaves  
\begin{eqnarray*}
\widehat{\Da}^{(m,k)}_{\mathfrak{Y},\Q}:=\varprojlim_{i\in\N} \gamma_{i}^*\Db^{(m,k)}_{Y}\otimes_{\mathfrak{o}}L\;\;\;\;\text{and}\;\;\;\;
\Da^{\dag}_{\mathfrak{Y},k}:= \varinjlim_{m\in\N} \widehat{\Da}^{(m,k)}_{\mathfrak{Y},\Q}.
\end{eqnarray*}
Let us consider the commutative diagram 
\begin{eqnarray*}
\begin{tikzcd}
Y_i \arrow[r, "pr_i"] \arrow[d, "\gamma_i"]
& X_i \arrow[d, "\gamma_i"]\\
Y \arrow[r, "pr"]
& X.
\end{tikzcd}
\end{eqnarray*}
Here $\text{pr}_{i}:Y_i\rightarrow X_i$ denotes the redaction of the morphism $pr$ module $\varpi^{i+1}$. We put $\underline{\La(\lambda)^{\vee}}:=\varprojlim_i \gamma_i^* \text{pr}^*\Lb(\lambda)^{\vee}$ and $\underline{\La(\lambda)}:=\varprojlim_i \gamma_i^* \text{pr}^*\Lb(\lambda)$. By using the preceding commutative diagram we have 
\begin{align*}
\gamma_{i}^{*}\Db^{(m,k)}_{Y}(\lambda) & = \gamma_i^*\left(\text{pr}^*\Lb (\lambda)\otimes_{\Ob_Y}\Db^{(m,k)}_{Y}\otimes_{\Ob_Y} \text{pr}^{*}\Lb(\lambda)^{\vee}\right)\\
& =  \gamma_i^{*}\left(\text{pr}^{*}\Lb(\lambda)\right)\otimes_{\Ob_{Y_i}} \gamma_i^*\Db^{(m,k)}_{Y}\otimes_{\Ob_{Y_i}}\gamma_i^*\left(\text{pr}^*\Lb(\lambda)^{\vee}\right).
\end{align*}
Taking the projective limit we get  the following description of the sheaves $\widehat{\Da}^{(m,k)}_{\mathfrak{Y},\Q}(\lambda)$
\begin{eqnarray*}
\widehat{\Da}^{(m,k)}_{\mathfrak{Y},\Q}(\lambda)= \underline{\La(\lambda)}_{\Q}\otimes_{\Ob_{\mathfrak{Y},\Q}}\widehat{\Da}^{(m,k)}_{\mathfrak{Y},\Q} \otimes_{\Ob_{\mathfrak{Y},\Q}}\underline{\La(\lambda)^{\vee}}_{\Q},
\end{eqnarray*} 
and by taking the inductive limit we get the characterization
\begin{eqnarray}\label{Blow_up_action_line_bundle}
\Da^{\dag}_{\mathfrak{Y},k}(\lambda)= \underline{\La(\lambda)}_{\Q}\otimes_{\Ob_{\mathfrak{Y},\Q}}\Da^{\dag}_{\mathfrak{Y},k} \otimes_{\Ob_{\mathfrak{Y},\Q}}\underline{\La(\lambda)^{\vee}}_{\Q}.
\end{eqnarray}
As in the preceding section, the sheaf $\underline{\La(\lambda)}_{\Q}$ is endowed with the following (left) $\Da^{\dag}_{\mathfrak{Y},k}(\lambda)$-action
\begin{eqnarray*}
\left(t\otimes P\otimes t^{\vee}\right)\bullet s\; :=\; \left(P\bullet < t^{\vee},s >\right)t\;\;\;\; (s,t \in \underline{\La(\lambda)}\;\;\text{and}\;\; t^{\vee}\in\underline{\La(\lambda)^{\vee}}).
\end{eqnarray*}
We end this first discussion by remarking that the relation $\text{pr}_i^*\;\circ\;\gamma_i^*\;=\; \gamma_i^{*}\;\circ\;\text{pr}^*$, which comes from the preceding commutative diagram, implies that
\begin{eqnarray}\label{Inverse_image_D_dag_X}
\Da^{\dag}_{\mathfrak{Y},k}(\lambda) = \text{pr}^* \Da^{\dag}_{\mathfrak{X},k}(\lambda).
\end{eqnarray}
\justify
Let us  suppose that $\pi: Y'\rightarrow Y$ is a morphism of admissible blow-ups (abusing of the notation, we will also denote by $\pi:\mathfrak{Y}'\rightarrow\mathfrak{Y}$ the respective morphism of formal admissible blow-ups in the sense of \cite[Part II, chapter 8, section 8.2, definition 3]{Bosch}). This means that we  have a commutative diagram
\begin{eqnarray*}
\begin{tikzcd}
Y' \arrow[rd, "pr' "] \arrow[r, "\pi "] & Y \arrow[d, "pr"] \\
& X.
\end{tikzcd}
\;\;\;\;\;\;
\text{resp.}
\;\;\;
\begin{tikzcd}
\mathfrak{Y}' \arrow[rd, "pr' "] \arrow[r, "\pi "] & \mathfrak{Y} \arrow[d, "pr"] \\
& \mathfrak{X}.
\end{tikzcd}
\end{eqnarray*}
Let $k\ge\text\{k_{Y'},k_Y\}$. Let us denote by $\Db^{(m,k)}_{X,i}(\lambda):=\Db^{(m,k)}_X(\lambda)/\varpi^{i+1}\Db^{(m,k)}_Y(\lambda)$ (we will use the same notations over $Y'_i$ and $Y_i$) and by $\pi_{i}: Y'_i\rightarrow Y_i$ the redaction module $\varpi^{i+1}$. The preceding commutative diagram implies that 
\begin{eqnarray}\label{rel_morph_blow_up}
\Db^{(m,k)}_{Y',i}(\lambda) = (pr'_i)^*\Db^{(m,k)}_{X_i}(\lambda) = \pi_i^*\Db^{(m,k)}_{Y_i}(\lambda).
\end{eqnarray}
In this way, she sheaf $\Db^{(m,k)}_{Y',i}(\lambda)$ can be endowed with a structure of right  $\pi_i^{-1}\Db^{(m,k)}_{Y_i}(\lambda)$-module. Passing to the projective limit, the sheaf $\widehat{\Da}^{(m,k)}_{\mathfrak{Y}'}(\lambda)$ is a sheaf of right $\pi^{-1}\widehat{\Da}^{(m,k)}_{\mathfrak{Y}}(\lambda)$-modules. So, passing to the inductive limit over $m$ we can conclude that $\Da^{\dag}_{\mathfrak{Y}',k}(\lambda)$ is a right $\pi^{-1}\Da^{\dag}_{\mathfrak{Y},k}(\lambda)$-module. For a $\Da^{\dag}_{\mathfrak{Y},k}(\lambda)$-module $\Ea$, we define
\begin{eqnarray*}
\pi^{!}\Ea := \Da^{\dag}_{\mathfrak{Y}',k}(\lambda)\otimes_{\pi^{-1}\Da^{\dag}_{\mathfrak{Y},k}(\lambda)}\pi^{-1}\Ea,
\end{eqnarray*}
with analogous definitions for $\widehat{\Da}^{(m,k)}_{\mathfrak{Y},\Q}(\lambda)$. 

\begin{theo}\label{Invariance}
Let $\pi:Y'\rightarrow Y$ be a morphism over $X$ of admissible blow-ups. Let $k\ge \text{max}\{k_{Y'}, k_Y\}$.
\begin{itemize}
\item[(i)] If  $\Ea$ is a coherent $\Da^{\dag}_{\mathfrak{Y}',k}(\lambda)$, then R$^j\pi_*\Ea =0$ for every $j>0$. Moreover, $\pi_{*}\Da^{\dag}_{\mathfrak{Y}',k}(\lambda)=\Da^{\dag}_{\mathfrak{Y},k}(\lambda)$, so $\pi_*$ induces an exact functor between coherent modules over $\Da^{\dag}_{\mathfrak{Y}',k}(\lambda)$ and $\Da^{\dag}_{\mathfrak{Y},k}(\lambda)$, respectively.
\item[(ii)] The formation $\pi^{!}$ is an exact functor from the category of coherent $\Da^{\dag}_{\mathfrak{Y},k}(\lambda)$-modules to the category of coherent $\Da^{\dag}_{\mathfrak{Y}',k}(\lambda)$-modules.
\item[(iii)] The functors $\pi_*$ and $\pi^{!}$ are quasi-inverse equivalences between the categories of coherent $\Da^{\dag}_{\mathfrak{Y}',k}(\lambda)$-modules and coherent $\Da^{\dag}_{\mathfrak{Y},k}(\lambda)$-modules.
\end{itemize}
\end{theo}
\justify
We remark for the reader that this theorem has an equivalent version for the sheaves $\widehat{\Da}^{(m,k)}_{\mathfrak{Y},\Q}(\lambda)$ and  $\widehat{\Da}^{(m,k)}_{\mathfrak{Y}',\Q}(\lambda)$.
\begin{proof}
 Let us first assume that $\Ea=\Da^{\dag}_{\mathfrak{Y}',k}(\lambda)$. Let us consider the covering $\Ba$ of $\mathfrak{X}$, defined in proposition \ref{locally_hatD} and let us take $\Ub\in\Ba$. We put $\Vb':=\text{pr}'^{-1}(\Ub)$ and $\Vb:=\text{pr}^{-1}(\Ub)$. By assumption $\Vb' = \pi^{-1}(\Vb)$ in such a way that 
 \begin{eqnarray*}
 R^j\pi_*\left(\Da^{\dag}_{\mathfrak{Y}',k}(\lambda)\right)|_{\Vb'} =  R^j\pi_*\left(\Da^{\dag}_{\mathfrak{Y}',k}(\lambda)|_{\Vb}\right) =  R^j\pi_*\left(\Da^{\dag}_{\mathfrak{Y}',k}|_{\Vb}\right) =  R^j\pi_*\left(\Da^{\dag}_{\mathfrak{Y}',k}\right)|_{\Vb'}.
 \end{eqnarray*}
 No we can use \cite[Theorem 2.3.8 (i)]{HSS} to conclude  that $R^j\pi_*\Da^{\dag}_{\mathfrak{Y}',k}(\lambda) = 0$ for every $j>0$. Furthermore, by (\ref{rel_morph_blow_up}) there exists a canonical map 
\begin{eqnarray*}
\Da^{\dag}_{\mathfrak{Y},k}(\lambda)\rightarrow \pi_*\Da^{\dag}_{\mathfrak{Y}',k}(\lambda)
\end{eqnarray*}
which is in fact an isomorphism by the preceding reasoning and \cite[Theorem 2.3.8 (i)]{HSS}.
\justify
To handle with the second part let us consider the following assertion for every $j\ge 1$. Let $a_j$: for any coherent $\Da^{\dag}_{\mathfrak{Y}',k}(\lambda)$-module $\Ea$ and for all $l\ge j$, $R^l\pi_*\Ea=0$. The assertion is true for $j=\text{dim}(\mathfrak{Y})+1$. Let us suppose that $a_{j+1}$ is true and let us take  a coherent $\Da^{\dag}_{\mathfrak{Y}',k}(\lambda)$-module $\Ea$. By proposition \ref{Globally_gene} there exists $b\in\N$ and a short exact sequence of coherent $\Da^{\dag}_{\mathfrak{Y}',k}(\lambda)$-modules
\begin{eqnarray*}
0\rightarrow \Fa\rightarrow\left(\Da^{\dag}_{\mathfrak{Y}',k}(\lambda)\right)^{\oplus b}\rightarrow\Ea\rightarrow 0.
\end{eqnarray*}
Since $R^j\pi_*\Da^{\dag}_{\mathfrak{Y}',k}(\lambda)=0$, the long exact sequence for $\pi_*$ gives us 
\begin{eqnarray*}
R^j\pi_*\Ea\simeq R^{j+1}\pi_*\Fa,
\end{eqnarray*}
which is 0 by induction hypothesis. This ends the proof of (i).
\justify
Let us show (ii)  for the sheaves $\Da^{\dag}_{\mathfrak{Y},k}(\lambda)$. The case for the sheaves $\widehat{\Da}^{(m,k)}_{\mathfrak{Y},\Q}(\lambda)$ being equal. Given that $\pi^!\Da^{\dag}_{\mathfrak{Y},k}(\lambda)=\Da^{\dag}_{\mathfrak{Y}',k}(\lambda)$, and since the tensor product is right exact, we can conclude that $\pi^!$ preserves coherence. 
\justify
Now, we have a morphism $\pi^{-1}\Ea\rightarrow \pi^!\Ea$ sending $m\mapsto 1\otimes m$. This maps induces the morphism $\Ea\rightarrow \pi_*\pi^!\Ea$. To show that this is an isomorphism is a local question on $\mathfrak{Y}$. If $\Vb\subseteq\mathfrak{Y}$ is the formal completion of an affine open subset $V\subseteq \text{pr}^{-1}(U)$, and $U\subseteq X$ is an affine open subset such that $\Db^{(m,k)}_{X}(\lambda)|_U\simeq \Db^{(m,k)}_X|_U$ (lemma \ref{algebraic_local_desc}), then by (\ref{trivial_line_bundle}) and \cite[Corollary 2.2.15]{HSS} we can conclude that the previous map is in fact an isomorphism over $\Vb$\footnote{This is the same reasoning that we have used in (i).}. Finally, if $\Fa$ is a coherent $\Da^{\dag}_{\mathfrak{Y}',k}(\lambda)$-module, then we have the map $\pi^!\pi_*\Fa\rightarrow\Fa$, sending $P\otimes m\mapsto Pm$. To see that this is an isomorphism we can use the preceding reasoning.
\end{proof}
\justify
Let us recall that if $\lambda\in\text{Hom}(\T,\G_m)$ is an algebraic character such that $\lambda+\rho\in\mathfrak{t}_L^*$ is a dominant and regular character of $\mathfrak{t}_L$, then by (\ref{Global_sect_dag}) we have
\begin{eqnarray*}
H^0\left(\mathfrak{X},\Da^{\dag}_{\mathfrak{X},k}(\lambda)\right)=D^{\dag}(\G(k))_{\lambda} = \Db^{\text{an}}(\G(k)^{\circ})_{\lambda} := \Db^{\text{an}}(\G(k)^{\circ})\big/\Db^{\text{an}}(\G(k)^{\circ}) (\text{Ker}).
\end{eqnarray*}
\justify
The previous theorem implies

\begin{coro}\label{equal_global_sections}
Let $\lambda\in\text{Hom}(\T,\G_m)$ be an algebraic character such that $\lambda+\rho\in\mathfrak{t}_L^*$ is a dominant and regular character of $\mathfrak{t}_L$. In the situation of the preceding theorem we have
\begin{eqnarray*}
H^0\left(\mathfrak{Y},\Da^{\dag}_{\mathfrak{Y},k}(\lambda)\right)=H^0\left(\mathfrak{X},\Da^{\dag}_{\mathfrak{X},k}(\lambda)\right)=D^{\dag}(\G(k))_{\lambda}= H^0\left(\mathfrak{Y}',\Da^{\dag}_{\mathfrak{Y}',k}(\lambda)\right).
\end{eqnarray*}
\end{coro}

\begin{theo}\label{BB_blow_up}
Let pr$:Y\rightarrow X$ be an admissible blow-up. Let us suppose that $\lambda\in\text{Hom}(\T,\G_m)$ is an algebraic character such that $\lambda+\rho\in\mathfrak{t}_L^*$ is a dominant and regular character of $\mathfrak{t}_L$.
\begin{itemize}
\item[(i)] For any coherent $\Da^{\dag}_{\mathfrak{Y},k}(\lambda)$-module $\Ea$ and for all $q>0$ one has $H^q(\mathfrak{Y},\Ea)=0$.
\item[(ii)] The functor $H^0(\mathfrak{Y},\bullet)$ is an equivalence between the category of coherent $\Da^{\dag}_{\mathfrak{Y},k}(\lambda)$-modules and the category of finitely presented $D^{\dag}(\G(k))_{\lambda}$-modules.
\end{itemize}
The same statement holds for coherent modules over $\widehat{\Da}^{(m,k)}_{\mathfrak{Y},\Q}(\lambda)$.
\end{theo}

\begin{proof}
The first part of the theorem follows from the fact that $H^0(\mathfrak{Y},\bullet)=H^0(\mathfrak{X},\bullet)\; \circ\; \pi_*$. Now we only have to apply the preceding theorem and theorem \ref{BB_for_dag}. 
\justify
Let us denote by $\Lb oc^{\dag}_{\mathfrak{Y},k}(\lambda)$ the exact functor defined by the composition 
\begin{eqnarray*}
\text{Finitely presented}\;D^{\dag}(\G(k))_{\lambda}-\text{modules}\xrightarrow{\La oc^{\dag}_{\mathfrak{X},k}(\lambda)} \text{Coherent}\; \Da^{\dag}_{\mathfrak{X},k}(\lambda)-\text{modules}\xrightarrow{\pi^{!}}\;\text{Coherent} \Da^{\dag}_{\mathfrak{Y},k}(\lambda)-\text{modules}.
\end{eqnarray*}
Let us compute this functor. To do that, we may fix a finitely presented $D^{\dag}(\G(k))_{\lambda}$-module $E$. Then
\begin{eqnarray*}
\pi^!\left(\La oc^{\dag}_{\mathfrak{X},k}(\lambda)(E)\right)=\Da^{\dag}_{\mathfrak{Y},k}(\lambda)\otimes_{\pi^{-1}\Da^{\dag}_{\mathfrak{X},k}(\lambda)}\pi^{-1}\Da^{\dag}_{\mathfrak{X},k}(\lambda)\otimes_{D^{\dag}(\G(k))_{\lambda}}E = \La oc^{\dag}_{\mathfrak{Y}}(\lambda)(E). 
\end{eqnarray*}
Now, to show that 
\begin{eqnarray*}
H^0\left(\mathfrak{Y},\pi^!\left(\La oc_{\mathfrak{X},k}^{\dag}(\lambda)(E)\right)\right)= H^0\left(\mathfrak{Y},\Da^{\dag}_{\mathfrak{Y},k}(\lambda)\otimes_{D^{\dag}(\G(k))_{\lambda}}E\right)= E,
\end{eqnarray*}
we can take a resolution
\begin{eqnarray*}
\left(D^{\dag}(\G(k))_{\lambda}\right)^{\oplus b}\rightarrow \left(D^{\dag}(\G(k))_{\lambda}\right)^{\oplus a}\rightarrow E\rightarrow 0,
\end{eqnarray*}
to get the following diagram
\begin{eqnarray*}
\begin{tikzcd}
\left(D^{\dag}(\G(k))_{\lambda}\right)^{\oplus b} \arrow[d] \arrow[r] 
& \left(D^{\dag}(\G(k))_{\lambda}\right)^{\oplus a} \arrow[d] \arrow[r] 
& E \arrow[d] \arrow[r] 
& 0\\
\left(D^{\dag}(\G(k))_{\lambda}\right)^{\oplus b}  \arrow[r] 
& \left(D^{\dag}(\G(k))_{\lambda}\right)^{\oplus a}\arrow[r] 
& H^0\left(\mathfrak{Y},\Da^{\dag}_{\mathfrak{Y},k}(\lambda)\otimes_{D^{\dag}(\G(k))_{\lambda}}E\right) \arrow[r] 
& 0.
\end{tikzcd} 
\end{eqnarray*}
where the sequence on the top is clearly exact. By definition $\La oc^{\dag}_{\mathfrak{Y},k}(\lambda)(\bullet)$ is an exact functor and by $(i)$ the global section functor $H^0(\mathfrak{Y},\bullet)$ is also exact. This shows that the sequence at the bottom is also exact and we end the proof of the theorem. 
\end{proof}
\justify 
In the sequel we will denote by $G_0$ the compact locally $L$-analytic group $G_0:=\G(\mathfrak{o})$.

\subsection{Group action on blow-ups}\label{group_action_on_blow_up}
\justify
Let $\mathfrak{G}$ be the formal completion of the group $\mathfrak{o}$-scheme $\G$, along its special fiber $\G_{\mathbb{F}_p}:=\G\times_{\text{Spec}(\mathfrak{o})}\text{Spec}(\mathfrak{o}/\varpi)$. Let us denote by $\alpha: \mathfrak{X}\times_{\text{Spf}(\mathfrak{o})}\mathfrak{G}\rightarrow \mathfrak{X}$ the induced right $\mathfrak{G}$-action on the formal flag $\mathfrak{o}$-scheme $\mathfrak{X}$ (cf. subsection \ref{Diff_op_acting_bundle}). For every $g\in \mathfrak{G}(\mathfrak{o})=G_0$ we have an automorphism $\rho_g$ of $\mathfrak{X}$ given by
\begin{eqnarray*}
\rho_g : \mathfrak{X}=\mathfrak{X}\times_{\text{Spf}(\mathfrak{o})}\text{Spf}(\mathfrak{o})\xrightarrow{ id_{\mathfrak{X}}\times g}\mathfrak{X}\times_{\text{Spf}(\mathfrak{o})}\mathfrak{G}\xrightarrow{\alpha} \mathfrak{X}.
\end{eqnarray*} 
As $\mathfrak{G}$ acts on the right, we have the following relation 
\begin{eqnarray}\label{cocy_action}
\left(\rho_g\right)_*\left(\rho_{h}^{\natural}\right)\circ\rho_g^{\natural}=\rho_{hg}^{\natural}\;\;\;(g,h\in G_0).
\end{eqnarray}
Here $\rho^{\natural}_{g}:\Ob_{\mathfrak{X}}\rightarrow (\rho_g)_*\Ob_{\mathfrak{X}}$ denotes the comorphism of $\rho_g$.
\justify
Let $H\subseteq G_0$ be an open subgroup. We say that an open ideal sheaf $\Ia\subseteq \Ob_{\mathfrak{X}}$ is $H$-\textit{stable} if for all $g\in H$ the comorphism $\rho^{\natural}_g$ maps $\Ia\subseteq \Ob_{\mathfrak{X}}$ into $(\rho_g)_*\Ia\subseteq (\rho_g)_*\Ob_{\mathfrak{X}}$. In this case $\rho_g^{\natural}$ induces a morphism of sheaves of graded rings 
\begin{eqnarray*}
\bigoplus_{d\in\N}\Ia^d\rightarrow \left(\rho_g\right)_*\left(\bigoplus_{d\in\N}\Ia^{d}\right)
\end{eqnarray*}
on $\mathfrak{X}$. This morphism of sheaves induces an automorphism of the blow-up $\mathfrak{Y}=\text{\textbf{Proj}}\left(\oplus_{d\in\N}\Ia^d\right)$, let us say $\rho_g$ by abuse of notation, and the action of $H$ on $\mathfrak{X}$ lifts to a right action of $H$ on $\mathfrak{Y}$, in the sense that for every $g,h\in G_0$ the relation (\ref{cocy_action}) is verified and we have a commutative diagram
\begin{eqnarray}\label{Lifted_action}
\begin{tikzcd}
\mathfrak{Y} \arrow[r, "\rho_g"] \arrow[d, "\text{pr}"]
& \mathfrak{Y} \arrow[d, "\text{pr}"]\\
\mathfrak{X} \arrow[r,"\rho_g"]
& \mathfrak{X}.
\end{tikzcd}
\end{eqnarray} 

\begin{defi}\label{Action_blow-up_subgroup}
Let $H\subseteq G_0$ be an open subgroup and $\text{pr}:\mathfrak{Y}\rightarrow \mathfrak{X}$ and admissible blow-up defined by an open ideal subsheaf $\Ia\subset\Ob_{\mathfrak{X}}$. We say that $\mathfrak{Y}$ is $H$-equivariant if $\Ia$ is $H$-stable.
\end{defi}
\justify
We will need the following result in the next sections. The reader can find its proof in \cite[Lemma 5.2.3]{HPSS}.

\begin{lem}\label{trivial_on_k+1}
Let $pr:\mathfrak{Y}\rightarrow\mathfrak{X}$ be an admissible blow-up, and let us assume that $k\ge k_Y=k_{\mathfrak{Y}}$. Then $\mathfrak{Y}$ is $G_k=\G(k)(\mathfrak{o})$-equivariant and the induced action of every $g\in G_{k+1}$ on the special fiber of $Y$ is the identity. Therefore, $G_{k+1}$ acts trivially on the underlying topological space of $\mathfrak{Y}$.
\end{lem}
\justify
By proposition \ref{Action_line_bundle} (cf. \cite[3.3.2]{HS1}), for every $g\in\mathfrak{G}(\mathfrak{o})=\G(\mathfrak{o})=G_0$ there exists an isomorphism 
\begin{eqnarray*}
\rho_g: \mathfrak{X}\xrightarrow{id_{\mathfrak{X}}\times g} \mathfrak{X}\times_{\text{Spec}(\mathfrak{o})}\mathfrak{X}\xrightarrow{\alpha}\mathfrak{X},
\end{eqnarray*}
\justify
which induces an $\Ob_{\mathfrak{X}}$-linear isomorphism $\Phi_g:\La (\lambda)\rightarrow (\rho_g)_{*}(\La (\lambda))$ verifying the cocycle condition
\begin{eqnarray}\label{cocycle_cond}
\Phi_{hg}=(\rho_{g})_*(\Phi_{h})\;\circ\;\Phi_{g}\;\;\;\text{and}\;\;\; (g,h\in\G(\mathfrak{o})).
\end{eqnarray}
\justify
In particular, we have an induced $G_0$-action on the sheaf $\Da^{\dag}_{\mathfrak{X},k}(\lambda)$
\begin{eqnarray}\label{Adjoint}
T_g: \Da^{\dag}_{\mathfrak{X},k}(\lambda)\rightarrow (\rho_g)_*\Da^{\dag}_{\mathfrak{X},k}(\lambda),\;\;\; P\mapsto \Phi_g\;\circ\; P\;\circ\; (\Phi_g)^{-1}.
\end{eqnarray}
\justify
Locally, if $\Ub\subseteq \mathfrak{X}$ and $P\in\Da^{\dag}_{\mathfrak{X},k}(\lambda)(\Ub)$ then the cocycle condition (\ref{cocycle_cond}) tells that the diagram 
\begin{eqnarray}\label{diag_cocycle}
\begin{tikzcd}  [column sep=5cm]
\La(\lambda)(\Ub . (hg)^{-1})=\La (\lambda)(\Ub .g^{-1}h^{-1}) \arrow [r, "T_{gh,\;(\Ub}(P)"] \arrow[d, "\Phi^{-1}_{h,\;\Ub . g^{-1}}\; = \; (\rho_g)_*\Phi^{-1}_{h,\; \Ub}"]
& \La (\lambda)(\Ub .g^{-1}h^{-1}) \\
\La (\lambda)(\Ub . g^{-1}) \arrow[d, "\Phi^{-1}_{g,\;\Ub}"] 
& \La (\Ub . g^{-1})  \arrow[u,"\Phi_{h,\;\Ub . g^{-1}}\; = \; (\rho_g)_*\Phi_{h,\;\Ub}"]\\
\La (\lambda)(\Ub) \arrow[r, "P"]
& \La (\lambda)(\Ub)  \arrow[u, "\Phi_{g,\;\Ub}"]
\end{tikzcd}
\end{eqnarray}
is commutative and we get the relation
\begin{eqnarray}\label{Cocycle_Adjoint}
T_{hg} = \left(\rho_g\right)_*T_h\;\circ\;T_g\;\;\; (g,h\in G_0).
\end{eqnarray}
\justify
 Let us suppose that $H\subseteq G_0$ is an open subgroup and that $\text{pr}:\mathfrak{Y}\rightarrow\mathfrak{X}$ is an $H$-equivariant admissible blow-up. Pulling back the isomorphism $(\rho_g)^{*}\La(\lambda)\rightarrow \La(\lambda)$, via $(\text{pr})^*$, and using the commutative diagram (\ref{Lifted_action}) we get $\text{pr}^*(\rho_g)^* \La(\lambda)= (\rho_g)^{*}\text{pr}^{*}\La(\lambda)= (\rho_g)^*\underline{\La(\lambda)}$ (notation given at the beginning of the preceding subsection). By adjontion we get the map
\begin{eqnarray*}
\text{R}_g: \underline{\La(\lambda)}\xrightarrow{\simeq}\left(\rho_g\right)_*\underline{\La(\lambda)}
\end{eqnarray*}
which satisfies, by functoriality, the cocycle condition
\begin{eqnarray} \label{cocycle_bundle_blow_up}
R_{hg}=\left(\rho_g\right)_*R_h\;\circ\;R_g\;\;\;\; (g,h\in H).
\end{eqnarray} 
As in (\ref{Adjoint}) we can define (from now on we will work on admissible blow-ups of $\mathfrak{Y}$ so we will use the same notation)
\begin{eqnarray}\label{Adjoint_blow_up}
T_g : \Da^{\dag}_{\mathfrak{Y},k}(\lambda) \rightarrow  \left(\rho_g\right)_*\Da^{\dag}_{\mathfrak{Y},k}(\lambda);\hspace{0.3 cm}
P \rightarrow & R_g\;\circ\;P\;\circ\;R_g^{-1}
\end{eqnarray}
and exactly as we have done in (\ref{diag_cocycle}) we can conclude that 
\begin{eqnarray*}
T_{hg}= \left(\rho_g\right)_* T_h\;\circ\; T_g,
\end{eqnarray*}
for every $g,h \in H$.

\section{Localization of locally analytic representations}
\justify
We recall for the reader that $G_0$ denotes the compact locally $L$-analytic group $G_0=\G(\mathfrak{o})$. In this section we will show how to localize admissible locally analytic representations of $G_0$. We will denote by $\Cb^{\text{an}}(G_0,L)$  the space of $L$-valued locally $L$-analytic functions on $G_0$  and by $D(G_0,L)$ its strong dual (the space of \textit{locally analytic distributions} in the sense of \cite[Section 11]{ST}). This space contains a set of delta distributions $\{\delta_g\}_{g\in G_0}$ defined by $\delta_g(f)=f(g)$, if $f\in\Cb^{\text{an}}(G_0,L)$, in such a way that the map $g\mapsto\delta_g$ is an injective group homomorphism from $G_0$ into $D(G_0,L)^{\times}$. We also recall that given that $G_0$ is compact, this space carries a structure of nuclear Fréchet-Stein algebra \cite[Theorem 24.1]{ST}. To our work it will be enough to define a weak Fréchet-Stein structure (in the sense of \cite[Definition 1.2.8]{Emerton2}) on the algebra $D(G_0,L)$. 
\justify
We finally recall that in (\ref{Analytic_algebra}) we have introduced Emerton's distribution algebra as the continuous dual space of the space of rigid-analytic functions on $\G(k)^{\circ}$
\begin{eqnarray*}
\Db^{\text{an}}(\G(k)^{\circ}):=\text{Hom}^{\text{Cont}}_{L}\left(\Ob_{\G(k)^{\circ}}(\G(k)^{\circ}),L\right).
\end{eqnarray*}

\subsection{Coadmissible modules}
\justify
let us start by recalling that $G_0$ acts on the space $\Cb^{\text{cts}}(G_0,L)$, of continuous $L$-valued functions, by the formula 
\begin{eqnarray*}
(g\bullet f)(x):=f(g^{-1}x)\;\;\; (g,x\in G_0,\; f\in \Cb^{\text{cts}}(G_0,L)).
\end{eqnarray*}
Moreover, given an admissible locally analytic representation $V$ of $G_0$ \cite[First definition of lecture VI]{ST} then, by definition, its strong dual $M:=(V)'_b$ is a coadmissible module over $D(G_0,L)$\footnote{We recall for the reader, that the category of coadmisisble $D(G_0,L)$-modules is a full abelian subcategory of the category of $D(G_0,L)$-modules and the "strong dual" functor induces an anti-equivalence of categories to the category of admissible locally analytic representations \cite[Theorem 20.1]{ST}.}.
\justify
Given a continuous representation $W$ of $G_0$, we can consider the subspace $W_{\G(k)^{\circ}}\subseteq W$ of $\G(k)^{\circ}$-analytic vectors \cite[3.4.1]{Emerton2}. In particular, the $G_0$-action on $\Cb^{\text{cts}}(G_0,L)$, defined at the beginning of this subsection, gives us 
\begin{eqnarray}\label{locally_and_vectors}
\varinjlim_{k}\Cb^{\text{cts}}(G_0,L)_{\G(k)^{\circ}-{\text{an}}} \xrightarrow{\simeq} \Cb^{\text{an}}(G_0,L).
\end{eqnarray} 
\justify
As in \cite[Proposition 5.3.1]{Emerton2}, for each $k\in\Z_{>0}$ we put 
\begin{eqnarray*}
D(\G(k)^{\circ},G_0):=\left(\Cb^{\text{cts}}(G_0,L)_{\G(k)^{\circ}-\text{an}}\right)'_b
\end{eqnarray*}
the strong dual of the space of  $\G(k)^{\circ}$-analytic vectors of $\Cb^{\text{cts}}(G_0,L)$ \cite[3.4.1]{Emerton2}. The ring structure on $\Db^{\text{an}}(\G(k)^{\circ})$ extends naturally to a ring structure on $D(\G(k)^{\circ},G_0)$, such that
\begin{eqnarray}\label{Iso_delta_g}
 D(\G(k)^{\circ},G_0)=\displaystyle\bigoplus_{g\in G_0/G_k}\Db^{\text{an}}(\G(k)^{\circ})\delta_g.
\end{eqnarray}
Dualizing the isomorphism (\ref{locally_and_vectors}) yields an isomorphism of topological $L$-algebras 
\begin{eqnarray}\label{Frechet_structure}
D(G_0,L)\xrightarrow{\simeq} \varprojlim_{k\in\Z_{>0}} D(\G(k)^{\circ},G_0).
\end{eqnarray}
\justify
This is the weak Fréchet-Stein structure on the locally analytic distribution algebra $D(G_0,L)$ (\cite[Proposition 5.3.1]{Emerton2}).
\justify
Let $V$ be an admissible locally analytic representation and $M:= V'_b$. By \cite[Lemma 6.1.6]{Emerton2} the subspace $V_{\G(k)^{\circ}-\text{an}}\subseteq V$ is a nuclear Fréchet space and therefore its strong dual $M_k:=\left(V_{\G(k)^{\circ}-\text{an}}\right)'_b$ is a space of compact type and a finitely generated topological $D(\G(k)^{\circ},G_0)$-module by \cite[Lemma 6.1.13]{Emerton2}. By \cite[Theorem 6.1.20]{Emerton2} the module $M$ is a coadmissible $D(G_0,L)$-module relative to the weak Fréchet -Stein structure of $D(G_0,L)$ defined in the previous paragraph. 
\justify
We have the following result from \cite[Lemma 5.1.7]{HPSS}.

\begin{lem}\label{linear_extension}
\begin{itemize}
\item[(i)] The $D(\G(k)^{\circ},G_0)$-module $M_k$ is finitely generated.
\item[(ii)] There are natural isomorphisms
\begin{eqnarray*}
D(\G(k-1)^{\circ},G_0)\otimes_{D(\G(k)^{\circ},G_0)}M_k\xrightarrow{\simeq} M_{k-1}.
\end{eqnarray*}
\item[(iii)] The natural map $D(\G(k-1)^{\circ},G_0)\otimes_{D(G_0,L)}M\rightarrow M_k$ is bijective.
\end{itemize}
\end{lem}
\justify
Now, let $\lambda\in\text{Hom}(\T,\G_m)$ be an algebraic character such that $\lambda+\rho+\mathfrak{t}^*_L$ is a dominant and regular character of $\mathfrak{t}_L$. Let us recall that we have identifications
\begin{eqnarray*}
\Db^{\text{an}}(\G(k)^{\circ})_{\lambda} =  D^{\dag}(\G(k))_{\lambda} =  \varinjlim_{m\in\N} \left(\widehat{D}^{(m)}(\G(k))_{\lambda}\right)\otimes_{\mathfrak{o}}L.
\end{eqnarray*}
The preceding relation and the fact that the ring structure of $\Db^{\text{an}}(\G(k)^{\circ})$ extends naturally to a ring structure on $D(\G(k)^{\circ},G_0)$ allow us to consider the ring 
\begin{eqnarray*}
D(\G(k)^{\circ},G_0)_{\lambda}:=D(\G(k)^{\circ},G_0)/\text{Ker}(\chi_{\lambda})D(\G(k)^{\circ},G_0).
\end{eqnarray*}
\justify
From now on, we will denote $\Cb_{G_0}$ the full subcategory of $\text{Mod}(D(\G_0,L))$ consisting of coadmissible modules, with respect to the preceding weak Fréchet-Stein structure on $D(G_0,L)$.

\begin{defi}
We define the category $\Cb_{G_0,\lambda}$ of coadmissible $D(G_0,L)$-modules with central character $\lambda\in\text{Hom}(\T,\G_m)$ by
\begin{eqnarray*}
\Cb_{G_0,\lambda}:=Mod\left(D(G_0,L)\big /Ker(\chi_{\lambda})D(G_0,L)\right)\cap \Cb_{G_0}.
\end{eqnarray*}
\end{defi}
\justify
We point out that the preceding definition is completely legal because the center $Z(\mathfrak{g}_L)$ of the universal enveloping algebra $\Ub(\mathfrak{g}_L)$ lies in the center of $D(G_0,L)$ \cite[Proposition 3.7]{ST1}. We also recall that the group $G_k:=\G(k)(\mathfrak{o})$ is contained in $\Db^{\text{an}}(\G(k)^{\circ})$ as a set of Dirac distributions. For each $g\in G_k$ we will write $\delta_g$ for the image of the Dirac distribution supported at $g$ in 
\begin{eqnarray*}
H^0\left(\mathfrak{Y},\Da^{\dag}_{\mathfrak{Y},k}(\lambda)\right)=\Db^{\text{an}}(\G(k)^{\circ})_{\lambda}.
\end{eqnarray*}
\justify
Inspired in \cite[Definition 5.2.7]{HPSS} we have the following definition. 

\begin{defi}\label{strongly_G_0_equivariant}
Let $H\subset G_0$ be an open subgroup and $\mathfrak{Y}$ an $H$-equivariant admissible blow-up of $\mathfrak{X}$. Let us suppose that $k\ge k_{\mathfrak{Y}}$ (notation as in \ref{notation_k_y}). A strongly $H$-equivariant $\Da^{\dag}_{\mathfrak{Y},k}(\lambda)$-module is a $\Da^{\dag}_{\mathfrak{Y},k}(\lambda)$-module $\Ma$ together with a family $(\varphi_{g})_{g\in H}$ of isomorphisms
\begin{eqnarray*}
\varphi_g:\Ma\rightarrow (\rho_g)_{*}\Ma
\end{eqnarray*}
of sheaves of $L$-vector spaces, satisfying the following conditions:
\begin{itemize}
\item[(i)] For all $g,h\in H$ we have $\left(\rho_g\right)_*\left(\varphi_h\right)\circ\varphi_g=\varphi_{hg}$.
\item[(ii)] For all open subset $\Ub\subset\mathfrak{Y}$, all $P\in\Da^{\dag}_{\mathfrak{Y},k}(\lambda)(\Ub)$, and all $m\in\Ma(\Ub)$ we have $\varphi_{g}(P\bullet m)=T_g(P)\bullet \varphi_{g}(m)$.
\item[(iii)]\footnote{This conditions makes sense because the elements $g\in G_{k+1}$ acts trivially on the underlying topological space of $\mathfrak{Y}$, cf. Lemma \ref{trivial_on_k+1}.} For all $g\in H\cap G_{k+1}$ the map $\varphi_g : \Ma\rightarrow \left(\rho_g\right)_*\Ma=\Ma$ is equal to multiplication by $\delta_g\in H^0\left(\mathfrak{Y},\Da^{\dag}_{\mathfrak{Y},k}(\lambda)\right)$.
\end{itemize}
\end{defi}
\justify
A morphism between two strongly $H$-equivariant $\Da^{\dag}_{\mathfrak{Y},k}(\lambda)$-modules $(\Ma, (\varphi_g^{\Ma})_{g\in H})$ and $(\Na, (\varphi_g^{\Na})_{g\in H})$ is a $\Da^{\dag}_{\mathfrak{Y},k}(\lambda)$ linear morphism $\psi:\Ma\rightarrow\Na$ such that for all $g\in H$
\begin{eqnarray*}
\varphi_{g}^{\Na}\;\circ\;\psi=(\rho_g)_*(\psi)\;\circ\;\varphi_g^{\Ma}.
\end{eqnarray*}
\justify
We denote the category of strongly $H$-equivariant coherent $\Da^{\dag}_{\mathfrak{Y},k}(\lambda)$-modules by $\text{Coh}\big(\Da^{\dag}_{\mathfrak{Y},k}(\lambda),G_0\big)$.

\begin{comm}\label{Commentary_action} Let $\Ma \in \text{Coh}\big(\Da^{\dag}_{\mathfrak{Y},k}(\lambda),G_0\big)$. In what follows we will use the notation $gm := \varphi_{g,\;\Ub}(m)\in \Ma(\Ub. g^{-1})$, for $\Ub\subseteq\mathfrak{Y}$ an open subset, $g\in G_0$ and $m\in\Ma (\Ub)$. This notation is inspired in property $(ii)$ of the previous definition. In fact, if $g,h\in G_0$, then by $(ii)$ we have $h(g\;m)=(hg)\; m$.
\end{comm}

\begin{theo}\label{First_equivalence}
Let $\lambda\in\text{Hom}(\T,\G_m)$ be an algebraic character such that $\lambda+\rho\in\mathfrak{t}_L^*$ is a dominant and regular character of $\mathfrak{t}_L$. Let $\text{pr}:\mathfrak{Y}\rightarrow\mathfrak{X}$ be a $G_0$-equivariant admissible blow-up, and let $k\ge k_{\mathfrak{Y}}$. The functors $\La oc^{\dag}_{\mathfrak{Y},k}(\lambda)$ and $H^{0}(\mathfrak{Y},\bullet)$ induce quasi-inverse equivalences between the category of finitely presented $D(\G(k)^{\circ},G_0)_{\lambda}$-modules and $\text{Coh}\big(\Da^{\dag}_{\mathfrak{Y},k}(\lambda),G_0\big)$.
\end{theo}
\justify
Before starting the proof, we recall for the reader that the functor $\La oc^{\dag}_{\mathfrak{Y},k}(\lambda)$ has been defined in the proof of theorem \ref{BB_blow_up}. An explicitly expression is given in (\ref{Rappel_Loc_blow-up}) below.

\begin{proof}
If $\Ma\in \text{Coh}\big(\Da^{\dag}_{\mathfrak{Y},k}(\lambda),G_0\big)$, then in particular $\Ma$ is a coherent $\Da^{\dag}_{\mathfrak{Y},k}(\lambda)$-module. Since by corollary \ref{equal_global_sections} and  theorem \ref{BB_blow_up} we have that $H^0(\mathfrak{Y},\Ma)$ is a finitely presented $\Db^{\text{an}}(\G(k)^{\circ})_{\lambda}$-module, then by (\ref{Iso_delta_g}) we can conclude that $H^{0}(\mathfrak{Y},\Ma)$ is a finitely presented $D(\G(k)^{\circ},\G_0)_{\lambda}$-module.
\justify
On the other hand, let us suppose that $M$ is a finitely presented $D(\G(k)^{\circ},\G_0)_{\lambda}$-module. By (\ref{Iso_delta_g}) we can consider 
\begin{eqnarray}\label{Rappel_Loc_blow-up}
\Ma := \La oc^{\dag}_{\mathfrak{Y},k}(\lambda)(M)=\Da^{\dag}_{\mathfrak{Y},k}(\lambda)\otimes_{\Db^{\text{an}}(\G(k)^{\circ})_{\lambda}}M.
\end{eqnarray} 
For every $g\in G_{0}$ we want to define an isomorphism of sheaves of $L$-vector spaces
\begin{eqnarray*}
\varphi_g : \Ma\rightarrow\left(\rho_g\right)_*\Ma
\end{eqnarray*}
satisfying the conditions $(i)$, $(ii)$ and $(iii)$ in the preceding definition. As we have remarked, the Dirac distributions induce an injective morphism from $G_0$ to the group of units of $D(G_0,L)$, since by (\ref{Frechet_structure}) $M$ is in particular a $G_0$-module, we have an isomorphism
\begin{eqnarray*}
\Ma \rightarrow \left(\left(\rho_g\right)_*\Da^{\dag}_{\mathfrak{Y},k}(\lambda)\right)\otimes_{\Db^{\text{an}}(\G(k)^{\circ})_{\lambda}}M,
\end{eqnarray*}
which on local sections is defined by $\varphi_{g,\;\Ub}(P\otimes m):=T_{g,\;\Ub}(P)\otimes g m$. Here $P\in\Da^{\dag}_{\mathfrak{Y},k}(\lambda)(\Ub)$, $\Ub\subseteq\mathfrak{Y}$ is an open subset, $m\in M$ and $T_g$ is the isomorphism defined in (\ref{Adjoint_blow_up}).
\justify
One has an isomorphism
\begin{eqnarray*}
\left(\rho_g\right)_*\left(\Ma\right)\xrightarrow{\simeq}\left(\left(\rho_g\right)_*\Da^{\dag}_{\mathfrak{Y},k}(\lambda)\right)\otimes_{\Db^{\text{an}}(\G(k)^{\circ})_{\lambda}}M.
\end{eqnarray*}
Indeed, $(\rho_g)_*$ is exact and so choosing a finite presentation of $M$ as $\Db^{\text{an}}(\G(k)^{\circ})_{\lambda}$-module reduces to the case $M=\Db^{\text{an}}(\G(k)^{\circ})_{\lambda}$ which is trivially true. This implies that the preceding isomorphism extends to an isomorphism
\begin{eqnarray*}
\varphi_g: \Ma\rightarrow \left(\rho_g\right)_* \Ma.
\end{eqnarray*}
Let $g,h\in G_0$, $\Ub\subseteq\mathfrak{Y}$ an open subset, $P,Q\in\Da^{\dag}_{\mathfrak{Y},k}(\lambda)(\Ub)$ and $m\in M$. Then
\begin{eqnarray*}
\varphi_{h,\; \Ub.g^{-1}}\left(\varphi_{g,\;\Ub}\right)(P\otimes m)  
= T_{h,\; \Ub.g^{-1}}(T_{g,\;\Ub}(P))\otimes hg\; m  
= T_{hg,\;\Ub}(P)\otimes (hg)\;m 
 = \varphi_{hg,\;\Ub}(P\otimes m),
\end{eqnarray*}
which verifies the first condition. Now, by definition $T_{g,\;\Ub}(PQ)=T_{g,\;\Ub}(P)T_{g,\;\Ub}(Q)$ and therefore $\varphi_{g,\;\Ub}(PQ\otimes m) = T_{g,\;\Ub}(P)\varphi_{g,\;\Ub}(Q\otimes m)$, which gives (ii). Finally, given that the delta distributions $\delta_g$ for $g$ in the normal subgroup $G_{k+1}$ of $G_0$ are contained in $\Db^{\text{an}}(\G(k)^{\circ})$ we have $g.P := T_{g}(P) = \delta_g\; P \; \delta_{g^{-1}}$, and therefore
\begin{eqnarray*}
\varphi_{g,\;\Ub}(P\otimes m) & = g.P\otimes\; g.m
= \delta_{g}P\delta_{g^{-1}}\delta_g\otimes \;m
= \delta_g P\otimes\; m.
\end{eqnarray*}
and condition $(iii)$  follows.
\end{proof}

\begin{rem}
If $\lambda\in\text{Hom}(\T,\G_m)$ denotes the trivial character, then $\Da^{\dag}_{\mathfrak{X},k}(\lambda)=\Da^{\dag}_{\mathfrak{X},k}$ is the sheaf of arithmetic differential operators introduced in \cite{HPSS}. Moreover, by construction, if $\text{pr}:\mathfrak{Y}\rightarrow\mathfrak{X}$ denotes an $H$-equivariant admissible blow-up, then $\Da^{\dag}_{\mathfrak{Y},k}(\lambda)=\Da^{\dag}_{\mathfrak{Y},k}$ and for every $g\in H$ the isomorphism $T_{g}$ equals the isomorphism $Ad(g)$ defined in \cite[(5.2.6)]{HPSS}.
\end{rem}
\justify
Now, let us take $\pi:\mathfrak{Y}'\rightarrow\mathfrak{Y}$ a morphism of $G_0$-equivariant admissible blow-ups of $\mathfrak{X}$ (whose lifted actions we denote by $\rho^{\mathfrak{Y}'}$ and $\rho^{\mathfrak{Y}}$), and let us suppose that $k\ge k_{\mathfrak{Y}}$ and $k'\ge\text{max}\{k_{\mathfrak{Y}}',k\}$. By (\ref{rel_morph_blow_up}) and theorem \ref{Invariance} we have an injective morphism of sheaves 
\begin{eqnarray}\label{from_k'_to_k}
\Psi:\pi_* \Da^{\dag}_{\mathfrak{Y}',k'}(\lambda)=\Da^{\dag}_{\mathfrak{Y},k'}(\lambda)\hookrightarrow \Da^{\dag}_{\mathfrak{Y},k}(\lambda).
\end{eqnarray}
Moreover, if $g\in G_0$ we have  
\begin{eqnarray*}
T_{g}^{\mathfrak{Y}}\; \circ\; \Psi = \left(\rho^{\mathfrak{Y}}_g\right)_*\left(\Psi\right)\; \circ\; \pi_*\left(T_{g}^{\mathfrak{Y}'}\right)
\end{eqnarray*}
and therefore $\Psi$ is $G_0$-equivariant. Now, let us consider $\Ma_{\mathfrak{Y}'}\in\text{Coh}\left(\Da^{\dag}_{\mathfrak{Y}',k'}(\lambda),G_0\right)$ and $\Ma_{\mathfrak{Y}}\in\text{Coh}\left(\Da^{\dag}_{\mathfrak{Y},k}(\lambda),G_0\right)$ together with a morphism $\psi:\pi_*\Ma_{\mathfrak{Y}'}\rightarrow \Ma_{\mathfrak{Y}}$ linear relative to $\Psi: \pi_*\Da^{\dag}_{\mathfrak{Y}',k'}(\lambda)\hookrightarrow\Da^{\dag}_{\mathfrak{Y},k}(\lambda)$ and which is $G_0$-equivariant, i.e. satisfying 
\begin{eqnarray*}
\varphi^{\Ma_{\mathfrak{Y}}}_g\;\circ\;\psi = \left(\rho^{\mathfrak{Y}}_g\right)_*\psi\; \circ\; \pi_*\left(\varphi^{\Ma_{\mathfrak{Y}'}}_g\right)
\end{eqnarray*}
for all $g\in G_0$. By using $\Psi$ we obtain a morphism of $\Da^{\dag}_{\mathfrak{Y},k}(\lambda)$-modules
\begin{eqnarray*}
\Da^{\dag}_{\mathfrak{Y},k}(\lambda)\otimes_{\pi_*\Da^{\dag}_{\mathfrak{Y}',k'}(\lambda)}\pi_*\Ma_{\mathfrak{Y}'}\rightarrow \Ma_{\mathfrak{Y}}.
\end{eqnarray*}
Let us denote by $\Ka$ the submodule of $\Da^{\dag}_{\mathfrak{Y},k}(\lambda)\otimes_{\pi_*\Da^{\dag}_{\mathfrak{Y}',k'}(\lambda)}\pi_*\Ma_{\mathfrak{Y}'}$ locally generated by all the elements of the form $P\delta_h\otimes m-P\otimes(h\bullet m)$, where $h\in G_{k+1}$, $m$ is a local section of $\pi_*\Ma_{\mathfrak{Y}'}$ and $P$ is a local section of $\Da^{\dag}_{\mathfrak{Y},k}(\lambda)$. As in \cite[Page 35]{HPSS} we will denote the quotient $\Da^{\dag}_{\mathfrak{Y},k}(\lambda)\otimes_{\pi_*\Da^{\dag}_{\mathfrak{Y}',k'}(\lambda)}\pi_*\Ma_{\mathfrak{Y}'}\big /\Ka$ by
\begin{eqnarray}\label{G_0-equi-quotient}
\Da^{\dag}_{\mathfrak{Y},k}(\lambda)\otimes_{\pi_*\Da^{\dag}_{\mathfrak{Y}',k'}(\lambda),\; G_{k+1}}\pi_*\Ma_{\mathfrak{Y}'}.
\end{eqnarray}
Let us see that this module lies in $\text{Coh}\big (\Da^{\dag}_{\mathfrak{Y},k}(\lambda),G_0\big )$. To do that let us first show that 
\begin{eqnarray}\label{iso_diagonal_act}
\left(\rho_g\right)_*\Da^{\dag}_{\mathfrak{Y},k}(\lambda)\otimes_{(\rho_g)_*\pi_{*}\Da^{\dag}_{\mathfrak{Y}',k'}(\lambda)}\left(\rho_g\right)_*\pi_*\Ma_{\mathfrak{Y}'} = \left(\rho_g\right)_*\left( \Da^{\dag}_{\mathfrak{Y},k}(\lambda)\otimes_{\pi_*\Da^{\dag}_{\mathfrak{Y}',k'}(\lambda)}\pi_*\Ma_{\mathfrak{Y}'}\right).
\end{eqnarray}
As $\Ma_{\mathfrak{Y}'}$ is a coherent $\Da^{\dag}_{\mathfrak{Y}',k'}(\lambda)$, by proposition \ref{generated_by_global_sect} and theorem \ref{Invariance} we can find a finite presentation of $\Ma_{\mathfrak{Y}'}$
\begin{eqnarray*}
\left(\Da^{\dag}_{\mathfrak{Y}',k'}(\lambda)\right)^{\oplus a} \rightarrow (\Da^{\dag}_{\mathfrak{Y}',k'}(\lambda))^{\oplus b} \rightarrow \Ma_{\mathfrak{Y}'}\rightarrow 0
\end{eqnarray*}
which induces, by exactness of $(\rho_g)_*$ and $\pi_*$ (theorem \ref{Invariance}.), the exact sequence 
\begin{eqnarray*}
\left(\left(\rho_g\right)_*\Da^{\dag}_{\mathfrak{Y},k'}(\lambda)\right)^{\oplus a}\rightarrow \left(\left(\rho_g\right)_*\Da^{\dag}_{\mathfrak{Y},k'}(\lambda)\right)^{\oplus b} \rightarrow \left(\rho_g\right)_*\pi_*\Ma_{\mathfrak{Y}'}\rightarrow 0.
\end{eqnarray*}
By base change over the preceding exact sequence we obtain the following commutative diagram
\begin{eqnarray*}
\begin{tikzcd}
\left(\left(\rho_g\right)_*\Da^{\dag}_{\mathfrak{Y},k}(\lambda)\right)^{\oplus a} \arrow[r] \arrow[d, "id"]
& \left(\left(\rho_g\right)_*\Da^{\dag}_{\mathfrak{Y},k}(\lambda)\right)^{\oplus b} \arrow[r] \arrow[d, "id"]
&  \left(\rho_g\right)_*\left(\Da^{\dag}_{\mathfrak{Y},k}(\lambda)\otimes_{\pi_*\Da^{\dag}_{\mathfrak{Y}',k'}(\lambda)}\pi_*\Ma_{\mathfrak{Y}'}\right) \arrow[d]  \arrow[r]
& 0
\\
\left(\left(\rho_g\right)_*\Da^{\dag}_{\mathfrak{Y},k}(\lambda)\right)^{\oplus a} \arrow[r]
& \left(\left(\rho_g\right)_*\Da^{\dag}_{\mathfrak{Y},k}(\lambda)\right)^{\oplus b} \arrow[r] 
& \left(\rho_g\right)_*\Da^{\dag}_{\mathfrak{Y},k}(\lambda)\otimes_{(\rho_g)_*\pi_{*}\Da^{\dag}_{\mathfrak{Y}',k'}(\lambda)}\left(\rho_g\right)_*\pi_*\Ma_{\mathfrak{Y}'} \arrow[r]
& 0
\end{tikzcd}
\end{eqnarray*}
(of course, here we have used theorem \ref{Invariance} to identify $\pi_*\Da^{\dag}_{\mathfrak{Y}',k'}(\lambda)=\Da^{\dag}_{\mathfrak{Y},k'}(\lambda)$). This shows (\ref{iso_diagonal_act}) and therefore we dispose of a diagonal action
\begin{eqnarray*}
\varphi_g : \Da^{\dag}_{\mathfrak{Y},k}(\lambda)\otimes_{\pi_{*}\Da^{\dag}_{\mathfrak{Y}',k'}(\lambda)}\pi_*\Ma_{\mathfrak{Y}'} \rightarrow \left(\rho_g\right)_*\left( \Da^{\dag}_{\mathfrak{Y},k}(\lambda)\otimes_{\pi_{*}\Da^{\dag}_{\mathfrak{Y}',k'}(\lambda)}\pi_*\Ma_{\mathfrak{Y}'}\right)
 \end{eqnarray*}
defined on simple tensor products by 
\begin{eqnarray} \label{act_quot}
g\bullet (P\otimes m):= g\bullet P\otimes g\bullet m,
\end{eqnarray}
for $g\in G_0$, and $P$ and $m$ local sections of $\Da^{\dag}_{\mathfrak{Y},k}(\lambda)$ and $\pi_*\Ma_{\mathfrak{Y}'}$, respectively (in order to soft the notation we use the accord introduced in the commentary \ref{Commentary_action} after the definition \ref{strongly_G_0_equivariant}). Now to see that (\ref{G_0-equi-quotient}) is a strongly $G_0$-equivariant $\Da^{\dag}_{\mathfrak{Y},k}(\lambda)$-module, we only need to check that $\varphi_g (\Ka)\subset \Ka$. This is, the diagonal action fix the submodule $\Ka$. We  have
\begin{align*}
g\bullet (P\delta_{h}\otimes m - P\otimes h\bullet m) & = g\bullet (P\delta_h)\otimes g\bullet m - g\bullet P \otimes g\bullet (h\bullet m)\\
& = (g\bullet P)(g\bullet \delta_h)\otimes g\bullet m - g\bullet P\otimes (ghg^{-1})\bullet (g\bullet m) \\
& = (g\bullet P)\delta_{ghg^{-1}}\otimes g\bullet m - g\bullet P\otimes (ghg^{-1})\bullet (g\bullet m),
\end{align*}
as $G_{k+1}$ is a normal subgroup we can conclude that $ghg^{-1}\in G_{k+1}$ and $G_0$ fix $\Ka$. 
Moreover, since the target of the preceding morphism is strongly $G_0$-equivariant, this factors through the quotient and we thus obtain a morphism of $\Da^{\dag}_{\mathfrak{Y},k}(\lambda)$-modules
\begin{eqnarray}\label{morph_for_equi}
\overline{\psi}:\Da^{\dag}_{\mathfrak{Y},k}(\lambda)\otimes_{\pi_*\Da^{\dag}_{\mathfrak{Y}',k'}(\lambda),\; G_{k+1}}\pi_*\Ma_{\mathfrak{Y}'}\rightarrow \Ma_{\mathfrak{Y}}.
\end{eqnarray}
By construction $\overline{\psi}\in\text{Coh}\big (\Da^{\dag}_{\mathfrak{Y},k}(\lambda),G_0\big )$. 

\section{Admissible blow-ups and formal models}\label{ABUFM}
\justify
The following discussion is given in \cite[3.1.1 and 5.2.13]{HPSS}. Let us start by considering the generic fiber $X_L:=X\times_{\text{Spec}(\mathfrak{o})}\text{Spec}(L)$ of the flag scheme $X$ (the flag \textit{variety}). For the rest of this work $X^{\text{rig}}$ will denote the rigid-analytic space associated via the GAGA functor to $X_L$ \cite[Part I, chapter 5, section 5.4, Definition and proposition 3]{Bosch}. Any admissible formal $\mathfrak{o}$-scheme $\mathfrak{Y}$ (in the sense of \cite[Part II, chapter 7, section 7.4, Definitions 1 and 4]{Bosch}) whose associated rigid-analytic space is isomorphic to $X^{\text{rig}}$ will be called a \text{formal model} of $\mathbb{X}^{\text{rig}}$.  For any two formal models $\mathfrak{Y}_1$ and $\mathfrak{Y}_2$ there exists a formal model $\mathfrak{Y}'$ and admissible formal blow-up morphisms $\mathfrak{Y}'\rightarrow\mathfrak{Y}_1$ and $\mathfrak{Y}'\rightarrow\mathfrak{Y}_2$ \cite[Part II, chapter 8, section 8.2, remark 10]{Bosch}.
\justify
Now, let us denote by $\Fb_{\mathfrak{X}}$ the set of  admissible formal blow-ups $\mathfrak{Y}\rightarrow\mathfrak{X}$. This set is ordered by $\mathfrak{Y}'\succeq \mathfrak{Y}$ if the blow-up morphism $\mathfrak{Y}'\rightarrow\mathfrak{X}$ factors as the composition of a morphism $\mathfrak{Y}'\rightarrow\mathfrak{Y}$ and the blow-up morphism $\mathfrak{Y}\rightarrow\mathfrak{X}$. In this case, the morphism $\mathfrak{Y}'\rightarrow\mathfrak{Y}$ is unique \cite[Part II, chapter 8, section 8.2, proposition 9]{Bosch}, and it is itself a blow-up morphism \cite[Chapter 8, section 8.1.3, proposition 1.12 (d) and theorem 1.24]{Liu}. By \cite[Part II, chapter 8, section 8.2, remark 10]{Bosch} the set $\Fb_{\mathfrak{X}}$ is directed and it is cofinal in the set of all formal models. Furthermore, any formal model $\mathfrak{Y}$ of $X^{\text{rig}}$ is dominated by one which is a $G_0$-equivariant admissible blow-up of $\mathfrak{X}$ \cite[Proposition 5.2.14]{HPSS}. In particular, if $\mathfrak{X}_{\infty}$ denotes the projective limit of all formal models of $\mathbb{X}^{\text{rig}}$, then
\begin{eqnarray*}
\mathfrak{X}_{\infty}= \varprojlim_{\Fb_{\mathfrak{X}}}\mathfrak{Y}.
\end{eqnarray*}
\justify
We will be interested in the following directed subset of $\Fb_{\mathfrak{X}}$.

\begin{defi}\label{Defi_F_underline}
We denote by $\underline{\Fb}_{\mathfrak{X}}$ the set of pairs $(\mathfrak{Y},k)$, where $\mathfrak{Y}\in\Fb_{\mathfrak{X}}$ and $k\in\N$ satisfies $k\ge k_{\mathfrak{Y}}$. This set is ordered by $(\mathfrak{Y}',k')\succeq (\mathfrak{Y},k)$ if and only if $\mathfrak{Y}\succeq\mathfrak{Y}$ and $k'\ge k$.
\end{defi}
\justify
We will need the following auxiliary result.

\begin{lem}\label{rel_coadmissible}
Let $\mathfrak{Y}'$, $\mathfrak{Y}\in\Fb_{\mathfrak{X}}$ be $G_0$-equivariant admissible blow-ups (definition  \ref{Action_blow-up_subgroup}). Suppose $(\mathfrak{Y}',k')\succeq (\mathfrak{Y},k)$ with canonical morphism $\pi:\mathfrak{Y}'\rightarrow\mathfrak{Y}$ over $\mathfrak{X}$ and let $M$ be a coherent $D(\G(k')^{\circ},G_0)_{\lambda}$-module with localization $\Ma = \La oc^{\dag}_{\mathfrak{Y}',k'}(\lambda)(M)\in \text{Coh}\big (\Da^{\dag}_{\mathfrak{Y}',k'}(\lambda),G_0\big )$. Then there exists a canonical isomorphism in $\text{Coh}\big (\Da^{\dag}_{\mathfrak{Y},k}(\lambda),G_0\big )$ given by
\begin{eqnarray*}
\Da^{\dag}_{\mathfrak{Y},k}(\lambda)\otimes_{\pi_*\Da^{\dag}_{\mathfrak{Y}',k'}(\lambda),\; G_{k+1}}\pi_*\Ma\xrightarrow{\simeq} \La oc^{\dag}_{\mathfrak{Y},k}(\lambda)\left(D(\G(k)^{\circ},G_0)\otimes_{D(\G(k')^{\circ},G_0)}M\right).
\end{eqnarray*}
\end{lem} 
\begin{proof} The proof follows word for word the reasoning given in \cite[Lemma 5.2.12]{HPSS} when $\lambda\in\text{Hom}(\T,\G_m)$ is equal to the trivial character. We will only indicate to the reader how the isomorphism is obtained. Let $\Sigma$ be a system of representatives in $G_{k+1}$ for the cosets in $G_{k+1}/G_{k'+1}$. By (\ref{Iso_delta_g}) we have a canonical map
\begin{eqnarray}\label{can_map_dist_alg}
\Db^{\text{an}}(\G(k)^{\circ})_{\lambda}\rightarrow D(\G(k)^{\circ},G_0)_{\lambda}
\end{eqnarray}
which is compatible with variation in $k$. Now, let us take $M$ a $D(\G(k')^{\circ},G_0)_{\lambda}$-module and let us consider the free $\Db^{\text{an}}(\G(k)^{\circ})_{\lambda}$-module
\begin{eqnarray*}
\Db^{\text{an}}(\G(k)^{\circ})_{\lambda}^{M\times\Sigma}:= \displaystyle\bigoplus_{(m,h)\in M\times\Sigma} \Db^{\text{an}}(\G(k)^{\circ})_{\lambda}\; e_{m,h},
\end{eqnarray*}
whose formation is functorial in $M$ and it comes with a linear map
\begin{eqnarray*}
\begin{matrix}
f_M : & \Db^{\text{an}}(\G(k)^{\circ})_{\lambda}^{M\times\Sigma} & \rightarrow & \Db^{\text{an}}(\G(k)^{\circ})_{\lambda}\otimes_{\Db^{\text{an}}(\G(k')^{\circ})_{\lambda}} M \\
        & \lambda_{m,h}e_{m,h} & \mapsto & (\lambda_{m,h}\delta_h)\otimes m - \lambda_{m,h}\otimes (\delta_h . m).
\end{matrix}
\end{eqnarray*}
which fits into an exact sequence
\begin{eqnarray*}
 \Db^{\text{an}}(\G(k)^{\circ})_{\lambda}^{M\times\Sigma} \xrightarrow{f_M} \Db^{\text{an}}(\G(k)^{\circ})_{\lambda}\otimes_{\Db^{\text{an}}(\G(k')^{\circ})_{\lambda}} M \xrightarrow{\text{can}_M} D(\G(k)^{\circ},G_0)_{\lambda}\otimes_{D(\G(k')^{\circ},G_0)_{\lambda}}M \rightarrow 0,
\end{eqnarray*}
\justify
 if $M$ is a finitely presented $D(\G(k')^{\circ},G_0)_{\lambda}$-module \cite[Claim 1 in the proof of lemma 5.2.12]{HPSS}.
\justify
Now, let $M$ be a finitely presented $\Db^{\text{an}}(\G(k')^{\circ})_{\lambda}$-module and $\Ma := \Lb oc^{\dag}_{\mathfrak{Y'},k'}(\lambda)(M)$. Then the natural morphism
\begin{eqnarray}\label{Iso_loc_an_alg}
\Lb oc^{\dag}_{\mathfrak{Y},k}(\lambda)\left(\Db^{\text{an}}(\G(k)^{\circ})_{\lambda}\otimes_{\Db^{\text{an}}(\G(k')^{\circ})_{\lambda}}M\right)\rightarrow \Da^{\dag}_{\mathfrak{Y},k}(\lambda)\otimes_{\pi_*\Da^{\dag}_{\mathfrak{Y}',k'}(\lambda)}\pi_*\Ma
\end{eqnarray}
is bijective. In fact, by theorem \ref{Invariance} we know that the functor $\pi_*$ is exact on coherent $\Da^{\dag}_{\mathfrak{Y}',k'}(\lambda)$-modules, so taking a finite presentation of $M$ we reduce to the case $M=\Db^{\text{an}}(\G(k')^{\circ})_{\lambda}$ which is clear.
\justify
Finally, let us take $M$ a finitely presented $D(\G(k')^{\circ},G_0)_{\lambda}$-module. Let $m_1, ..., m_a$ be generators for $M$ as a $\Db^{\text{an}}(\G(k')^{\circ})_{\lambda}$-module. We have a sequence of $\Db^{\text{an}}(\G(k)^{\circ})_{\lambda}$-modules
\begin{eqnarray*}
\displaystyle\bigoplus_{(i,h)} \Db^{\text{an}}(\G(k)^{\circ})_{\lambda}\; e_{m_i,h}\xrightarrow{f_a} \Db^{\text{an}}(\G(k)^{\circ})_{\lambda}\otimes_{\Db^{\text{an}}(\G(k')^{\circ})_{\lambda}}M \xrightarrow{\text{can}_M} D(\G(k)^{\circ},G_0)_{\lambda}\otimes_{D(\G(k')^{\circ},G_0)_{\lambda}}M\rightarrow 0
\end{eqnarray*}
where $f_a$ denotes the restriction of the map $f_M$ to the free submodule of $ \Db^{\text{an}}(\G(k)^{\circ})_{\lambda}^{M\times\Sigma}$ generated by the finitely many vectors $e_{m_i,h}$, with $1\le i\le a$ and $h\in\Sigma$. Since $\text{im}(f_a)=\text{im}(f_{M})$ the sequence is exact. Since it consists of finitely presented $ \Db^{\text{an}}(\G(k)^{\circ})_{\lambda}$-modules, we can apply the localisation functor $\Lb oc^{\dag}_{\mathfrak{Y},k}(\lambda)$ to it. As 
\begin{eqnarray*}
\Lb oc^{\dag}_{\mathfrak{Y},k}(\lambda)\left(\displaystyle\bigoplus_{(i,h)} \Db^{\text{an}}(\G(k)^{\circ})_{\lambda}\; e_{m_i,h}\right) = \Da^{\dag}_{\mathfrak{Y},k}(\lambda)\otimes_ {\Db^{\text{an}}(\G(k)^{\circ})_{\lambda}}\displaystyle\bigoplus_{(i,h)} \Db^{\text{an}}(\G(k)^{\circ})_{\lambda}\; e_{m_i,h} = \Da^{\dag}_{\mathfrak{Y},k}(\lambda)^{\oplus a|\Sigma|}
\end{eqnarray*}
then (\ref{Iso_loc_an_alg}) gives us the exact sequence
\begin{eqnarray*}
\begin{matrix}
 \Da^{\dag}_{\mathfrak{Y},k}(\lambda)^{\oplus a|\Sigma|} & \rightarrow &  \Da^{\dag}_{\mathfrak{Y},k}(\lambda)\otimes_{\pi_*\Da_{\mathfrak{Y}',k'}(\lambda)}\pi_*\Ma & \rightarrow &  \La oc^{\dag}_{\mathfrak{Y},k}(\lambda)\left(D(\G(k)^{\circ},G_0)_{\lambda}\otimes_{D(\G(k')^{\circ},G_0)_{\lambda}}M\right) & \rightarrow & 0\\
 e_{m_i,h}\otimes P & \mapsto & (P\delta_h\otimes m_i - P\otimes \delta_h m)
\end{matrix}
\end{eqnarray*}
where $\Ma := \Lb oc^{\dag}_{\mathfrak{Y}',k'}(\lambda)(M)$. The cokernel of the first map in this sequence equals by definition 
\begin{eqnarray*}
\Da^{\dag}_{\mathfrak{Y},k}(\lambda)\otimes_{\pi_*\Da^{\dag}_{\mathfrak{Y}',k'}(\lambda),\; G_{k+1}}\pi_*\Ma,
\end{eqnarray*}
and we get the desired isomorphism.
\end{proof}
\justify
Now, let $\Ia$ be an open ideal sheaf on $\mathfrak{X}$, and let $g\in G_0$. Then $\Ja:=(\rho^{\natural}_g)^{-1}((\rho_g)_*(\Ia))$ is again an  open ideal sheaf on $\mathfrak{X}$. Let $\mathfrak{Y}$ be the blow-up of $\Ia$ and $\mathfrak{Y}.g$ the blow-up of $\Ja$, with canonical morphism $\text{pr} _g:\mathfrak{Y}.g\rightarrow\mathfrak{X}$. We have the following result from \cite[lemma 5.2.16]{HPSS}.

\begin{lem}\label{blow-up_action}
There exists a morphism $\rho_g:\mathfrak{Y}\rightarrow \mathfrak{Y}.g$ such that the following diagram is commutative
\begin{eqnarray*}
\begin{tikzcd}
\mathfrak{Y} \arrow[r, "\rho_g"] \arrow[d, "pr"]
& \mathfrak{Y}.g \arrow[d, "\text{pr}_g"]\\
\mathfrak{X} \arrow[r, "\rho_g"]
& \mathfrak{X}.
\end{tikzcd}
\end{eqnarray*}
Moreover, we have $k_{\mathfrak{Y}.g}=k_{\mathfrak{Y}}$ and for any two elements $g,h\in G_0$, we have a canonical isomorphism $(\mathfrak{Y}.g).h\simeq \mathfrak{Y}.(gh)$, and the morphism $\mathfrak{Y}\rightarrow \mathfrak{Y}.g\rightarrow (\mathfrak{Y}.g).h\simeq \mathfrak{Y}.(gh)$ is equal to $\rho_{gh}$. This gives a right action of the group $G_0$ on the family $\Fb_{\mathfrak{X}}$.
\end{lem}
\justify
Let $\text{pr}:\mathfrak{Y}\rightarrow\mathfrak{X}$ be an admissible blow-up and let us denote by $\underline{\La(\lambda)}$ the invertible sheaf on $\mathfrak{Y}$ induced by pulling back the invertible sheaf on $\mathfrak{X}$ induced by the character $\lambda$. This is $\underline{\La(\lambda)}:=\text{pr}^*\La(\lambda)$. Furthermore, for $g\in G_0$ if $\rho_g:\mathfrak{Y}\rightarrow \mathfrak{Y}.g$ is the morphism given by the previous lemma and $\text{pr}_g :\mathfrak{Y}.g\rightarrow\mathfrak{X}$ is the blow-up morphism, then we will denote 
\begin{eqnarray*}
\underline{\La_g(\lambda)} := pr_g^* \La (\lambda).
\end{eqnarray*}
The notation being fixed, we prevent the reader that in order to simplify the notation, in the rest of this work we will avoid to underline these sheaves if the context is clear and there is not risk to any confusion.
\justify
Let us recall that in subsection \ref{group_action_on_blow_up} we have built for any $g\in G_0$ an $\Ob_{\mathfrak{X}}$-linear isomorphism $\Phi_g : \La(\lambda)\rightarrow (\rho_g)_*\La(\lambda)$, being $\rho_g := \alpha\;\circ\; (id_{\mathfrak{X}}\times g)$ the translation morphism ($\alpha$ the right $\mathfrak{G}$-action on $\mathfrak{X}$). By pulling back this morphism and using the commutative diagram in the previous lemma ($\rho_g^*\;\circ\; pr_g^*\;=\; pr^*\;\circ\; \rho_g^*$) we have an $\Ob_{\mathfrak{Y}}$-linear isomorphism $(\rho_g)^* pr_g^* \La(\lambda)\rightarrow pr^*\La(\lambda)$. By adjointness and following the accord established in the previous paragraph, we get an $\Ob_{\mathfrak{Y}.g}$-liner morphism  
\begin{eqnarray*}
R_g : \La_g(\lambda)\rightarrow (\rho_g)_*\La(\lambda).
\end{eqnarray*}
By construction $R_g$ satisfies the cocycle condition (\ref{cocycle_bundle_blow_up}). This means that for every $g,h\in G_0$ we have
\begin{eqnarray}\label{cocycle_cond_action}
R_{hg} = \La_{hg}(\lambda)\xrightarrow{R_g}(\rho_g)_* \La_h(\lambda) \xrightarrow{(\rho_g)_* R_h} (\rho_{hg})_* \La(\lambda).
\end{eqnarray}
In particular $R_g$ is an isomorphism for every $g\in G_0$.
\justify
Exactly as we have done in (\ref{Adjoint_blow_up}), and given that by construction $\Da^{\dag}_{\mathfrak{Y},k}(\lambda)$ acts on $\La (\lambda)$ (resp. $\Da^{\dag}_{\mathfrak{Y}.g,k}(\lambda)$ acts on $\La_{g}(\lambda)$), we can build an isomorphism 
\begin{eqnarray*}
\begin{matrix}
T_g : & \Da^{\dag}_{\mathfrak{Y}.g,k}(\lambda) & \rightarrow & (\rho_g)_*\Da^{\dag}_{\mathfrak{Y},k}(\lambda) \\
        & P & \mapsto & R_g\;\circ\; P\;\circ\; R_g^{-1}.
\end{matrix}
\end{eqnarray*}
\justify
Satisfying the following cocycle condition
\begin{eqnarray}\label{cocyle_action}
T_{hg} = (\rho_g)_*T_h\;\circ\; T_g \hspace{1cm} (g,h\in G_0).
\end{eqnarray} 
\justify
From the previous lemma we get \cite[Corollary 5.2.18]{HPSS}

\begin{coro}\label{blow-up_action_F}
Let us suppose that $(\mathfrak{Y}',k')\succeq (\mathfrak{Y},k)$ for $\mathfrak{Y},\; \mathfrak{Y}'\in \Fb_{\mathfrak{X}}$ and let $\pi:\mathfrak{Y}'\rightarrow\mathfrak{Y}$ be the unique morphism over $\mathfrak{X}$. Let $g\in G_0$. Then $(\mathfrak{Y}'.g, k')\succeq (\mathfrak{Y}.g,k)$ and if we denote by $\pi .g:\mathfrak{Y}'.g\rightarrow\mathfrak{Y}.g$ the unique morphism over $\mathfrak{X}$, we have a commutative diagram
\begin{eqnarray*}
\begin{tikzcd}
\mathfrak{Y}' \arrow[r, "\rho_g "] \arrow[d, "\pi "]
& \mathfrak{Y}'.g \arrow[d, "\pi .g "]\\
\mathfrak{Y} \arrow[r, "\rho_g "]
& \mathfrak{Y}.g.
\end{tikzcd}
\end{eqnarray*}
\end{coro}
\justify
Based on \cite[Definition 5.2.19]{HPSS} we have the following definition.

\begin{defi}\label{defi_coadmisible_G0_eq}
A coadmissible $G_0$-equivariant $\Da(\lambda)$-module on $\Fb_{\mathfrak{X}}$ consists of a family $\Ma:=\left(\Ma_{\mathfrak{Y},k}\right)_{(\mathfrak{Y},k)}$ of coherent $\Da^{\dag}_{\mathfrak{Y},k}(\lambda)$-modules $\Ma_{\mathfrak{Y},k}$ for all $(\mathfrak{Y},k)\in \underline{\Fb}_{\mathfrak{X}}$, with the following properties:
\begin{itemize}
\item[(a)] For any $g\in G_0$ with morphism $\rho_g:\mathfrak{Y}\rightarrow \mathfrak{Y}.g$, there exists an isomorphism
\begin{eqnarray*}
\varphi_g : \Ma_{\mathfrak{Y}.g,k}\rightarrow \left(\rho_g\right)_*\Ma_{\mathfrak{Y},k}
\end{eqnarray*}
of sheaves of $L$-vector spaces, satisfying the following properties:
\begin{itemize}
\item[(i)] For all $g,h\in G_0$ we have $(\rho_g)_*(\varphi_h)\;\circ\;\varphi_g=\varphi_{hg}$.
\item[(ii)] For all open subset $\Ub\subseteq \mathfrak{Y}.g$, all $P\in\Da^{\dag}_{\mathfrak{Y}.g,k}(\lambda)(\Ub)$, and all $m\in\Ma_{\mathfrak{Y}.g,k}(\Ub)$ one has $\varphi_{g}(P\bullet m)= T_{g,\;\Ub}(P)\bullet \varphi_{g,\;\Ub} (m)$.
\item[(iii)]\footnote{As is remarked in \cite[Definition 5.2.19 (iii)]{HPSS}, if $g\in G_{k+1}$, then $\mathfrak{Y}.g = \mathfrak{Y}$ and $g$ acts trivially on the underlying topological space $|\mathfrak{Y}|$.} For all $g\in G_{k+1}$ the map $\varphi_g: \Ma_{\mathfrak{Y}.g,k}=\Ma_{\mathfrak{Y},k}\rightarrow (\rho_g)_*  \Ma_{\mathfrak{Y},k}=\Ma_{\mathfrak{Y},k}$ is equal to multiplication by $\delta_g\in H^{0}\big (\mathfrak{Y},\Da^{\dag}_{\mathfrak{Y},k}(\lambda)\big )$.
\end{itemize} 
\item[(b)] Suppose $\mathfrak{Y},\mathfrak{Y}' \in\Fb_{\mathfrak{X}}$ are both $G_0$-equivariant, and assume further that $(\mathfrak{Y}',k')\succeq (\mathfrak{Y},k)$, and that $\pi:\mathfrak{Y}'\rightarrow\mathfrak{Y}$ is the unique morphism over $\mathfrak{X}$. We require the existence of a transition morphism $\psi_{\mathfrak{Y}',\mathfrak{Y}}:\pi_*\Ma_{\mathfrak{Y}',k'}\rightarrow \Ma_{\mathfrak{Y},k}$, linear relative to the canonical morphism $\Psi: \pi_*\Da^{\dag}_{\mathfrak{Y}',k'}(\lambda)\rightarrow\Da^{\dag}_{\mathfrak{Y},k}(\lambda)$. By using the commutative diagram in the preceding corollary, we required
\begin{eqnarray*}
\varphi_g\;\circ\; \psi_{\mathfrak{Y}'.g,\mathfrak{Y}.g} = (\rho_g)_*(\psi_{\mathfrak{Y}',\mathfrak{Y}})\;\circ\; (\pi.g)_*(\varphi_g). 
\end{eqnarray*}
The  morphism induced by $\psi_{\mathfrak{Y}',\mathfrak{Y}}$ 
\begin{eqnarray}
\overline{\psi}_{\mathfrak{Y}',\mathfrak{Y}}:\Da^{\dag}_{\mathfrak{Y},k}(\lambda)\otimes_{\pi_*\Da^{\dag}_{\mathfrak{Y}',k'}(\lambda),\; G_{k+1}}\pi_*\Ma_{\mathfrak{Y}'}\rightarrow \Ma_{\mathfrak{Y}}
\end{eqnarray}
is required to be an isomorphism of $\Da^{\dag}_{\mathfrak{Y},k}(\lambda)$-modules. Additionally, the morphisms $\psi_{\mathfrak{Y}',\mathfrak{Y}}$ are required to satisfy the transitivity condition $\psi_{\mathfrak{Y}',\mathfrak{Y}}\;\circ\; \pi_* (\psi_{\mathfrak{Y}'',\mathfrak{Y}'})=\psi_{\mathfrak{Y}'',\mathfrak{Y}}$ for $(\mathfrak{Y}'',k'')\succeq (\mathfrak{Y}',k')\succeq (\mathfrak{Y},k)$ in $\underline{\Fb}_{\mathfrak{X}}$. Moreover, $\psi_{\mathfrak{Y},\mathfrak{Y}}= id_{\Ma_{\mathfrak{Y},k}}$.
\end{itemize}
\end{defi}
\justify
A morphism $\Ma \rightarrow \Na$ between such modules consists of morphisms $\Ma_{\mathfrak{Y},k}\rightarrow\Na_{\mathfrak{Y},k}$ of $\Da^{\dag}_{\mathfrak{Y},k}(\lambda)$-modules, which is compatible with the extra structures imposed in $(a)$ and $(b)$. We denote the resulting category by $\Cb_{\mathfrak{X},\lambda}^{G_0}$.
\justify
Let us build now the bridge to the category $\Cb_{G_0,\lambda}$ of coadmissible $D(G_0,L)_{\lambda}$-modules. Given such a module $M$ we have its associated admissible locally analytic $G_0$-representation $V:=M'_b$ together with its subspace of $\G(k)^{\circ}$-analytic vectors $V_{\G(k)^{\circ}-\text{an}}\subseteq V$. As we have remarked, this is stable under the $G_0$-action and its dual $M_k:=\left(V_{\G(k)^{\circ}-\text{an}}\right)'_b$ is a finitely presented $D(\G(k)^{\circ},G_0)_{\lambda}$-module. In this situation we produce a coherent $\Da^{\dag}_{\mathfrak{Y},k}(\lambda)$-module
\begin{eqnarray*}
\La oc^{\dag}_{\mathfrak{Y},k}(\lambda)(M_k)=\Da^{\dag}_{\mathfrak{Y},k}(\lambda)\otimes_{\Db^{\text{an}}(\G(k)^{\circ})_{\lambda}}M_k
\end{eqnarray*}  
for any element $(\mathfrak{Y},k)\in\underline{\Fb}_{\mathfrak{X}}$. We will denote the resulting family by
\begin{eqnarray*}
\La oc^{G_0}_{\lambda}(M):= \left(\La oc^{\dag}_{\mathfrak{Y},k}(\lambda)(M_k)\right)_{(\mathfrak{Y},k)\in\underline{\Fb}_{\mathfrak{X}}}.
\end{eqnarray*}
\justify
On the other hand, let $\Ma$ be an arbitrary coadmissible $G_0$-equivariant arithmetic $\Da(\lambda)$-module on $\Fb_{\mathfrak{X}}$. The transition morphisms $\psi_{\mathfrak{Y}',\mathfrak{Y}}:\pi_*\Ma_{\mathfrak{Y}',k'}\rightarrow\Ma_{\mathfrak{Y},k}$ induce maps $H^0\left(\mathfrak{Y}',\Ma_{\mathfrak{Y}',k'}\right)\rightarrow H^0\left(\mathfrak{Y},\Ma_{\mathfrak{Y},k}\right)$ on global sections. We let
\begin{eqnarray*}
\Gamma(\Ma):= \varprojlim_{(\mathfrak{Y},k)\in\underline{\Fb}_{\mathfrak{X}}}  H^0\left(\mathfrak{Y},\Ma_{\mathfrak{Y},k}\right).
\end{eqnarray*}
The projective limit is taken in the sense of abelian groups. We have the following theorem. Except for some technical details the proof follows word for word the reasoning given in \cite[Theorem 5.2.23]{HPSS}.

\begin{theo}\label{second_equivalence}
Let us suppose that $\lambda\in\text{Hom}(\T,\G_m)$ is an algebraic character such that $\lambda+\rho\in\mathfrak{t}^*_L$ is a dominant and regular character of $\mathfrak{t}_L$.
The functors $\La oc^{G_0}_{\lambda}$ and $\Gamma(\bullet)$ induce quasi-inverse equivalences between the categories $\Cb_{G_0,\lambda}$ (of coadmissible $D(G_0,L)_{\lambda}$-modules) and $\Cb^{G_0}_{\mathfrak{X},\lambda}$.
\end{theo}

\begin{proof}
Let us take $M\in\Cb_{G_0,\lambda}$ and $\Ma\in\Cb^{G_0}_{\mathfrak{X},\lambda}$. As in \cite[Proof of theorem 5.2.23]{HPSS} we will organise the proof in four steps.
\justify
\textbf{Step 1.} We have $\La oc^{G_0}_{\lambda}(M)\in\Cb^{G_0}_{\mathfrak{X},\lambda}$ and $\La oc^{G_0}_{\lambda}(M)$ is functorial in $M$.
\justify
\textit{Proof.} Let us start by defining 
\begin{eqnarray*}
\varphi_g: \La oc^{\dag}_{\mathfrak{Y}.g,k}(\lambda)(M_k)\rightarrow \left(\rho_g\right)_* \La oc^{\dag}_{\mathfrak{Y},k}(\lambda)(M_k)\;\;\;\; (g\in G_0)
\end{eqnarray*}
satisfying $(i)$, $(ii)$ and $(iii)$ in the preceding definition. Let $\tilde{\varphi}_g:M_k\rightarrow M_k$ denote the map dual to the map $V_{\G(k)^{\circ}-\text{an}}\rightarrow V_{\G(k)^{\circ}-\text{an}}$ given by $w\mapsto g^{-1}w$. By definition $\tilde{\varphi}_h\;\circ\;\tilde{\varphi}_g=\tilde{\varphi}_{hg}$. Let $\Ub\subseteq \mathfrak{Y}.g$ be an open subset and $P\in\Da^{\dag}_{\mathfrak{Y}.g,k}(\lambda)(\Ub)$, $m\in M_k$. We define
\begin{eqnarray*}
\varphi_{g,\;\Ub}(P\otimes m):= T_{g,\;\Ub}(P)\otimes \tilde{\varphi}_g(m).
\end{eqnarray*}
Given that $(\rho_g)_*$ is exact we can choose a finite presentation of $M_k$ as a $\Db^{\text{an}}(\G(k)^{\circ})_{\lambda}$-module to conclude that we have a canonical isomorphism
\begin{eqnarray*}
\left(\rho_g\right)_*\left(\La oc^{\dag}_{\mathfrak{Y},k}(\lambda)(M_k) \right)\xrightarrow{\simeq} \left(\left(\rho_g\right)_* \Da^{\dag}_{\mathfrak{Y},k}(\lambda)\right)\otimes_{\Db^{\text{an}}(\G(k)^{\circ})_{\lambda}}M_k.
\end{eqnarray*}
This means that the above definition extends to a map
\begin{eqnarray*}
\varphi_{g}: \Da^{\dag}_{\mathfrak{Y}.g,k}(\lambda)\otimes_{\Db^{\text{an}}(\G(k)^{\circ})_{\lambda}}M_k\rightarrow \left(\rho_g\right)_*\left(\La oc^{\dag}_{\mathfrak{Y},k}(\lambda)(M_k) \right).
\end{eqnarray*}
The family $\{\varphi_g\}_{g\in G_0}$ satisfies $(i)$, $(ii)$ and $(iii)$ in $(a)$. Let us verify condition (b). We suppose that $\mathfrak{Y}'$, $\mathfrak{Y}$ are $G_0$-equivariant and that $(\mathfrak{Y}',k)\succeq (\mathfrak{Y},k)$ with canonical morphism $\pi: \mathfrak{Y}'\rightarrow\mathfrak{Y}$ over $\mathfrak{X}$. As $\pi_*$ is exact we have an isomorphism
\begin{eqnarray*}
\pi_{*}\left(\La oc^{\dag}_{\mathfrak{Y}',k'}(\lambda)(M_{k'})\right)\xrightarrow{\simeq} \pi_*\left(\Da^{\dag}_{\mathfrak{Y}',k'}(\lambda)\right)\otimes_{\Db^{\text{an}}(\G(k)^{\circ})_{\lambda}}M_{k'}.
\end{eqnarray*}
(This is an argument already given in the text for the functor $(\rho_g)_*$). On the other hand, we have that $\G(k')^{\circ}\subseteq \G(k)^{\circ}$ and we have a map $\tilde{\psi}_{\mathfrak{Y}',\mathfrak{Y}}: M_{k'}\rightarrow M_k$ obtained as the dual map of the natural inclusion $V_{\G(k)^{\circ}-{\text{an}}}\hookrightarrow V_{\G(k')^{\circ}-\text{an}}$. Let $\Ub\subseteq\mathfrak{Y}$ be an open subset and $P\in\pi_*\Da^{\dag}_{\mathfrak{Y}',k'}(\lambda)(\Ub)$, $m\in M_{k'}$. We define
\begin{eqnarray*}
\psi_{\mathfrak{Y}',\mathfrak{Y}}(P\otimes m):= \Psi_{\mathfrak{Y}',\mathfrak{Y}}(P)\otimes \tilde{\psi}_{\mathfrak{Y}',\mathfrak{Y}}(m),
\end{eqnarray*}
where $\Psi$ is the canonical injection $\pi_*\Da^{\dag}_{\mathfrak{Y}',k'}(\lambda)\hookrightarrow \Da^{\dag}_{\mathfrak{Y},k}(\lambda)$. By using the preceding isomorphism we can conclude that this morphisms extends naturally to a map
\begin{eqnarray*}
\psi_{\mathfrak{Y}',\mathfrak{Y}}:\pi_{*}\left(\La oc^{\dag}_{\mathfrak{Y}',k'}(\lambda)(M_{k'})\right)\rightarrow \La oc^{\dag}_{\mathfrak{Y},k}(\lambda)(M_k).
\end{eqnarray*}
The cocycle condition translates into the diagram
\begin{eqnarray}\label{coc_2_eq}
\begin{tikzcd} [column sep=huge]
\left(\rho_{g}^{\mathfrak{Y}}\right)_*\pi_*\left(\La oc^{\dag}_{\mathfrak{Y}',k'}(\lambda)(M_{k'})\right) =(\pi .g)_*\left(\rho_g^{\mathfrak{Y}}\right)_*\left(\La oc^{\dag}_{\mathfrak{Y}',k'}(\lambda)(M_{k'})\right)\arrow[r, "\left(\rho^{\mathfrak{Y}}_g\right)_*\psi_{\mathfrak{Y}',\mathfrak{Y}}"] 
& \left(\rho_g^{\mathfrak{Y}}\right)_*\left(\La oc^{\dag}_{\mathfrak{Y},k}(\lambda)(M_{k})\right)\\
(\pi .g)_{*}\left(\La oc^{\dag}_{\mathfrak{Y}'.g,k'}(\lambda)(M_{k'})\right) \arrow[r, "\psi_{\mathfrak{Y}',\mathfrak{Y}}"] \arrow[u,"(\pi .g)_{*}\varphi_g"]
& \La oc^{\dag}_{\mathfrak{Y}.g,k}(\lambda)(M_{k}) \arrow[u, "\varphi_g"]
\end{tikzcd}
\end{eqnarray}
By construction, the diagrams
\begin{eqnarray}\label{diag_first_eq}
\begin{tikzcd} [column sep=huge]
(\pi .g)_*\left(\rho^{\mathfrak{Y}'}_g\right)_*\Da^{\dag}_{\mathfrak{Y}',k'}(\lambda)=\left(\rho^{\mathfrak{Y}}_g\right)_* \pi_*\Da^{\dag}_{\mathfrak{Y}',k'}(\lambda)\arrow[r, "(\rho^{\mathfrak{Y}}_g)_*\Psi_{\mathfrak{Y}',\mathfrak{Y}}"] 
& \left(\rho^{\mathfrak{Y}}_g\right)_*\Da^{\dag}_{\mathfrak{Y},k}(\lambda)\\
(\pi_g)_*\Da^{\dag}_{\mathfrak{Y}'.g,k'}(\lambda) \arrow[u, "(\pi . g)_* T_{g}"] \arrow[r, "\Psi_{\mathfrak{Y}'.g,\mathfrak{Y}.g}"]
& \Da^{\dag}_{\mathfrak{Y}.g,k}(\lambda) \arrow[u, "T_{g}"]
\end{tikzcd}
\hspace{0.7 cm}
\begin{tikzcd}
M_{k'} \arrow[r, "\tilde{\psi}_{\mathfrak{Y}',\mathfrak{Y}}"] \arrow[d, "\tilde{\varphi}_g"]
& M_k \arrow[d, "\tilde{\varphi}_g"]\\
M_{k'} \arrow[r, "\tilde{\psi}_{\mathfrak{Y}',\mathfrak{Y}}"]
& M_k
\end{tikzcd}
\end{eqnarray}
are commutative and therefore (\ref{coc_2_eq}) is also a commutative diagram. The transitivity properties are clear. Let us see that the induced morphism $\overline{\psi}_{\mathfrak{Y}',\mathfrak{Y}}$ is in fact an isomorphism. The morphism $\overline{\psi}_{\mathfrak{Y}',\mathfrak{Y}}$ corresponds under the isomorphism of lemma \ref{rel_coadmissible} to the linear extension
\begin{eqnarray*}
D(\G(k)^{\circ},G_0)\otimes_{D(\G(k')^{\circ},G_0)}M_{k'}\rightarrow M_{k}
\end{eqnarray*}
of $\tilde{\psi}_{\mathfrak{Y}',\mathfrak{Y}}$ via functoriality of $\La oc^{\dag}_{\mathfrak{Y},k}(\lambda)$. By lemma \ref{linear_extension} this linear extension is an isomorphism and hence, so is $\overline{\psi}_{\mathfrak{Y}',\mathfrak{Y}}$. We conclude that $\La oc^{G_0}_{\lambda}(M)\in \Cb^{G_0}_{\mathfrak{X},\lambda}$. Given a morphism $M\rightarrow N$ in $\Cb_{G_0,\lambda}$, we get, by definition, morphisms $M_k\rightarrow N_k$ for any $k\in\Z_{>0}$ compatible with $\tilde{\varphi}_g$ and $\tilde{\psi}_{\mathfrak{Y}',\mathfrak{Y}}$. By functoriality of $\La oc^{\dag}_{\mathfrak{Y},k}(\lambda)$, they give rise to linear maps 
\begin{eqnarray*}
\La oc^{\dag}_{\mathfrak{Y},k}(\lambda)(M_k)\rightarrow \La oc^{\dag}_{\mathfrak{Y},k}(\lambda)(N_k)
\end{eqnarray*} 
which are compatible with the maps $\varphi_g$ and $\psi_{\mathfrak{Y}',\mathfrak{Y}}$.
\justify
\textbf{Step 2.} $\Gamma(\Ma)$ is an object in $\Cb_{G_0,\lambda}$. 
\justify
\textit{Proof.} For $k\in\N$ we choose $(\mathfrak{Y},k)\in\Fb_{\mathfrak{X}}$ and we put $N_k:= H^0(\mathfrak{Y},\Ma_{(\mathfrak{Y},k)})$. By (\ref{morph_for_equi}), lemma \ref{rel_coadmissible} and the fact that $\Ma\in\Cb^{G_0}_{\mathfrak{X},\lambda}$ we get linear isomorphisms
\begin{eqnarray*}
D(\G(k)^{\circ},G_0)\otimes_{D(\G(k')^{\circ},G_0)}N_{k'}\rightarrow N_{k}
\end{eqnarray*}
for $k'\ge k$. This implies that the modules $N_k$ form a $\left(D(\G(k)^{\circ},G_0)\right)_{k\in\N}$-sequence and the projective limit is a coadmissible module.
\justify
\textbf{Step 3.} $\Gamma\;\circ\; \La oc^{G_0}_{\lambda}(M)\simeq M$.
\justify
\textit{Proof.} If $V:= M'_b$, then we have by definition compatible isomorphisms 
\begin{eqnarray*}
H^0\left(\mathfrak{Y},\La oc^{G_0}_{\lambda}(M)_{(\mathfrak{Y},k)}\right)= H^0\left(\mathfrak{Y},\La oc^{\dag}_{\mathfrak{Y},k}(\lambda)(M_k)\right)= \left(V_{\G(k)^{\circ}-\text{an}}\right)'_b,
\end{eqnarray*}
which imply that the coadmissible modules $\Gamma\;\circ\; \La oc^{G_0}_{\lambda}(M)$ and $M$ have isomorphic $\left(D(\G(k)^{\circ},G_0)\right)_{k\in\N}$-sequences.
\justify
\textbf{Step 4.} $\La oc^{G_0}_{\lambda} \;\circ\; \Gamma(\Ma)\simeq\Ma$.
\justify
\textit{Proof.} Let $N:=\Gamma(\Ma)$ and $V:=N'_b$ the corresponding admissible representation. Let $\Na:=\La oc^{G_0}_{\lambda}(N)$. According to the lemma \ref{linear_extension} 
\begin{eqnarray*}
N_k:= D(\G(k)^{\circ},G_0)\otimes_{D(G_0,L)}N_{k'}\rightarrow N
\end{eqnarray*}
produces a $\left(D(\G(k)^{\circ},G_0)\right)_{k\in\N}$-sequence for the coadmissible module $N$ which is isomorphic to its constituting sequence $\left(H^0(\mathfrak{Y},\Ma_{\mathfrak{Y},k})\right)_{(\mathfrak{Y},k)\in\Fb_{\mathfrak{X}}}$ from step 2. Now let $(\mathfrak{Y},k)\in\Fb_{\mathfrak{X}}$. We have the following isomorphisms
\begin{eqnarray*}
\Na_{\mathfrak{Y},k}= \La oc^{\dag}_{\mathfrak{Y},k}(\lambda)(N_k)\simeq  \La oc^{\dag}_{\mathfrak{Y},k}(\lambda)\left(H^0(\mathfrak{Y},\Ma_{\mathfrak{Y},k})\right)\simeq \Ma_{\mathfrak{Y},k}.
\end{eqnarray*}
By $T_{g}$-linearity the action maps $\varphi_g^{\Ma_{\mathfrak{Y},k}}$ and $\varphi_g^{\Na_{\mathfrak{Y},k}}$, constructed in step 1, are the same. Similarly if $(\mathfrak{Y}',k')\succeq (\mathfrak{Y},k)$ are $G_0$-equivariant then the transition maps $\psi^{\Ma_{\mathfrak{Y}',\mathfrak{Y}}}$ and $\psi^{\Na_{\mathfrak{Y}',\mathfrak{Y}}}$ coincide, by $\Psi_{\mathfrak{Y}',\mathfrak{Y}}$-linearity. Hence $\Na\simeq\Ma$ in $\Cb^{G_0}_{\mathfrak{X},\lambda}$.
\end{proof}

\subsection{Coadmissible $G_0$-equivariant $\Da(\lambda)$-modules on the Zariski-Riemann space}\label{G_0-ZR}
\justify
Let us recall that $\mathfrak{X}_{\infty}$ denotes the projective limit of all formal models of $X^{\text{rig}}$ (the rigid-analytic space associated by the GAGA functor to the flag variety $X_L$). The set $\Fb_{\mathfrak{X}}$ of admissible formal blow-ups $\mathfrak{Y}\rightarrow\mathfrak{X}$ is ordered by setting $\mathfrak{Y}'\succeq\mathfrak{Y}$ if the blow-up morphism $\mathfrak{Y}'\rightarrow \mathfrak{X}$ factors as $\mathfrak{Y}'\xrightarrow{\pi}\mathfrak{Y}\rightarrow\mathfrak{X}$, with $\pi$ a blow-up morphism. The set $\Fb_{\mathfrak{X}}$ is directed in the sense that any two elements have a common upper bound, and it is cofinal in the set of all formal models. In particular, $\mathfrak{X}_{\infty}=\varprojlim_{\Fb_{\mathfrak{X}}}\mathfrak{Y}$. The space $\mathfrak{X}_{\infty}$ is also known as the Zariski-Riemann space \cite[Part II, chapter 9, section 9.3]{Bosch}\footnote{In this reference this space is denoted by $\big<\mathfrak{X} \big >$, cf. \cite[subsection 3.2]{HSS}.}. In this subsection we indicate how to realize coadmissible $G_0$-equivariant $\Da(\lambda)$-modules on $\Fb_{\mathfrak{X}}$ as sheaves on the Zariski-Riemann space $\mathfrak{X}_{\infty}$. We start with the following proposition whose proof can be found in \cite[Proposition 5.2.14]{HPSS}.

\begin{prop}
Any formal model $\mathfrak{Y}$ of $X^{\text{rig}}$ s dominated by one which is a $G_0$-equivariant admissible blow-up of $\mathfrak{X}$.
\end{prop}

\begin{rem}\label{G0_eq_system}
As $\Fb_{\mathfrak{X}}$ is cofinal in the set of all formal models, the preceding proposition tells us that the set of all $G_0$-equivariant admissible blow-ups of $\mathfrak{X}$ is also cofinal in the set of all formal models of $\mathfrak{X}$. From now on, we will assume that if $\mathfrak{Y}\in\Fb_{\mathfrak{X}}$, then $\mathfrak{Y}$ also $G_0$-equivariant, and we will denoted by $\rho_{g}^{\mathfrak{Y}}:\mathfrak{Y}\rightarrow\mathfrak{Y}$ the morphism induced by every $g\in G_0$.
\end{rem}
\justify
For every $\mathfrak{Y}\in\Fb_{\mathfrak{X}}$ we denote by $\text{sp}_{\;\mathfrak{Y}}:\mathfrak{X}_{\infty}\rightarrow\mathfrak{Y}$ the canonical projection map. Let $\mathfrak{Y}'\succeq\mathfrak{Y}$ with blow-up morphism $\pi' :\mathfrak{Y}'\rightarrow\mathfrak{Y}$ and $g\in G_0$. Let us consider the following commutative diagram coming from the $G_0$-equivariance of the family $\Fb_{\mathfrak{X}}$
\begin{eqnarray*}
\begin{tikzcd} [column sep=huge]
\mathfrak{X}_{\infty} \arrow[r, "sp_{\;\mathfrak{Y}}"] \arrow[d, "sp_{\;\mathfrak{Y}'}"]
& \mathfrak{Y} \arrow[r, "\rho_g^{\mathfrak{Y}}"]
& \mathfrak{Y} \\
\mathfrak{Y}' \arrow[r, "\rho_{g}^{\mathfrak{Y}'}"] \arrow[ur, "\pi"]
& \mathfrak{Y}'. \arrow[ur, "\pi' "]
\end{tikzcd}
\end{eqnarray*} 
This diagram allows to define a continuous function 
\begin{eqnarray}
\begin{matrix}\label{G_0-action_on_ZR}
\rho_g : & \mathfrak{X}_{\infty}  & \rightarrow & \mathfrak{X}_{\infty}\\
           & (a_{\mathfrak{Y}})_{\mathfrak{Y}\in\Fb_{\mathfrak{X}}} & \mapsto & (\rho_g^{\mathfrak{Y}} (a_{\mathfrak{Y}}))_{\mathfrak{Y}\in\Fb_{\mathfrak{X}}}.
\end{matrix}
\end{eqnarray}
which defines a $G_0$-action on the space $\mathfrak{X}_{\infty}$.
\justify
Let $\Ub\subset \mathfrak{Y}$ be an open subset and let us take $V:= \text{sp}_{\;\mathfrak{Y}}^{-1}(\Ub)\subset \mathfrak{X}_{\infty}$. Using the relation $\text{sp}_{\;\mathfrak{Y}}=\text{sp}_{\;\mathfrak{Y}'}  \;\circ \; \pi$ we see that
\begin{eqnarray*}
\text{sp}_{\;\mathfrak{Y}'}(V) = \text{sp}_{\;\mathfrak{Y}'}(\text{sp}_{\;\mathfrak{Y}}^{-1}(\Ub))
 = \text{sp}_{\;\mathfrak{Y}'}(\text{sp}_{\;\mathfrak{Y}'}^{-1}(\pi'^{-1}(\Ub)))
 = \pi'^{-1}(\Ub),
\end{eqnarray*}
which implies that $\text{sp}_{\;\mathfrak{Y}'}(V)$ is an open subset of $\mathfrak{Y}'$. Now, let us suppose that $\mathfrak{Y}''\xrightarrow{\pi''}\mathfrak{Y}'\xrightarrow{\pi '}\mathfrak{Y}$ are morphisms over $\mathfrak{Y}$. The commutative diagram 
\begin{eqnarray*}
\begin{tikzcd}[column sep=large, row sep=normal]
& \mathfrak{X}_{\infty}\supseteq V:=\text{sp}_{\;\mathfrak{Y}}^{-1}(\Ub) \arrow{ddl}[swap]{\text{sp}_{\;\mathfrak{Y}''}}\arrow{ddr}{\text{sp}_{\;\mathfrak{Y}}}\arrow{d}{\text{sp}_{\mathfrak{Y}'}} & \\
& \mathfrak{Y}' \arrow{dr}[swap]{\pi'}  & \\
\mathfrak{Y}'' \arrow{rr}[swap]{} \arrow{ur}{\pi''} & & \mathfrak{Y}\supseteq\Ub 
\end{tikzcd}
\end{eqnarray*}
implies that 
\begin{eqnarray}\label{sp_V}
\pi''^{-1}(\text{sp}_{\;\mathfrak{Y}'}(V)) = \pi''^{-1}(\pi''(\text{sp}_{\;\mathfrak{Y}''}(V))) = \text{sp}_{\;\mathfrak{Y}''}(V). 
\end{eqnarray}
In this situation, the morphism $\Psi_{\mathfrak{Y}'',\mathfrak{Y}'}: \pi''_{*}\Da^{\dag}_{\mathfrak{Y}'',k''}(\lambda)\rightarrow \Da^{\dag}_{\mathfrak{Y}',k'}(\lambda)$ (defined in (\ref{from_k'_to_k}) ) induces the ring homomorphism 
\begin{eqnarray*}
\Da^{\dag}_{\mathfrak{Y}'',k''}(\lambda)(\text{sp}_{\;\mathfrak{Y}''}(V)) = \pi ''_*\Da^{\dag}_{\mathfrak{Y}'',k''}(\lambda)(\text{sp}_{\;\mathfrak{Y}'}(V)) \xrightarrow{\Psi_{\mathfrak{Y}'',\mathfrak{Y}'}} \Da^{\dag}_{\mathfrak{Y}',k'}(\lambda)(\text{sp}_{\;\mathfrak{Y}'}(V))
\end{eqnarray*}
and we can form the projective limit as in \cite[(5.2.25)]{HPSS}
\begin{eqnarray*}
\Da(\lambda)(V):= \varprojlim_{\mathfrak{Y}'\rightarrow\mathfrak{Y}} \Da^{\dag}_{\mathfrak{Y}',k'}(\lambda) (\text{sp}_{\;\mathfrak{Y}'}(V)).
\end{eqnarray*}
By definition, the open subsets of the form $V:=\text{sp}_{\;\mathfrak{Y}}^{-1}(\Ub)$ form a basis for the topology of $\mathfrak{X}_{\infty}$ and $\Da(\lambda)$ is a presheaf on this basis. The associated sheaf on $\mathfrak{X}_{\infty}$ to this presheaf will also be denoted by $\Da(\lambda)$.
 \justify
Since $ \left(\rho_{g}^{\mathfrak{Y}'} \right)_* \; \circ \; \pi''_{*}\ = \pi''_{*}\  \; \circ \;  \left(\rho_g^{\mathfrak{Y}''}\right)_*$, then relation (\ref{sp_V}), and commutativity of  the first diagram in (\ref{diag_first_eq})  tells us that
\begin{eqnarray*}
\begin{tikzcd}[row sep=large, column sep=large]
\Da^{\dag}_{\mathfrak{Y}'',k''}(\lambda)\left(\text{sp}_{\;\mathfrak{Y}''}(V)\right) = \pi''_{*}\Da^{\dag}_{\mathfrak{Y}'',k''}(\lambda)\left(\text{sp}_{\;\mathfrak{Y}'}(V)\right)  \arrow[r, "\Psi_{\text{sp}_{\;\mathfrak{Y}'}(V)}"] \arrow[d, "T_{g,\;\text{sp}_{\;\mathfrak{Y}''}(V)}^{\mathfrak{Y}''}"]
& \Da^{\dag}_{\mathfrak{Y}',k'}(\lambda)\left(\text{sp}_{\;\mathfrak{Y}'}(V)\right) \arrow[d, "T_{g,\;\text{sp}_{\mathfrak{Y}'}(V)}^{\mathfrak{Y}'}"]\\
\Da^{\dag}_{\mathfrak{Y}'',k''}(\lambda)\left(\left(\rho_g^{\mathfrak{Y}''}\right)^{-1}\left(\text{sp}_{\;\mathfrak{Y}''}(V)\right)\right) = \left(\rho^{\mathfrak{Y}'}_{g}\right)_*\pi''_*\Da^{\dag}_{\mathfrak{Y}'',k''}(\lambda)\left(\text{sp}_{\;\mathfrak{Y}'}(V)\right) \arrow[r, "\Psi_{\text{sp}_{\mathfrak{Y}'}(\rho_g^{-1}(V))}"]
& \left(\rho_{g}^{\mathfrak{Y}'}\right)_*\Da^{\dag}_{\mathfrak{Y}',k'}(\lambda)\left(\text{sp}_{\;\mathfrak{Y}'}(V)\right)
\end{tikzcd}
\end{eqnarray*}
is also a commutative diagram. Let us identify 
\begin{eqnarray*}
\Da(\lambda)(V) = \left\{P:=\left(P_{\mathfrak{Y}',k'}\right)_{(\mathfrak{Y}',k')\in\Fb_{\mathfrak{X}}}\in \displaystyle\prod_{(\mathfrak{Y}',k')\Fb_{\mathfrak{X}}} \Da^{\dag}_{\mathfrak{Y'},k'}(\lambda)\left(\text{sp}_{\;\mathfrak{Y}'}(V)\right)\; | \; \Psi_{\mathfrak{Y}'',\mathfrak{Y}'}(P_{\mathfrak{Y}'',k''}) = P_{\mathfrak{Y}',k'}\right\}
\end{eqnarray*}
and let us consider the sequence 
\begin{eqnarray*}
g.P :=\left(  T_{g,\;\text{sp}_{\;\mathfrak{Y}''}(V)}^{\mathfrak{Y}''}(P_{\mathfrak{Y}'',k''})\right) _{(\mathfrak{Y}'',k'')\in\Fb_{\mathfrak{X}}}\in \displaystyle\prod_{(\mathfrak{Y}'',k'')\in\Fb_{\mathfrak{X}}}\Da^{\dag}_{\mathfrak{Y}'',k''}(\lambda)\left(\left(\rho^{\mathfrak{Y}''}_{g}\right)^{-1}\text{sp}_{\;\mathfrak{Y}''}(V)\right).
\end{eqnarray*}
Using the commutativity of the preceding diagram we see that 
\begin{eqnarray*}
\Psi_{\text{sp}_{\mathfrak{Y}'}(\rho_g^{-1}(V))}\left(  T_{g,\;\text{sp}_{\;\mathfrak{Y}''}(V)}^{\mathfrak{Y}''}(P_{\mathfrak{Y}'',k''})\right) = T_{g,\;\text{sp}_{\mathfrak{Y}'}(V)}^{\mathfrak{Y}'}\left(\Psi_{\text{sp}_{\;\mathfrak{Y}'}(V)}  (P_{\mathfrak{Y}'',k''})\right)
= T_{g,\;\text{sp}_{\mathfrak{Y}'}(V)}^{\mathfrak{Y}'} (P_{\mathfrak{Y}',k'})
\end{eqnarray*}
and therefore, for $g\in G_0$, the actions $T^{\mathfrak{Y}}_g$ assemble to an action 
\begin{eqnarray*}
T_g : \Da(\lambda) \xrightarrow{\simeq} (\rho_g)_* \Da(\lambda).
\end{eqnarray*}
This action is on the left, in the sense that if $g,h\in G_0$ then $(\rho_g)_* T_h\;\circ\; T_g = T_{hg}$. Let us suppose now that $\Ma =(\Ma_{\mathfrak{Y},k})\in \Ca^{G_0}_{\mathfrak{X},\lambda}$. We have the transition maps $\psi_{\mathfrak{Y}'',\mathfrak{Y}'}:\pi''_*\Ma_{\mathfrak{Y}'',k''}\rightarrow \Ma_{\mathfrak{Y}',k'}$ which are linear relative to the morphism (\ref{from_k'_to_k}). As before, we have the map
\begin{eqnarray*}
\Ma_{\mathfrak{Y}'',k''}\left(\text{sp}_{\;\mathfrak{Y}''}(V) \right) = \pi''_* \Ma_{\mathfrak{Y}'',k''}\left(\text{sp}_{\;\mathfrak{Y}'}(V)\right)\xrightarrow{\psi_{\text{sp}_{\;\mathfrak{Y}'}(V)}} \Ma_{\mathfrak{Y}',k'}\left(\text{sp}_{\;\mathfrak{Y}'}(V)\right)
\end{eqnarray*}
which allows us to define $\Ma_{\infty}$ as the sheaf on $\mathfrak{X}_{\infty}$ associated to the presheaf 
\begin{eqnarray*}
\Ma_{\infty}(V):= \varprojlim_{\mathfrak{Y}'\rightarrow\mathfrak{Y}} \Ma_{\mathfrak{Y}',k'}\left(\text{sp}_{\;\mathfrak{Y}'}(V)\right).
\end{eqnarray*}
By definition, we have the following commutative diagram
\begin{eqnarray*}
\begin{tikzcd}[row sep=large, column sep=large]
\Ma_{\mathfrak{Y}'',k''}\left(\text{sp}_{\;\mathfrak{Y}''}(V)\right) = \pi''_{*}\Ma_{\mathfrak{Y}'',k''}\left(\text{sp}_{\;\mathfrak{Y}'}(V)\right)  \arrow[r, "\psi_{\text{sp}_{\;\mathfrak{Y}'}(V)}"] \arrow[d, "\varphi_{g,\;\text{sp}_{\;\mathfrak{Y}''}(V)}^{\mathfrak{Y}''}"]
& \Ma_{\mathfrak{Y}',k'}\left(\text{sp}_{\;\mathfrak{Y}'}(V)\right) \arrow[d, "\varphi_{g,\;\text{sp}_{\mathfrak{Y}'}(V)}^{\mathfrak{Y}'}"]\\
\Ma^{\dag}_{\mathfrak{Y}'',k''}\left(\left(\rho_g^{\mathfrak{Y}''}\right)^{-1}\left(\text{sp}_{\;\mathfrak{Y}''}(V)\right)\right) = \left(\rho^{\mathfrak{Y}'}_{g}\right)_*\pi''_*\Ma_{\mathfrak{Y}'',k''}\left(\text{sp}_{\;\mathfrak{Y}'}(V)\right) \arrow[r, "\psi_{\text{sp}_{\mathfrak{Y}'}(\rho_g^{-1}(V))}"]
& \left(\rho_{g}^{\mathfrak{Y}'}\right)_*\Ma_{\mathfrak{Y}',k'}\left(\text{sp}_{\;\mathfrak{Y}'}(V)\right).
\end{tikzcd}
\end{eqnarray*}
Identifying 
\begin{eqnarray*}
\Ma_{\infty}(V) = \left\{m:=\left(m_{\mathfrak{Y}',k'}\right)_{(\mathfrak{Y}',k')\in\Fb_{\mathfrak{X}}}\in \displaystyle\prod_{(\mathfrak{Y}',k')\Fb_{\mathfrak{X}}} \Ma_{\mathfrak{Y'},k'}\left(\text{sp}_{\;\mathfrak{Y}'}(V)\right)\; | \; \psi_{\mathfrak{Y}'',\mathfrak{Y}'}(m_{\mathfrak{Y}'',k''}) = m_{\mathfrak{Y}',k'}\right\}
\end{eqnarray*} 
we see as before that if 
\begin{eqnarray*}
g.m := \left(  \varphi_{g,\;\text{sp}_{\;\mathfrak{Y}''}(V)}^{\mathfrak{Y}''}(m_{\mathfrak{Y}'',k''})\right) _{(\mathfrak{Y}'',k'')\in\Fb_{\mathfrak{X}}}\in \displaystyle\prod_{(\mathfrak{Y}'',k'')\in\Fb_{\mathfrak{X}}}\Ma_{\mathfrak{Y}'',k''}\left(\left(\rho^{\mathfrak{Y}''}_{g}\right)^{-1}\text{sp}_{\;\mathfrak{Y}''}(V)\right),
\end{eqnarray*}
then the preceding commutative diagram implies that
\begin{eqnarray*}
\psi_{\text{sp}_{\mathfrak{Y}'}(\rho_g^{-1}(V))}\left(  \varphi_{g,\;\text{sp}_{\;\mathfrak{Y}''}(V)}^{\mathfrak{Y}''}(m_{\mathfrak{Y}'',k''})\right) = \varphi_{g,\;\text{sp}_{\mathfrak{Y}'}(V)}^{\mathfrak{Y}'}\left(\psi_{\text{sp}_{\;\mathfrak{Y}'}(V)}  (m_{\mathfrak{Y}'',k''})\right)
= \varphi_{g,\;\text{sp}_{\mathfrak{Y}'}(V)}^{\mathfrak{Y}'} (m_{\mathfrak{Y}',k'}),
\end{eqnarray*}
and therefore we get a family $(\varphi_{g})_{g\in G_0}$ of isomorphisms
\begin{eqnarray} \label{G_0-equi}
\varphi_g : \Ma_{\infty}\rightarrow (\rho_g)_*\Ma_{\infty}
\end{eqnarray}
of sheaves of $L$-vector spaces. By definition \ref{defi_coadmisible_G0_eq} we have that if $g,h\in G_0$ then $\varphi_{hg}=(\rho_g)_*\varphi_h\;\circ\;\varphi_g$. Furthermore, under the preceding identifications, if $P=(P_{\;\mathfrak{Y}',k'})\in\Da(\lambda)(V)$ and $m=(m_{\;\mathfrak{Y}',k'})\in\Ma_{\infty}(V)$, then
$P.m = (P_{\;\mathfrak{Y}',k'}.m_{\;\mathfrak{Y}',k'})_{(\mathfrak{Y}',k')\in\Fb_{\mathfrak{X}}}$ and therefore
\begin{align*}
\varphi_{g,\; V} (P.m) = \left(\varphi^{\mathfrak{Y}'}_{g,\;\text{sp}_{\mathfrak{Y}'}(V)}(P_{\;\mathfrak{Y}',k'}.m_{\;\mathfrak{Y}',k'})\right)_{(\mathfrak{Y}',k')\in\Fb_{\mathfrak{X}}} 
&= \left(T^{\mathfrak{Y}'}_{g,\;\text{sp}_{\mathfrak{Y}'}(V)}(P_{\mathfrak{Y}',k'}).\varphi^{\mathfrak{Y}'}_{g,\;\text{sp}_{\mathfrak{Y}'}(V)}(m_{\mathfrak{Y}',k'})\right)_{(\mathfrak{Y}',k')\in\Fb_{\mathfrak{X}}}\\
 &= T_{g,\; V} (P).\varphi_{g,\; V}(m).
\end{align*}
In particular,  $\Ma_{\infty}$ is a $G_0$-equivariant $\Da(\lambda)$-module on the topologial $G_0$-space $\mathfrak{X}_{\infty}$. Let us see that the formation of $\Ma_{\infty}$ is functorial. Let $\gamma: \Ma\rightarrow \Na$ be a morphism in $\Ca^{G_0}_{\mathfrak{X},\lambda}$. Then, in particular we have the following commutative diagram
\begin{eqnarray*}
\begin{tikzcd} [column sep=large, row sep=large]
\pi''_*\Ma_{\mathfrak{Y}'',k''} \arrow[r, "\psi^{\Ma}_{\mathfrak{Y}'',\mathfrak{Y}'}"] \arrow[d, "\pi''_*(\gamma_{\mathfrak{Y}'',k''})"]
& \Ma_{\mathfrak{Y}',k'} \arrow[d, "\gamma_{\mathfrak{Y}',k'}"]\\
\pi_{*}\Na_{\mathfrak{Y},k} \arrow[r, "\psi^{\Na}_{\mathfrak{Y}'',\mathfrak{Y}'}"]
& \Na_{\mathfrak{Y},k}.
\end{tikzcd}
\end{eqnarray*}
Let $m=(m_{\mathfrak{Y},k})_{(\mathfrak{Y},k)\in\Fb_{\mathfrak{X}}}\in \Ma_{\infty}(V)$ and 
\begin{eqnarray*}
s:=\left(\gamma_{\mathfrak{Y}'',k''}(m_{\mathfrak{Y}'',k''})\right)_{(\mathfrak{Y''},k'')\in\Fb_{\mathfrak{X}}}\in \displaystyle\prod_{(\mathfrak{Y}'',k'')\in\Fb_{\mathfrak{X}}}\Na_{\mathfrak{Y}'',k''}\left(\text{sp}_{\;\mathfrak{Y}''}(V)\right).
\end{eqnarray*}
Commutativity in the preceding diagram implies that 
\begin{eqnarray*}
\psi^{\Na}_{\text{sp}_{\mathfrak{Y}''}(V)}\left(s_{\mathfrak{Y}'',k''}\right) = \psi^{\Na}_{\text{sp}_{\;\mathfrak{Y}''}(V)} \left(\gamma_{\text{sp}_{\mathfrak{Y}''}(V)}\left(m_{\mathfrak{Y}'',k''}\right)\right)
= \gamma_{\text{sp}_{\mathfrak{Y}'}(V)}\left(\psi^{\Ma}_{\text{sp}_{\;\mathfrak{Y}'}(V)}(m_{\mathfrak{Y}'',k''})\right)
=  \gamma_{\text{sp}_{\mathfrak{Y}'}(V)} \left(m_{\mathfrak{Y}',k'}\right)
= s_{\mathfrak{Y}',k'},
\end{eqnarray*}
therefore $s\in\Na_{\infty}(V)$ and $\gamma$ induces a morphism $\gamma_{\infty}:\Ma_{\infty}\rightarrow \Na_{\infty}$.This shows that the preceding  construction is functorial. The next proposition is the twisted analogue of \cite[Proposition 5.2.29]{HPSS}.

\begin{prop}\label{G_0-equivalence}
Let $\lambda\in \text{Hom}(\T,\G_m)$ be an algebraic character which induces, via derivation, a dominant and regular character of $\mathfrak{t}^*_L$. The functor $\Ma\rightsquigarrow\Ma_{\infty}$ from the category $\Ca^{G_0}_{\mathfrak{X},\lambda}$ to $G_0$-equivariant $\Da(\lambda)$-modules is a faithful functor.
\end{prop}
\begin{proof}
We start the proof by remarking that $\text{sp}_{\;\mathfrak{Y}}(\mathfrak{X}_{\infty})=\mathfrak{Y}$ for every $\mathfrak{Y}\in\Fb_{\mathfrak{X}}$. By remark \ref{G0_eq_system}, the global sections of $\Ma_{\infty}$ equal to
\begin{eqnarray*}
H^{0}(\mathfrak{X}_{\infty},\Ma_{\infty}) = \varprojlim_{(\mathfrak{Y},k)\in\underline{\Fb}_{\mathfrak{X}}}H^{0}(\mathfrak{Y},\Ma_{\mathfrak{Y},k}) = \Gamma (\Ma).
\end{eqnarray*} 
Now, let $f,h: \Ma\rightarrow \Na$ be two morphisms in $\Ca^{G_0}_{\mathfrak{X},\lambda}$ such that $f_{\infty}= h_{\infty}$. By theorem \ref{second_equivalence}, it is enough to verify $\Gamma(f)=\Gamma(h)$ which is clear since $H^{0}(\mathfrak{X}_{\infty},f_{\infty})=H^0(\mathfrak{X}_{\infty},h_{\infty})$.
\end{proof}
\justify
If $(\bullet)_{\infty}$ denotes the previous functor, then we will denote by $\Lb oc^{G_0}_{\infty}(\lambda)$ the composition of the functor $\Lb oc^{G_0}_{\lambda}$ with $(\bullet)_{\infty}$, i.e.,
\begin{eqnarray*}
\{\text{Coadmissible}\; D(G_0,L)_{\lambda}-\text{modules}\}\xrightarrow{\Lb oc^{G_0}_{\infty}(\lambda)}\{G_0-\text{equivariant}\;\Da(\lambda)-\text{modules}\}.
\end{eqnarray*}
Since $\Lb oc^{G_0}_{\lambda}$ is an equivalence of categories, the preceding proposition implies that $\Lb oc^{G_0}_{\infty}(\lambda)$ is a faithful functor.

\section{$G$-equivariant modules}
\justify
Thorough this section we will denote by $G=\G(L)$ and by $\Bb$ the semi-simple Bruhat-Tits building of the $p$-adic group $G$ (\cite{BT1} et \cite{BT2}). This is a simplicial complex endowed with a natural right $G$-action.
\justify
The purpose of this section is to extend the above results from $G_0$-equivariant objects to objects equivariants for the whole group $G$.
\justify
We start by fixing some notation.\footnote{This is exactly as in \cite[5.3.1]{HPSS}.} To each special vertex $v\in\Bb$ the Bruhat-Tits theory associates a connected reductive group $\mathfrak{o}$-scheme $\G_v$, whose generic fiber $(\G_v)_{L} : = \G_v\times_{\text{Spec}(\mathfrak{o})}\text{Spec}(L)$ is canonically isomorphic to $\G_L$. We denote by $X_v$ the smooth flag scheme of $\G_v$ whose generic fiber $(X_v)_L$ is canonically isomorphic to the flag variety $X_L$. We will distinguish the next constructions by adding the corresponding vertex to them. For instance, we will write $Y_v$ for an (algebraic) admissible blow-up of the smooth model $X_v$, $G_{v,0}$ for the group of points $\G_{v}(\mathfrak{o})$ and $G_{v,k}$ for the group of points $\G_{v}(k)(\mathfrak{o})$. We will use the same conventions if we deal with formal completions. For instance, $\mathfrak{Y}_v$ will always denote an admissible formal blow-up of $\mathfrak{X}_v$. We point out to the reader that the blow-up morphism $\mathfrak{Y}_v\rightarrow\mathfrak{X}_v$ will make part of the datum of $\mathfrak{Y}_v$, and that even if for another special vertex $v'\neq v$ the formal $\mathfrak{o}$-scheme $\mathfrak{Y}_v$  is also a blow-up of the smooth formal model $\mathfrak{X}_{v'}$, we will only consider  it as a blow-up of $\mathfrak{X}_v$. We will denote by $\Fb_v:=\Fb_{\mathfrak{X}_v}$, the set of all admissible formal blow-ups $\mathfrak{Y}_v\rightarrow\mathfrak{X}_v$ of  $\mathfrak{X}_v$ and by $\underline{\Fb}_{v}:= \underline{\Fb}_{\mathfrak{X}_v}$ the respective directed system of definition \ref{Defi_F_underline}. By the preceding accord, the sets $\Fb_v$ and $\Fb_{v'}$ are disjoint if $v\neq v'$. Let 
\begin{eqnarray*}
\Fb := \displaystyle\bigsqcup_{v}\Fb_v
\end{eqnarray*}
where $v$ runs over all special vertices of $\Bb$. We recall for the reader that $\mathfrak{X}_{\infty}$ is equal to the projective limit of all formal models of $X^{\text{rig}}$.

\begin{rem}\label{partial_order_F}
The set $\Fb$ is partially ordered in the following way. We say that $\mathfrak{Y}_{v'}\succeq\mathfrak{Y}_v$ if the projection $\text{sp}_{\;\mathfrak{Y}_{v'}}:\mathfrak{X}_{\infty}\rightarrow\mathfrak{Y}_{v'}$ factors through the projection $\text{sp}_{\;\mathfrak{Y}_{v}}:\mathfrak{X}_{\infty}\rightarrow\mathfrak{Y}_{v}$
\begin{eqnarray*}
\begin{tikzcd}
 & \mathfrak{X}_{\infty} \arrow{dr}{\text{sp}_{\;\mathfrak{Y}_v}} \arrow{dl}{\text{sp}_{\;\mathfrak{Y}_{v'}}} \\
\mathfrak{Y}_{v'} \arrow{rr}{} && \mathfrak{Y}_v.
\end{tikzcd}
\end{eqnarray*}
\end{rem} 

\begin{defi}\label{partial_order_underline_F}
We will denote by $\underline{\Fb} := \bigsqcup_{v}\underline{\Fb}_{v}$, where $v$ runs over all the special vertices of $\Bb$. This set is partially ordered as follows. We say that $(\mathfrak{Y}_{v'},k')\succeq (\mathfrak{Y}_{v},k)$ if $\mathfrak{Y}_{v'}\succeq \mathfrak{Y}_{v}$ and $\text{Lie}(\G_{v'}(k'))\subset \text{Lie}(\G_{v}(k))$ (or equivalent $\varpi^{k'}\text{Lie}(\G_{v'})\subset \varpi^{k}\text{Lie}(\G_v)$) as lattices in $\mathfrak{g}_L$.
\end{defi}
\justify
For any special vertex $v\in\Bb$, any element $g\in G$ induces an isomorphism 
\begin{eqnarray*}
\rho^{v}_g : X_v\rightarrow X_{v.g}.
\end{eqnarray*}
The isomorphism induced by $\rho^v_{g}$ on the generic fibers $(X_v)_L\simeq X_L \simeq (X_{v.g})_L$ coincides with right translation by $g$ on $X_L$
\begin{eqnarray*}
\rho_g : X_L = X_L \times_{\text{Spec}(L)}\text{Spec}(L) \xrightarrow{id_{X_L}\times g} X_L\times_{\text{Spec}(L)}\text{Spec}(\G_L) \xrightarrow{\alpha_L} X_L,
\end{eqnarray*}
where we have used $\G(L)=\G_L(L)$. Moreover, $\rho_g^{v}$ induces a morphism $\mathfrak{X}_v\rightarrow \mathfrak{X}_{v.g}$, which we denote again by $\rho^{v}_g$, and which coincides with the right translation on $\mathfrak{X}_v$ if $g\in G_{v,0}$ (of course in this case $vg=v$). Let $(\rho^{v}_g)^{\natural}: \Ob_{\mathfrak{X}_{vg}}\rightarrow (\rho^{v}_g)_*\Ob_{\mathfrak{X}_v}$ be the comorphism of $\rho^v_g$. If $\pi:\mathfrak{Y}_v\rightarrow\mathfrak{X}_v$ is an admissible blow-up of an ideal $\Ib\subset \Ob_{\mathfrak{X}_v}$, then blowing-up $((\rho^{v}_g)^{\natural})^{-1}((\rho_g^v)_*\Ib)$ produced a formal scheme $\mathfrak{Y}_{vg}$ (cf. lemma \ref{blow-up_action}), together with an isomorphism $\rho_g^{v}:\mathfrak{Y}_v\rightarrow\mathfrak{Y}_{vg}$. As in lemma \ref{blow-up_action} we have $k_{\mathfrak{Y}_v} = k_{\mathfrak{Y}_{vg}}$. For any $g,h\in G$ and any admissible formal blow-up $\mathfrak{Y}_v\rightarrow\mathfrak{X}_v$, we have $\rho^{vg}_h\;\circ\; \rho^v_g = \rho^v_{gh} : \mathfrak{Y}_v \rightarrow \mathfrak{Y}_{vgh}$. This gives a right $G$-action on the family $\Fb$ and on the projective limit $\mathfrak{X}_{\infty}$. Finally, if $\mathfrak{Y}_{v'}\succeq \mathfrak{Y}_v$ with morphism $\pi :\mathfrak{Y}_{v'}\rightarrow \mathfrak{Y}_v$ and $g\in G$, then $\mathfrak{Y}_{v'g}\succeq\mathfrak{Y}_{vg}$, and we have the relation $\rho_g^v  \; \circ \;  \pi = \pi g  \; \circ \; \rho_g^{v'}$ (here $\pi g : \mathfrak{Y}_{v'g}\rightarrow\mathfrak{Y}_{vg}$). Now, over every special vertex $v\in\Bb$ the algebraic character $\lambda$ induces an invertible sheaf $\Lb_v(\lambda)$ on $X_v$, such that for every $g\in G$ there exists an isomorphism
\begin{eqnarray*}
R^{v}_{g}: \Lb_{vg}(\lambda)\rightarrow (\rho^{v}_g)_*\Lb_v (\lambda),
\end{eqnarray*}
\justify
satisfying the cocycle condition
\begin{eqnarray} \label{Translation_bruhat_tits}
R_{hg}^{vhg} = \left(\rho^{vh}_g\right)_* R^{v}_{h}\;\circ\; R^{vh}_{g}
\;\;\;\;\;\;\;
(h,g\in G).
\end{eqnarray}
\justify
As usual, for every special vertex $v\in\Bb$, we will denote by $\La_v(\lambda)$ the $p$-adic completion of the sheaf $\Lb_{v}(\lambda)$, which is considered as an invertible sheaf on $\mathfrak{X}_v$. Let $(\mathfrak{Y}_v,k)\in \Fb$ with blow-up morphism $\text{pr}:\mathfrak{Y}_v\rightarrow\mathfrak{X}_v$. At the level of differential operators, we will denote by $\Da^{\dag}_{\mathfrak{Y}_v,k}(\lambda)$ the sheaf of arithmetic differential operators on $\mathfrak{Y}_v$ acting on the line bundle $\La_{v}(\lambda)$\footnote{Here we abuse of the notation and we denote again by $\La_v(\lambda)$ the invertible sheaf $\text{pr}^*\La_v(\lambda)$ on $\mathfrak{Y}_v$.}. We have the following important properties. Let $g\in G$. As in (\ref{Adjoint_blow_up}) the isomorphism (\ref{Translation_bruhat_tits}) induces a left action
\begin{eqnarray*}
\begin{matrix}
T_g^v : & \Da^{\dag}_{\mathfrak{Y}_{vg},k}(\lambda) & \xrightarrow{\simeq} & \left(\rho_{g}^{v}\right)_*  \Da^{\dag}_{\mathfrak{Y}_{v},k}(\lambda)\\

       & P & \mapsto & R_g^{v}\; P\; (R_g^v)^{-1}.
\end{matrix}
\end{eqnarray*}
Now, we identify the global sections $H^0 (\mathfrak{Y}_v,  \Da^{\dag}_{\mathfrak{Y}_{v},k}(\lambda))$ with $\Db^{\text{an}}(\G_v(k)^{\circ})_{\lambda}$ and obtain the group homomorphism
\begin{eqnarray*}
\begin{matrix}
G_{v,k+1} & \rightarrow & H^0 \left(\mathfrak{Y}_v,  \Da^{\dag}_{\mathfrak{Y}_{v},k}(\lambda)\right)^{\times}\\
g & \mapsto & \delta_g ,
\end{matrix}
\end{eqnarray*}
where $G_{v,k+1}=\G_{v}(k)^{\circ}(L)$ denotes the group of $L$-rational points (or $\mathfrak{o}$-points of $\G_{v}(k+1)$). We will follow the same lines of reasoning given in \cite[Proposition 5.3.2]{HPSS} to prove the following proposition.

\begin{prop}\label{Transition_morph}
Suppose $(\mathfrak{Y}_{v'},k')\succeq (\mathfrak{Y}_v,k)$ for  pairs $(\mathfrak{Y}_{v'},k'),\; (\mathfrak{Y}_v,k)\in \underline{\Fb}$ with morphism $\pi: \mathfrak{Y}_{v'}\rightarrow \mathfrak{Y}_v$. There exists a canonical morphism of sheaves of rings 
\begin{eqnarray*}
\Psi : \pi_*  \Da^{\dag}_{\mathfrak{Y}_{v'},k'}(\lambda)\rightarrow  \Da^{\dag}_{\mathfrak{Y}_{v},k}(\lambda)
\end{eqnarray*}
which is $G$-equivariant in the sense that for every $g\in G$ we have $T_g\;\circ\;\Psi = \left(\rho_g^v\right)_*\Psi \;\circ\;(\pi g)_*T_g$.
\end{prop}

\begin{proof}
Let us denote by $\text{pr}': \mathfrak{Y}_{v'}\rightarrow \mathfrak{X}_{v'}$ and  $\text{pr}: \mathfrak{Y}_{v}\rightarrow \mathfrak{X}_{v}$ the blow-ups morphisms, and let us put $\widetilde{\text{pr}}:= \text{pr}\;\circ\; \pi$. We have the following commutative diagram 
\begin{eqnarray*}
\begin{tikzcd}
\mathfrak{Y}_{v'} \arrow[to=Z, "\text{pr}' "] \arrow[to=2-2, "\widetilde{\text{pr}}"]
& \mathfrak{Y}_{v} \arrow[d, "\text{pr}"] \\
|[alias=Z]| \mathfrak{X}_{v'}
& \mathfrak{X}_v
\arrow[from=ul, to=1-2, "\pi"].
\end{tikzcd}
\end{eqnarray*}  
Let us fix $m\in\N$. As in \cite[Proposition 5.3.6]{HPSS} we show first the existence of a canonical morphism of sheaves of $\mathfrak{o}$-algebras
\begin{eqnarray}\label{canonical_tilde}
\Db_{Y_{v'}}^{(m,k)}(\lambda) \rightarrow \widetilde{\text{pr}}^* \Db^{(m,k)}_{X_v}(\lambda).
\end{eqnarray}
Here $Y_{v'}$, $Y_v$, $X_{v'}$ and $X_v$ denote the $\mathfrak{o}$-scheme of finite type whose completions are $\mathfrak{Y}_{v'}$, $\mathfrak{Y}_v$, $\mathfrak{X}_{v'}$ and $\mathfrak{X}_v$, respectively. The morphisms between these schemes will be denoted by the same letters, for instance $\text{pr}: Y_v\rightarrow X_{v}$. We recall for the reader that the sheaf $\Db_{Y_{v'}}^{(m,k')}(\lambda)$ is filtered by locally free sheaves of finite rank 
\begin{eqnarray*}
\Db_{Y_{v'},d}^{(m,k')}(\lambda)  = \text{pr}'^{*}\Lb_{v'}(\lambda)\otimes_{\Ob_{Y_{v'}}} \text{pr}'^{*}\Db_{X_{v'},d}^{(m,k')} \otimes_{\Ob_{Y_{v'}}} \text{pr}'^{*} \Lb_{v'}(\lambda)^{\vee} 
 = \text{pr}'^{*}\left(\Db_{X_{v'},d}^{(m,k')}(\lambda)\right),
\end{eqnarray*}
and therefore by the projection formula \cite[Part II, Section 5, exercise 5.1 (d) ]{Hartshorne1} and  given that $\text{pr}'_*\Ob_{Y_{v'}}=\Ob_{X_{v'}}$ (cf. \cite[Lemma 3.2.3 (iii)]{HPSS}) we have for every $d\in\N$
\begin{eqnarray*}
\text{pr}'_*\left(\Db_{Y_{v'},d}^{(m,k')}(\lambda)\right) = \text{pr}'_*\left(\Ob_{Y_{v'}}\otimes_{\Ob_{Y_{v'}}} \text{pr}'^{*}\Db_{X_{v'},d}^{(m,k')}(\lambda)\right)
= \text{pr}'_* (\Ob_{Y_{v'}})\otimes_{\Ob_{X_{v'}}}\Db_{X_{v'},d}^{(m,k')}(\lambda) 
= \Db_{X_{v'},d}^{(m,k')}(\lambda),
\end{eqnarray*}
which implies that 
\begin{eqnarray*}
\text{pr}'_*\left(\Db_{Y_{v'}}^{(m,k')}(\lambda)\right) = \Db_{X_{v'}}^{(m,k')}(\lambda)
\end{eqnarray*}
because the direct image commutes with inductive limits on a noetherian space. By proposition  \ref{canonical_lambda} and the preceding relation we have a canonical map of filtered $\mathfrak{o}$-algebras 
\begin{eqnarray*}
D^{(m)}(\G_{v'}(k'))\rightarrow H^0\left(X_{v'},\Db_{X_{v'}}^{(m,k')}(\lambda)\right) = H^0\left(X_{v'},\text{pr}'_*\left(\Db_{Y_{v'}}^{(m,k')}(\lambda)\right) \right) = H^0\left(Y_{v'},\Db_{Y_{v'}}^{(m,k')}(\lambda)\right),
\end{eqnarray*} 
in particular we get a morphism of sheaves of filtered $\mathfrak{o}$-algebras (this is exactly as we have done in (\ref{morph_with_cong_level_alg}))
\begin{eqnarray}\label{canonical_lambda_blow-up}
\Phi^{(m,k')}_{Y_{v'}} : \Ab^{(m,k')}_{Y_{v'}}:= \Ob_{Y_{v'}}\otimes_{\mathfrak{o}} D^{(m)}(\G_{v'}(k')) \rightarrow \Db_{Y_{v'}}^{(m,k')}(\lambda).
\end{eqnarray}
Applying $\text{Sym}^{(m)}(\bullet)\;\circ\; \varpi^{k'}\text{pr}'^*(\bullet)$ to the surjection (\ref{sur_map-H})  we obtain a surjection 
\begin{eqnarray*}
\Ob_{Y_{v'}}\otimes_{\mathfrak{o}} \text{Sym}^{(m)}\left(\text{Lie}(\G_{v'}(k'))\right)\rightarrow \text{Sym}^{(m)}\left(\varpi^{k'}\text{pr}'^*\;\Tb_{X_{v'}}\right)
\end{eqnarray*}
which equals the associated graded morphism of (\ref{canonical_lambda_blow-up}) by proposition \ref{fini_blow_up}. Hence $\Phi^{(m,k')}_{Y_{v'}}$ is surjective. On the other hand, if we apply $\widetilde{\text{pr}}^*$ to the surjection
\begin{eqnarray*}
\Phi^{(m,k)}_{X_v} : \Ab^{(m,k)}_{X_v} = \Ob_{X_v}\otimes_{\mathfrak{o}} D^{(m)}(\G_v(k)) \rightarrow \Db^{(m,k)}_{X_v}(\lambda)
\end{eqnarray*}
we obtain the surjection $\Ob_{Y_{v'}}\otimes_{\mathfrak{o}} D^{(m)}(\G_v(k))\rightarrow \widetilde{\text{pr}}^* \Db^{(m,k)}_{X_v}(\lambda)$. Let us recall that $(\mathfrak{Y}_{v'},k')\succeq (\mathfrak{Y}_v,k)$ implies, in particular, that $\text{Lie}(\G_{v'}(k'))\subseteq \text{Lie}(\G_v (k))$ and therefore $\varpi^{k'}\text{Lie}(\G_{v'})\subset \varpi^k \text{Lie}(\G_v)$. By (\ref{Description_D(Gk)}), the preceding inclusion gives rise to an injective ring homomorphism $D^{(m)}(\G_{v'}(k'))\hookrightarrow D^{(m)}(\G_v(k))$. Let us see that the composition 
\begin{eqnarray*}
\Ob_{Y_{v'}}\otimes_{\mathfrak{o}} D^{(m)}(\G_{v'}(k'))\hookrightarrow \Ob_{Y_{v'}}\otimes_{\mathfrak{o}}D^{(m)}(\G_v(k)) \twoheadrightarrow \widetilde{\text{pr}}^* \Db^{(m,k)}_{X_v}(\lambda)
\end{eqnarray*}
factors through $\Db^{(m,k')}_{X_{v'}}(\lambda)$.
\begin{eqnarray*}
\begin{tikzcd}
\Ob_{Y_{v'}}\otimes_{\mathfrak{o}} D^{(m)}(\G_{v'}(k')) \arrow[r] \arrow[d]  
& \widetilde{\text{pr}}^* \Db^{(m,k)}_{X_v}(\lambda)\\
\Db^{(m,k')}_{Y_{v'}}(\lambda). \arrow[ru,dashrightarrow]
\end{tikzcd}
\end{eqnarray*}
Since by lemma \ref{algebraic_local_desc} all those sheaves are $\varpi$-torsion free, this can be checked after tensoring with $L$ in which case we have that $\Db^{(m,k')}_{Y_{v'}}\otimes_{\mathfrak{o}}L \simeq \widetilde{\text{pr}}^* \Db^{(m,k)}_{X_{v}}\otimes_{\mathfrak{o}}L$ is the (push-forward of the) sheaf of algebraic differential operators on the generic fiber of $Y_{v'}$ (cf. discussion given at the beginning of subsection \ref{diff_on_adm_bup}). We thus get the canonical morphism of sheaves (\ref{canonical_tilde}). Passing to completions we get a canonical morphism $\widehat{\Da}^{(m,k')}_{\mathfrak{Y}_{v'}}(\lambda)\rightarrow \widetilde{\text{pr}}^* \widehat{\Da}^{(m,k)}_{\mathfrak{X}_{v}}(\lambda)$. Taking inductive limit over all $m$ and inverting $\varpi$ gives a canonical morphism $\Da^{\dag}_{\mathfrak{Y}_{v'},k'}(\lambda)\rightarrow \widetilde{\text{pr}}^* \Da^{\dag}_{\mathfrak{X}_v,k}(\lambda)$. Now, let us consider the formal scheme $\mathfrak{Y}_{v'}$ as a blow-up of $\mathfrak{X}_v$ via $\widetilde{\text{pr}}$. Then $\pi$ becomes a morphism of formal schemes over $\mathfrak{X}_v$ and we consider $\widetilde{\text{pr}}^* \Da^{\dag}_{\mathfrak{X}_v,k}(\lambda)$ as the sheaf of arithmetic differential operators with congruence level $k$ defined on $\mathfrak{Y}_{v'}$ via $\widetilde{\text{pr}}^*$. Using the invariance theorem (theorem \ref{Invariance}) we get $\pi_*\left(\widetilde{\text{pr}}^* \Da^{\dag}_{\mathfrak{X}_v,k}(\lambda)\right)= \Da^{\dag}_{\mathfrak{Y}_v,k}$. Then applying $\pi_*$ to the morphism $\Da^{\dag}_{\mathfrak{Y}_{v'},x'}(\lambda)\rightarrow \widetilde{\text{pr}}^* \Da^{\dag}_{\mathfrak{X}_v,k}(\lambda)$ gives the morphism 
\begin{eqnarray*}
\Psi : \pi_* \Da^{\dag}_{\mathfrak{Y}_{v'},k'}(\lambda)\rightarrow \Da^{\dag}_{\mathfrak{Y}_v,k}
\end{eqnarray*}
of the statement. As in \cite[Proposition 5.3.8]{HPSS}, making use of the maps $\Phi^{(m,k)}_{Y_v}$, as above, the assertion about the $G$-equivariance is reduced to the functorial properties of the rings $D^{(m)}(\G_v(k))$.
\end{proof}

\begin{defi}\label{coadmissible_G_arithmetic}
A coadmissible $G$-equivariant arithmetic $\Da(\lambda)$-module on $\Fb$ consists of a family $\Ma:= (\Ma_{\mathfrak{Y}_v,k})_{(\mathfrak{Y}_v,k)\in\Fb}$ of coherent $\Da^{\dag}_{\mathfrak{Y}_v,k}(\lambda)$-modules with the following properties:
\begin{itemize}
\item[(a)] For any special vertex $v\in \Bb$ and $g\in G$ with isomorphism $\rho_g^v : \mathfrak{Y}_v\rightarrow\mathfrak{Y}_{vg}$, there exists an isomorphism 
\begin{eqnarray*}
\phi^v_g : \Ma_{\mathfrak{Y}_{vg},k}\rightarrow \left(\rho^v_g\right)_* \Ma_{\mathfrak{Y},k}
\end{eqnarray*}
of sheaves of $L$-vector spaces, satisfying the following conditions:
\begin{itemize}
\item[(i)] For all $h,g\in G$ we have \footnote{Here we use the fact the action of $G$ on $\Bb$ is on the right and therefore $(\rho^{vh}_g)_*\;\circ\; (\rho^v_h)_*= (\rho_{hg}^v)_*$.}
\begin{eqnarray*}
(\rho^{vh}_g)_* \phi^v_h\;\circ\; \phi^{vh}_g = \phi^v_{hg}.
\end{eqnarray*}
\item[(ii)] For all open subsets $\Ub\subseteq \mathfrak{Y}_{vg}$, all $P\in\Da^{\dag}_{\mathfrak{Y}_{vg},k}(\lambda)(\Ub)$, and all $m\in \Ma_{\mathfrak{Y}_{vg},k}(\Ub)$ one has $\phi^{v}_{g,\Ub}(P.m)= T^v_{g,\Ub}(P).\phi^v_{g,\Ub}(m)$.
\item[(iii)] For all $g\in G_{k+1,v}$ the map $\phi^v_g:\Ma_{\mathfrak{Y},k}\rightarrow (\rho^v_g)_*\Ma_{\mathfrak{Y},k}=\Ma_{\mathfrak{Y},k}$ is equal to the multiplication by $\delta_g \in H^{0}(\mathfrak{Y}_v, \Da^{\dag}_{\mathfrak{Y}_v,k}(\lambda))$.
\end{itemize}
\item[(b)] For any two pairs $(\mathfrak{Y}_{v'},k')\succeq (\mathfrak{Y}_v,k)$ in $\Fb$ with morphism $\pi:\mathfrak{Y}_{v'}\rightarrow\mathfrak{Y}_v$ there exists a transition morphism $\psi_{\mathfrak{Y}_{v'},\mathfrak{Y}_{v}}:\pi_*\Ma_{\mathfrak{Y}_{v'}}\rightarrow\Ma_{\mathfrak{Y}_v}$, linear relative to the canonical morphism $\Psi:\pi_*\Da^{\dag}_{\mathfrak{Y}_{v'},k'}(\lambda)\rightarrow \Da^{\dag}_{\mathfrak{Y}_v,k}(\lambda)$ (in the preceding proposition) such that
\begin{eqnarray}\label{G_compatibility}
\phi^v_g\;\circ\; \psi_{\mathfrak{Y}_{v'g},\mathfrak{Y}_{vg}} = (\rho_g^v)_*\psi_{\mathfrak{Y}_{v'},\mathfrak{Y}_v}\;\circ\; (\pi .g)_* \phi^{v'}_g
\end{eqnarray}
for any $g\in G$ (where we have use the relation $(\rho_g^v)_*\;\circ\; \pi_* = (\pi . g)_*\;\circ\; (\rho^{v'}_g)_*$). If $v'=v$, and $(\mathfrak{Y}_v',k')\succeq (\mathfrak{Y}_v,k)$ in $\underline{\Fb}_{v}$, and if $\mathfrak{Y}_v'$, $\mathfrak{Y}_v$ are $G_{v,0}$-equivariant, then we require additionally that the morphism induced by $\psi_{\mathfrak{Y}_v',\mathfrak{Y}_v}$ (cf. (\ref{morph_for_equi}))
\begin{eqnarray}
\overline{\psi}_{\mathfrak{Y}'_v,\mathfrak{Y}_v}: \Da^{\dag}_{\mathfrak{Y}_v,k}(\lambda)\otimes_{\pi_* \Da^{\dag}_{\mathfrak{Y}'_v,k'}(\lambda), G_{v,k+1}} \pi_* \Ma_{\mathfrak{Y}'_v,k'} \rightarrow \Ma_{\mathfrak{Y}_v,k}
\end{eqnarray}
is an isomorphism of $\Da^{\dag}_{\mathfrak{Y}_v,k}(\lambda)$-modules. As in theorem \ref{second_equivalence}, the morphisms $\psi_{\mathfrak{Y}_{v'},\mathfrak{Y}_{v}}:\pi_*\Ma_{\mathfrak{Y}_{v'},k'}\rightarrow\Ma_{\mathfrak{Y}_v,k}$ are required to satisfy the transitive condition 
\begin{eqnarray*}
\psi_{\mathfrak{Y}_{v'},\mathfrak{Y}_v}\;\circ\; \pi_* (\psi_{\mathfrak{Y}_{v''},\mathfrak{Y}_{v'}}) = \psi_{\mathfrak{Y}_{v''},\mathfrak{Y}_v},
\end{eqnarray*}
\justify
whenever $(\mathfrak{Y}_{v''},k'')\succeq (\mathfrak{Y}_{v'},k')\succeq (\mathfrak{Y}_{v},k)$ in $\underline{\Fb}$. Moreover, $\psi_{\mathfrak{Y}_v,\mathfrak{Y}_v}=id_{\Ma_{\mathfrak{Y}_v,k}}$.
\end{itemize}
A morphism $\Ma \rightarrow \Na$ between two coadmissible $G$-equivariant arithmetic $\Da(\lambda)$-modules consists in a family of morphisms $\Ma_{\mathfrak{Y},k}\rightarrow
\Na_{\mathfrak{Y},k}$ of $\Da^{\dag}_{\mathfrak{Y},k}(\lambda)$-modules, which respect the extra conditions imposed in $(a)$ and $(b)$. We denote the resulting category by $\Ca^{\Fb}_{G, \lambda}$.
\end{defi}
\justify
We recall for the reader that $D(G_0,L)$ is a Fréchet-Stein algebra \cite[Theorem 24.1]{ST}. Moreover, a $D(G,L)$-module is called coadmissible if it is coadmissible as a $D(H,L)$-module for every compact open subgroup $H\subseteq G$ (cf. \cite[Definition subsection 6]{ST0}). Given that for any two compact open subgroups $H\subseteq H'\subseteq G$ the algebra $D(H',L)$ is finitely generated free and hence coadmissible as a $D(H,L)$-module, it follows from \cite[Lemma 3.8]{ST0} that the preceding condition needs to be tested only for a single compact open subgroup $H\subseteq G$. This motivates the following definition where we will consider the weak Freéchet-Stein structure of $D(G_0,L)$ defined in (\ref{Frechet_structure}). 

\begin{defi}
We say that $M$ is a coadmissble $D(G,L)$-module if $M$ is coadmissisble as a $D(G_0,L)$-module.
\end{defi}
\justify
Let us construct now the bridge to the category of coadmisible $D(G,L)_{\lambda}$-modules. Let $M$ be such a coadmissible $D(G,L)_{\lambda}$-module and let  $V := M'_b$. We fix $v\in \Bb$ a special vertex. Let $V_{\G_v(k)^{\circ}-\text{an}}$ \footnote{Here we use the fact that $(\G_v)_L=\G_L$.} be the subspace of $\G_{v}(k)^{\circ}$-analytic vectors and let $M_{v,k}$ be its continuous dual. For any $(\mathfrak{Y}_v,k)\in\underline{\Fb}$ we have a coherent $\Da^{\dag}_{\mathfrak{Y}_v,k}(\lambda)$-module
\begin{eqnarray*}
\Lb oc^{\dag}_{\mathfrak{Y}_v,k}(\lambda)(M_{v,k}) = \Da^{\dag}_{\mathfrak{Y}_v,k}(\lambda)\otimes_{\Db^{\text{an}}(\G_{v}(k)^{\circ})_{\lambda}}M_{v,k}
\end{eqnarray*} 
and we can consider the family
\begin{eqnarray*}
\Lb oc^{G}_{\lambda}(M) := \left(\Lb oc^{\dag}_{\mathfrak{Y}_v,k}(\lambda)(M_{v,k})\right)_{(\mathfrak{Y}_v,k)\in\underline{\Fb}}.
\end{eqnarray*}
\justify
On the other hand, given an object $\Ma \in \Ca^{\Fb}_{G,\lambda}$, we may consider the projective limit 
\begin{eqnarray*}
\Gamma (\Ma):= \varprojlim_{(\mathfrak{Y},k)\in\underline{\Fb}} H^0 (\mathfrak{Y}, \Ma_{\mathfrak{Y},k})
\end{eqnarray*}
with respect to the transition maps $\psi_{\mathfrak{Y}',\mathfrak{Y}}$. Here the projective limit is taken in the sens of abelian groups and over the cofinal family of pairs $(\mathfrak{Y}_v ,k)\in\underline{\Fb}$ with $G_{v,0}$-equivariant $\mathfrak{Y}_v$ , cf. remark \ref{G0_eq_system}.

\begin{theo}\label{third_equivalence} Let us suppose that $\lambda\in\text{Hom}(\T,\G_m)$ is an algebraic character such that $\lambda+\rho\in\mathfrak{t}^*_L$ is a dominant and regular character of $\mathfrak{t}_L$. (and therefore, a dominant and regular character on every special vertex of $\Bb$). The functors $\Lb oc^{G}_{\lambda} (\bullet)$ and $\Gamma (\bullet)$ induce quasi-inverse equivalences between the category of coadmissible $D(G,L)_{\lambda}$-modules and $\Ca^{\Fb}_{G,\lambda}$.
\end{theo}
\justify
The proof follows the same lines of reasoning given in \cite[Theorem 5.3.12]{HPSS}.
\begin{proof}
The proof is an extension of the the proof of theorem \ref{second_equivalence}, taking into account the additional $G$-action. Let $M$ be a coadmissible $D(G,L)_{\lambda}$-module and let $\Ma \in \Ca^{G}_{\Fb,\lambda}$. The proof of the theorem follows the following steps.
\justify
\textit{Claim 1.} \textit{One has $\Lb oc^{G}_{\lambda}(M)\in\Ca^{\Fb}_{G,\lambda}$ and $\Lb oc^{G}_{\lambda}(\bullet)$ is functorial.}
\begin{proof}
Let $g\in G$, $v\in\Bb$ a special vertex and $\rho^v_g:\mathfrak{Y}_v\rightarrow\mathfrak{Y}_{vg}$ the respective isomorphism. For conditions $(a)$ for $\Lb oc^{G}_{\lambda}(M)$ we need the maps
\begin{eqnarray*}
\phi_g : \Lb oc^{G}_{\lambda}(M)_{\mathfrak{Y}_v,k} := \Lb oc^{\dag}_{\mathfrak{Y}_v,k}(\lambda)(M_{v,k}) \rightarrow (\rho^v_g)_*\Lb  oc^{G}_{\lambda}(M)_{\mathfrak{Y}_v,k}
\end{eqnarray*}
satisfying the properties $(i)$, $(ii)$ and $(iii)$. Let $\tilde{\phi}^v_g : M_{vg,k}\rightarrow M_{v,k}$ denote the dual map to \footnote{Here we use $\G_{vg}(k)^{\circ}=g^{-1}\G_{v}(k)^{\circ}g$ in $\G^{\text{rig}}$.} $V_{\G_v(k)^{\circ}-\text{an}}  \rightarrow  V_{\G_{vg}(k)^{\circ}-\text{an}}$; $w \mapsto  g^{-1}w$. Let $\Ub\subseteq \mathfrak{Y}_{vg}$ be an open subset and $P\in\Da^{\dag}_{\mathfrak{Y}_{vg},k}(\lambda)(\Ub)$, $m\in M_{vg,k}$. We define
\begin{eqnarray} \label{G-action}
\phi^v_{g,\;\Ub}(P\otimes m) := T^v_{g,\;\Ub}(P)\otimes \tilde{\phi}^v_g (m).
\end{eqnarray}
Exactly as we have done in theorem \ref{second_equivalence}, the family $(\phi^v_g)$ satisfies the requirements $(i)$, $(ii)$ and $(iii)$. Let us verify now condition $(b)$. Given $(\mathfrak{Y}_{v'},k')\succeq (\mathfrak{Y}_v,k)$ in $\underline{\Fb}$, we have $\G_{v'}(k')^{\circ}\subseteq \G_{v}(k)^{\circ}$ in $\G^{\text{rig}}$ and we denote by $\tilde{\psi}_{\mathfrak{Y}_{v'},\mathfrak{Y}_v}: M_{v',k'}\rightarrow M_{v,k}$ the map dual to the natural inclusion $V_{\G_v(k)^{\circ}-\text{an}}\subseteq V_{\G_{v'}(k')^{\circ}-\text{an}}$. Let $\Ub\subseteq \mathfrak{Y}_{v'}$ be an open subset and $P\in \pi_{*}\Da^{\dag}_{\mathfrak{Y}_{v'},k'}(\lambda)(\Ub)$, $m\in M_{v',k'}$. We then define \footnote{We avoid the subscript $\Ub$ in order to soft the notation.}
\begin{eqnarray*}
\psi_{\mathfrak{Y}_{v'},\mathfrak{Y}_v}(P\otimes m) := \Psi_{\mathfrak{Y}_{v'},\mathfrak{Y}_v}(P)\otimes \tilde{\psi}_{\mathfrak{Y}_{v'},\mathfrak{Y}_{v}}(m)
\end{eqnarray*}
where $\Psi_{\mathfrak{Y}_{v'},\mathfrak{Y}_v}: \pi_{*}\Da^{\dag}_{\mathfrak{Y}_{v'},k'}(\lambda)\rightarrow \Da^{\dag}_{\mathfrak{Y}_v,k}(\lambda)$ is the canonical morphism given by the preceding proposition. This definition extends to a map
\begin{eqnarray*}
\psi_{\mathfrak{Y}_{v'},\mathfrak{Y}_{v}}: \pi_* \Lb oc^{G}_{\lambda}(M)_{\mathfrak{Y}_{v'},k'} \rightarrow \Lb oc^{G}_{\lambda}(M)_{\mathfrak{Y}_v,k}
\end{eqnarray*}
which satisfies all the required conditions. The functoriality of $\Lb oc^{G}_{\lambda}(\bullet)$ can be verified exactly as we have done for the functor $\Lb oc^{G_0}_{\lambda}(\bullet)$.
\end{proof}
\justify
\textit{Claim 2.} $\Gamma (\Ma)$ is a coadmissible $D(G,L)_{\lambda}$-module.
\begin{proof}
We already know that $\Gamma(\Ma)$ is a coadmissible $D(G_{v,0},L)_{\lambda}$-module for any $v$ (theorem \ref{second_equivalence}). So it suffices to exhibit a compatible $G$-action on $\Gamma (\Ma)$. Let $g\in G$. The isomorphisms $\phi^v_g : \Ma_{\mathfrak{Y}_{vg},k}\rightarrow (\rho^v_g)_* \Ma_{\mathfrak{Y},k}$, which are compatibles with transitions maps, induce isomorphisms at the level of global sections (which we denote again by $\phi^v_g$ to soft the notation)
\begin{eqnarray*}
\phi^v_g : H^0(\mathfrak{Y}_{{vg},k}, \Ma_{\mathfrak{Y}_{vg},k})\rightarrow H^0( \mathfrak{Y}_{v} ,\Ma_{\mathfrak{Y},k}).
\end{eqnarray*}
Let us identify
\begin{align*}
\Gamma(\Ma) &= \varprojlim_{(\mathfrak{Y}_{vg},k)\in\underline{\Fb}_{vg}}H^0(\mathfrak{Y}_{{vg},k}, \Ma_{\mathfrak{Y}_{vg},k})\\
& = \left\{ \left( m_{\mathfrak{Y}_{vg},k} \right)_{(\mathfrak{Y}_{vg},k)\in\underline{\Fb}_{vg}} \in \displaystyle\prod_{(\mathfrak{Y}_{vg},k)\in\underline{\Fb}_{vg}}  H^0(\mathfrak{Y}_{{vg},k}, \Ma_{\mathfrak{Y}_{vg},k})\; |\; \psi_{\mathfrak{Y}'_{vg},\mathfrak{Y}_{vg}}(m_{\mathfrak{Y}'_{vg},k}) = m_{\mathfrak{Y}_{vg},k} \right\}
\end{align*}
Where we have abused of the notation and we have denoted by $\psi_{\mathfrak{Y}'_{vg},\mathfrak{Y}_{vg}}$ the morphism obtained by taking global sections on the morphism  $\psi_{\mathfrak{Y}'_{vg},\mathfrak{Y}_{vg}}: (\pi .g)_*\Da^{\dag}_{\mathfrak{Y}'_{vg},k'}(\lambda)\rightarrow \Da^{\dag}_{\mathfrak{Y}_{vg},k}(\lambda)$. For $g\in G$ and $m:=(m_{\mathfrak{Y}_{vg},k} )_{(\mathfrak{Y}_{vg},k)\in\underline{\Fb}_{vg}}\in \Gamma (\Ma)$ we define 
\begin{eqnarray}
g.m := \left(\phi^v_g (m_{\mathfrak{Y}_{vg},k})\right)_{(\mathfrak{Y}_{vg},k)\in \underline{\Fb}_{vg}}\in\displaystyle\prod_{(\mathfrak{Y}_v,k)\in\underline{\Fb}_v}H^0(\mathfrak{Y}_v,\Ma_{\mathfrak{Y}_v,k}),
\hspace{1 cm}
g.m_{\;(\mathfrak{Y}_v,k)\in\underline{\Fb}_v} :=  \phi^v_g (m_{\mathfrak{Y}_{vg},k})
\end{eqnarray}
\justify
We want to see that $g.m \in \Gamma (\Ma)=\varprojlim_{(\mathfrak{Y}_v,k)\in\underline{\Fb}_v}H^0(\mathfrak{Y}_v,\Ma_{\mathfrak{Y}_v,k})$ and that this assignment defines a left $G$-action on $\Gamma(\Ma)$.  Taking global sections on (\ref{G_compatibility}) we get the relation $\phi_g^v\;\circ\;\psi_{\mathfrak{Y}'_{vg},\mathfrak{Y}_{vg}} = \psi_{\mathfrak{Y}'_v,\mathfrak{Y}_v}\;\circ\; \phi^v_g$, which implies that
\begin{eqnarray*}
\psi_{\mathfrak{Y}'_{v},\mathfrak{Y}_{v}} (g.m_{\;\mathfrak{Y}'_v,k'}) =
 \psi_{\mathfrak{Y}'_{v},\mathfrak{Y}_{v}} (\phi^v_g (m_{\mathfrak{Y}'_{vg},k'})) 
  = \phi^v_g (\psi_{\mathfrak{Y}'_{vg},\mathfrak{Y}_{vg}}(m_{\mathfrak{Y}'_{vg},k'})) 
  = \phi^v_g (m_{\mathfrak{Y}_{vg},k}) 
  = g.m_{\;\mathfrak{Y}_v,k}.
\end{eqnarray*}
We obtain an isomorphism
\begin{eqnarray*}
\Gamma (\Ma) = \varprojlim_{\underline{\Fb}_{vg}}H^{0}(\mathfrak{Y}_{vg},\Ma_{\mathfrak{Y}_{vg},k})\xrightarrow{g} \varprojlim_{\underline{\Fb}_{v}} H^{0}(\mathfrak{Y}_{v},\Ma_{\mathfrak{Y}_{v},k}) = \Gamma(\Ma).
\end{eqnarray*}
According to $(i)$ in $(a)$ we have the sequence 
\begin{eqnarray*}
\phi^{v}_{hg} : H^0(\mathfrak{Y}_{vhg},\Ma_{\mathfrak{Y}_{vhg},k}) \xrightarrow{\phi^{vh}_g} H^0(\mathfrak{Y}_{vh},\Ma_{\mathfrak{Y}_{vh},k}) \xrightarrow{\phi^v_h} H^0(\mathfrak{Y}_{v},\Ma_{\mathfrak{Y}_{v},k})
\end{eqnarray*}
which tells us that $h.(g.m)=(hg).m$, for $h,g\in G$ and $m\in\Gamma(\Ma)$. This gives a $G$-action on $\Gamma(\Ma)$ which, by construction, is compatible with its various $D(G_{v,0},L)$-module structures.
\end{proof}
\justify
\textit{Claim 3.} $\Gamma\;\circ\; \Lb oc^{G}_{\lambda}(M) \simeq M$.
\begin{proof}
By theorem \ref{second_equivalence} we know that this holds as a coadmissible $D(G_0,L)_{\lambda}$-module, so we need to identify the $G$-action on both sides. Let $v$ be a special vertex. According to (\ref{G-action}), the action
\begin{eqnarray*}
\Gamma\; \circ \; \Lb oc^{G}_{\lambda}(M) \simeq \varprojlim_{k} M_{vg,k} \rightarrow \varprojlim_v M_{v,k} \simeq \Gamma \;\circ\; \Lb oc^{G}_{\lambda}(M)
\end{eqnarray*}
of an element $g\in G$ on $\Gamma\; \circ \; \Lb oc^{G}_{\lambda}(M)$ is induced by $\tilde{\phi}^v_g: M_{vg,k}\rightarrow M_{v,k}$. By dualizing 
\begin{eqnarray*}
V = \displaystyle\bigcup_{k\in\N} V_{\G_{vg}(k)^{\circ}-\text{an}} =  \displaystyle\bigcup_{k\in\N} V_{\G_{v}(k)^{\circ}-\text{an}}
\end{eqnarray*}
we obtain the identification 
\begin{eqnarray*}
M \simeq \varprojlim_{k} M_{vg,k} \simeq \varprojlim_{k} M_{v,k},
\end{eqnarray*}
and therefore we get back the original action of $g$ on $M$.
\end{proof}
\justify
\textit{Claim 4.} $\Lb oc^{G}_{\lambda}\;\circ\;\Gamma (\Ma) \simeq \Ma$.
\begin{proof}
We know that $\Lb oc^{G}_{\lambda}(\Gamma (\Ma))_{\mathfrak{Y}_v,k} = \Ma_{\mathfrak{Y}_v,k}$ as $\Da^{\dag}_{\mathfrak{Y}_v,k}(\lambda)$-modules for any $(\mathfrak{Y}_v,k)\in\underline{\Fb}$, cf. theorem \ref{second_equivalence}. It remains to verify that these isomorphisms are compatible with the maps $\phi^v_g$ and $\psi_{\mathfrak{Y}_{v'},\mathfrak{Y}_v}$ on both sides. To do that, let us see that the maps $\phi^v_g$ on the left-hand side are induced by the maps of the right-hand side. Given 
\begin{eqnarray*}
\phi^v_g : \Ma_{\mathfrak{Y}_v,k}\rightarrow (\rho^v_g)_* \Ma_{\mathfrak{Y}_v,k},
\end{eqnarray*}
the corresponding map
\begin{eqnarray*}
\phi^v_g : \Lb oc^{G}_{\lambda}(\Gamma(\Ma))_{\mathfrak{Y}_{vg},k} \rightarrow (\rho^v_g)_* (\Lb oc^{G}_{\lambda}(\Gamma (\Ma))_{\mathfrak{Y}_v,k})
\end{eqnarray*}
equals the map 
\begin{eqnarray*}
\Da^{\dag}_{\mathfrak{Y}_{vg},k}(\lambda)\otimes_{\Db^{\text{an}}(\G_{vg}(k)^{\circ})_{\lambda}}H^{0}\left(\mathfrak{Y}_{vg},\Ma_{\mathfrak{Y}_{vg},k}\right) \rightarrow (\rho^v_g)_*\left(\Da^{\dag}_{\mathfrak{Y}_v,k}(\lambda)\otimes_{\Db^{\text{an}}(\G_v(k)^{\circ})_{\lambda}}H^0 \left(\mathfrak{Y}_v,\Ma_{\mathfrak{Y}_{v,k}}\right)\right)
\end{eqnarray*}
given locally by $T^v_{g,\mathfrak{Y}_{gv}}\otimes H^0(\mathfrak{Y}_{vg}, \phi^v_g)$, cf. (\ref{G-action}). Let $\Ub\subseteq \mathfrak{Y}_v$ be an open subset and $P\in\Da^{\dag}_{\mathfrak{Y}_v,k}(\lambda)(\Ub)$, $m\in M_{v,k}= H^0(\mathfrak{Y}_{vg},\Ma_{\mathfrak{Y}_{vg},k})$. The isomorphism $\Lb oc^{G}_{\lambda}(\Gamma(\Ma))_{\mathfrak{Y}_{v},k}\simeq \Ma_{\mathfrak{Y}_v,k}$ are induced (locally) by $P\otimes m \mapsto P . (m|_{\Ub})$. Condition $(ii)$ tells us that these morphisms interchange the maps $\phi^v_g$, as desired. The compatibility with transitions maps $\psi_{\mathfrak{Y}_{v'},\mathfrak{Y}_v}$ for two models $(\mathfrak{Y}_{v'},k')\succeq (\mathfrak{Y},k)$ in $\underline{\Fb}$ is deduced in a entirely similar manner as we have done in theorem \ref{second_equivalence} and the fact that $\psi_{\mathfrak{Y}_{v'},\mathfrak{Y}_v}$ is linear relative to the canonical morphism $\Psi : \pi_* \Da^{\dag}_{\mathfrak{Y}_{v'},k'}(\lambda)\rightarrow \Da^{\dag}_{\mathfrak{Y}_v,k}$.
\end{proof}
\justify
This ends the proof of the theorem.
\end{proof}
\justify
As in the case of the group $G_0$, we now indicate how objects from $\Ca^{\Fb}_{G , \lambda}$ can be realized as $G$-equivariant sheaves on the $G$-space $\mathfrak{X}_{\infty}$. The following discussion is an adaptation of the discussion given in \cite[5.4.3 and proposition 5.4.5]{PSS2} to our case.

\begin{prop}
The $G_0$-equivariant structure of the sheaf $\Da(\lambda)$ extends to a $G$-equivariant structure.
\end{prop}
\begin{proof}
Let $g\in G$ and let $v,v' \in \Bb$ be special vertexes. Let us suppose that $(\mathfrak{Y}_{v'},k')\succeq (\mathfrak{Y}_v,k)$ in $\underline{\Fb}$. The isomorphism $\rho^{v'}_g : \mathfrak{Y}_{v'}\rightarrow \mathfrak{Y}_{v'g}$ induces a ring isomorphism 
\begin{eqnarray*}
T_g^{v'} : \Da^{\dag}_{\mathfrak{Y}_{v'g},k'}(\lambda)\rightarrow \left(\rho_g^{v'}\right)_* \Da^{\dag}_{\mathfrak{Y}_{v'},k'}(\lambda).
\end{eqnarray*}
On the other hand, and exactly as we have done in (\ref{G_0-action_on_ZR}), the commutative diagram
\begin{eqnarray*}
\begin{tikzcd} [column sep=huge]
\mathfrak{X}_{\infty} \arrow[r, "sp_{\;\mathfrak{Y}_{v}}"] \arrow[d, "sp_{\;\mathfrak{Y}_{v'}}"]
& \mathfrak{Y}_{v} \arrow[r, "\rho_g^{v}"]
& \mathfrak{Y}_{vg} \\
\mathfrak{Y}_{v'} \arrow[r, "\rho_{g}^{v'}"] \arrow[ur, "\pi"]
& \mathfrak{Y}_{v'g}. \arrow[ur, "\pi .g "]
\end{tikzcd}
\end{eqnarray*} 
defines a continuous function 
\begin{eqnarray*}
\begin{matrix}
\rho_g : & \mathfrak{X}_{\infty}  & \rightarrow & \mathfrak{X}_{\infty}\\
           & (a_{v}) & \mapsto & (\rho_g^{v} (a_{v})),
\end{matrix}
\end{eqnarray*}
which satisfies $\text{pr}_{\mathfrak{Y}_{v'g}}\;\circ\;\rho_g = \rho_g^{v'}\;\circ\;\text{pr}_{\mathfrak{Y}_{v'}}$. In particular, if $V\subseteq \mathfrak{X}_{\infty}$ is the open subset $V:=\text{pr}_{\mathfrak{Y}_{v}}^{-1}(\Ub)$ with $\Ub\subseteq \mathfrak{Y}_v$ an open subset. Then 
\begin{eqnarray*}
\left(\rho_g^{v'}\right)^{-1}\left( \text{pr}_{\mathfrak{Y}_{v'g}}(V)\right) = \text{pr}_{\mathfrak{Y}_{v'}}\left(\rho_g^{-1}(V)\right)
\end{eqnarray*}
and so the map $T_g^{v'}$ induces the morphism 
\begin{eqnarray}\label{G_projective_limit}
\Da^{\dag}_{\mathfrak{Y}_{v'g},k'}(\lambda)\left(\text{pr}_{\mathfrak{Y}_{v'g}}(V)\right)\rightarrow \Da^{\dag}_{\mathfrak{Y}_{v'},k'}(\lambda)\left(\text{pr}_{\mathfrak{Y}_{v'}}\left( \rho_g^{-1}(V) \right)\right).
\end{eqnarray}
Moreover, if $(\mathfrak{Y}_{v''},k'')\succeq (\mathfrak{Y}_{v'},k')\succeq (\mathfrak{Y}_v ,k)$ in $\underline{\Fb}$, and as before $V:= \text{pr}_{\mathfrak{Y}_{v}}^{-1}(\Ub)\subseteq \mathfrak{X}_{\infty}$ with $\Ub\subseteq\mathfrak{Y}_v$ an open subset, then the commutative diagram
\begin{eqnarray*}
\begin{tikzcd}
\Da^{\dag}_{\mathfrak{Y}_{v''g},k''}(\lambda)\left(\text{pr}_{\mathfrak{Y}_{v''g}}(V)\right) \arrow[r] \arrow[d]
& \Da^{\dag}_{\mathfrak{Y}_{v''},k''}(\lambda)\left(\text{pr}_{\mathfrak{Y}_{v''}}\left( \rho_g^{-1}(V) \right)\right).
 \arrow[d]\\
\Da^{\dag}_{\mathfrak{Y}_{v'g},k'}(\lambda)\left(\text{pr}_{\mathfrak{Y}_{v'g}}(V)\right) \arrow[r]
& \Da^{\dag}_{\mathfrak{Y}_{v'},k'}(\lambda)\left(\text{pr}_{\mathfrak{Y}_{v'}}\left( \rho_g^{-1}(V) \right)\right).
\end{tikzcd}
\end{eqnarray*}
implies that if, by cofinality, we identify $\Da(\lambda)(V) = \varprojlim_{(\mathfrak{Y}_{vg},k)\in\underline{\Fb}_{vg}} \Da^{\dag}_{\mathfrak{Y}_{vg},k}(\lambda)\left(\text{pr}_{\mathfrak{Y}_{vg}}(V)\right)$ and we take projective limits in (\ref{G_projective_limit}), then we get a ring homomorphism
\begin{eqnarray*}
T_{g,V} : \Da(\lambda)(V) \rightarrow (\rho_g)_*\Da (\lambda) (V)
\end{eqnarray*}
which implies that the sheaf $\Da(\lambda)$ is $G$-equivariant. Furthermore, from construction this $G$-quivariant structure extends the $G_0$-structure defined in the subsection \ref{G_0-ZR}.
\end{proof}

\justify
Finally, let us recall the faithful functor 
\begin{eqnarray*}
\Ma\rightsquigarrow \Ma_{\infty}
\end{eqnarray*}
from coadmissible $G_0$-equivariant arithmetic $\Da(\lambda)$-modules on $\Fb_{\mathfrak{X}}$ to $G_0$-equivariant $\Da(\lambda)$-modules on $\mathfrak{X}_{\infty}$. If $\Ma$ comes from a coadmissible $G$-equivariant $\Da(\lambda)$-module on $\Fb$, then $\Ma_{\infty}$ is in fact $G$-equivariant (as in (\ref{G_0-equi}), this can be proved by using the family of $L$-linear isomorphisms $(\phi_g^v)_{g\in G}$. As in proposition \ref{G_0-equivalence}, the preceding theorem gives us

\begin{theo}\label{fully_faithful}
Let us suppose that $\lambda\in\text{Hom}(\T,\G_m)$ is an algebraic character such that $\lambda+\rho\in\mathfrak{t}^*_L$ is a dominant and regular character of $\mathfrak{t}_L$. The functor $\Ma\rightsquigarrow\Ma_{\infty}$ from the category $\Ca^{\Fb}_{G,\lambda}$ to $G$-equivariant $\Da(\lambda)$-modules on $\mathfrak{X}_{\infty}$ is a faithful functor.
\end{theo}
\begin{Final_remark} \label{Final_remar} \textbf{Final remark.} We end this work by remarking to the reader that the functors in proposition \ref{G_0-equivalence} and theorem \ref{fully_faithful} become \textit{fully faithful} functors if we required that the objects in the target category carry a structure of locally convex topological $\Da(\lambda)$-modules (cf. \cite[Propositions 5.2.31 and 5.3.16]{HPSS}). In fact, following \cite[5.2.30]{HPSS} we can see that $\Da(\lambda)$ carries a natural structure of a sheaf of locally convex topological $L$-algebras and more generally, if $\Ma \in \Cb^{G_0}_{\mathfrak{X},\lambda}$ (resp. $\Ma\in\Cb^{G}_{\mathfrak{X},\lambda}$) then $\Ma_{\infty}$ becomes a $G_0$-equivariant (resp. a $G$-equivariant) sheaf of locally convex topological $L$-vector spaces, endowed with the structure of a topological $\Da(\lambda)$-module.
\end{Final_remark}

% REFERENCES


\newpage
\begin{thebibliography}{X}

\bibitem{BB} \textsc{a. beilinson \& j. bernstein}, {Localisation de $\mathfrak{g}$-modules}. Comptes Rendus, 292 (1981), 15–18. MathSciNet MATH.

\bibitem{Berthelot0} \textsc{p. berthelot}, { Cohomologie rigide et cohomologie rigide à supports propres}. Preprint of the university of Rennes, 1996.

\bibitem{Berthelot1} \textsc{p. berthelot}, {$\Db$-modules arithmétiques I. Opérateurs différentiels de niveau fini}. Ann. Sci. École Norm. Sup.(4), 29(2), 185-272.

\bibitem{Berthelot3}\textsc{p. berthelot}, {Introduction a la theorie arithmetique des $\Db$-modules, Cohomologies $p$-adiques et applications arithmetiques. II}.  Astérisque, 2002, vol. 279, p. 1-80.

\bibitem{Berthelot2}\textsc{p. berthelot \& a. ogus}, {Notes on crystalline cohomology. (MN-21)}. Princeton University Press, 2015.

\bibitem{Bosch} \textsc{s. bosch}, {Lectures on formal and rigid geometry}. Berlin/Heidelberg/New York : Springer, 2014.

\bibitem{BLR} \textsc{s. bosch, w. lütkebohmert \& m. raynaud}, {Néron models}. Springer Science \& Business Media, 2012.

\bibitem{Brion1} \textsc{m. brion}, {Linearization of algebraic groups actions}. 2018.

\bibitem{BT1} \textsc{f. bruhat \& j. tits}, {Groupes réductifs sur un corps local I}. Hautes Etudes Sci. Publ. Math.,
(41): 5–252, 1972.

\bibitem{BT2} \textsc{f. bruhat \& j.tits}, {Groupes réductifs sur un corps local II}.  Hautes Etudes Sci. Publ. Math.,
(60): 5–184, 1984.

\bibitem{Dixmier} \textsc{j. dixmier}, {Enveloping algebras}. Vol. 14. Newnes, 1977.

\bibitem{Dolgachev} \textsc{i. dolgachev}, {Lectures on invariant theory}. Cambridge University Press, 2003.

\bibitem{Emerton1} \textsc{m. emerton}, {Locally analytic representation theory of $p$-adic reductive groups: A summary of some recent developments}. London mathematical society lecture note series, 2007, vol. 320, p. 407.

\bibitem{Emerton2} \textsc{m. emerton}, {Locally Analytic Vectors in Representations of Locally $p$-adic Analytic Groups}. American Mathematical Soc., 2017.

\bibitem{Grothendieck1} \textsc{a. grothendieck}, {Éléments de géométrie algébrique: I. Le langage des schémas}.  Publications Mathématiques de l'IHÉS, 1960, vol. 4, p. 5-228.

\bibitem{Grothendieck2} \textsc{a. grothendieck}, {Éléments de Géométrie algébrique: II. Étude globale élémentaire de quelques classes de morphismes}. Publications Mathématiques de l'IHES, 1961, vol. 8, p. 5-222.

\bibitem{Hartshorne1} \textsc{r. hartshorne}, {Algebraic geometry}. Vol. 52. Springer Science \& Business Media, 2013.

\bibitem{Hartshorne2} \textsc{r. hartshorne},  \textit{ Residues and duality}.  Lectures notes of a seminar on the work of A. Grothendieck, given at Harvard 1963/64. Springer.

\bibitem{HTT} \textsc{r. hotta, k. takeuchi \& t. tanasaki}, { $\Db$-modules, perverse sheaves and representation theory}. Vol. 236. Springer Science \& Business Media, 2007.

\bibitem{Huyghe1} \textsc{c. huyghe}, {$\Db^\dag$-affinité de l'espace projectif}. Compositio Mathematica, 1997, vol. 108, no 3, p. 277-318.

\bibitem{Huyghe2} \textsc{c. huyghe},
{Un théorème de Beilinson-Bernstein pour les $\Db$-modules arithmétiques}. Bull. Soc. math. France
137 (2), 2009, p. 159–183.

\bibitem{HS1} \textsc{c. huyghe \& t. schmidt}, {$\Db$-modules arithmétiques, distributions et localisation}. Rendiconti del Seminario Matematico della Università di Padova 139, 1--76 (2018).

\bibitem{HS2} \textsc{c. huyghe \& t. schmidt}, {$\Db$-modules arithmétiques sur la variété de drapeaux}. To appear in Journal für die Reine und Angewandte Mathematik. doi 10.1515/ crelle-2017-0021.

\bibitem{HSS} \textsc{c. huyghe, t. schmidt \& m. strauch}, {Arithmetic structures for differential operators on formal schemes}. arXiv:1709.00555v3.

\bibitem{HPSS} \textsc{c. huyghe, d. patel, t. schmidt \& m. strauch}, {$\Da^{\dag}$-affinity of formal models of flag varieties}. Preprint 2017.  arXiv:1501.05837v2 (to appear in Mathematical Research Letters).

\bibitem{Jantzen} \textsc{j.c. jantzen}, {Representation of algebraic groups}. American mathematical soc., 2007.

\bibitem{Liu} \textsc{q. liu}, {Algebraic geometry and arithmetic curves}. Volume 6 of Oxford  Graduated Texts in Mathematics. Oxford University Press, Oxford, 2002. Translated from the French by Reiné Erné, Oxford Sciences Publications.

\bibitem{MRS} \textsc{j.c mcconnell, j.c robson \& l.w small}, {Noncommutative notherian rings}. American Mathematical Soc., 2001.

\bibitem{Milicic1} \textsc{d. mili\v{c}i\'{c}}, {Localization and representation theory of reductive Lie groups}.  Preprint (April 1993), http://www. math. utah. edu/milicic/Eprints/book. pdf (1993).

\bibitem{PSS1}\textsc{d. patel, t. schmidt \& m. strauch}, {Locally analytic representations and sheaves on the Bruhat–Tits building}. Algebra \& Number Theory 8.6 (2014): 1365-1445.

\bibitem{PSS2}\textsc{d. patel, t. schmidt \& m. strauch}, {Locally analytic representations of GL(2,$L$) via semistable models of $\mathbb{P}^1$}. Journal of the Institute of Mathematics of Jussieu, 2017, p. 1-63.

\bibitem{Sarrazola} \textsc{a. sarrazola alzate}, {A Beilinson-Bernstein theorem for twisted arithmetic differential operators on the formal flag variety}. arXiv preprint arXiv:1811.11567v4.

\bibitem{ST0} \textsc{p. schneider \& j. teitelbaum}, {Algebras of $p$-adic distributions and admissible representations}.  Inventiones mathematicae, 2003, vol. 153, no 1, p. 145-196.

\bibitem{ST} \textsc{p. schneider  \& j. teitelbaum}, {Continuous and locally analytic representation theory}. Lectures at Hangzhou 2004.

\bibitem{ST1} \textsc{p. schneider \& j. teitelbaum}, {Locally analytic distributions and $p$-adic representation theory, with applications to GL$_2$}. Journal of the American Mathematical Society, 2002, vol. 15, no 2, p. 443-468.

\bibitem{WW} \textsc{w.c waterhouse \& b. weisfeiler}, {One-dimensional affine group schemes}. Journal of Algebra, 66(2), 550-568.

\end{thebibliography}
\end{document}